\documentclass[draft]{amsart}
\usepackage[pdfborder={0 0 0},draft=false,pagebackref]{hyperref}

\makeatletter\def\HyPsd@CatcodeWarning#1{}\makeatother

\usepackage{esint,amssymb,tikz,mathtools}

\newtheorem{thm}[equation]{Theorem}
\newtheorem{lem}[equation]{Lemma}
\newtheorem{cor}[equation]{Corollary}

\theoremstyle{definition}

\theoremstyle{remark}
\newtheorem{rmk}[equation]{Remark}

\numberwithin{equation}{section}
\numberwithin{figure}{section}

\newcommand\abs[2][empty]{\csname#1\endcsname \lvert{#2}\csname#1\endcsname\rvert}
\newcommand\doublebar[2][empty]{\csname#1\endcsname \lVert{#2}\csname#1\endcsname\rVert}

\newcommand\mat[1]{\boldsymbol{#1}}
\newcommand\arr[1]{\boldsymbol{\dot{#1}}}
\newcommand\dist{\mathop{\mathrm{dist}}\nolimits}
\newcommand\Div{\mathop{\mathrm{div}}\nolimits}
\newcommand\Tr{\mathop{\smash{\arr{\mathrm{Tr}}}\vphantom{T}}\nolimits}

\newcommand\Trace{\mathop{\mathrm{Tr}}\nolimits}
\newcommand\M{\mathop{\smash{\arr{\mathrm{M}}}\vphantom{M}}\nolimits}
\newcommand\Ext{\mathop{\mathrm{Ext}}\nolimits}
\newcommand\supp{\mathop{\mathrm{supp}}\nolimits}
\newcommand\diam{\mathop{\mathrm{diam}}\nolimits}

\newcommand\esssup{\mathop{\mathrm{ess\,sup}}}
\newcommand\re{\mathop{\mathrm{Re}}\nolimits}
\newcommand\im{\mathop{\mathrm{Im}}\nolimits}

\newcommand\R{\mathbb{R}} 
\newcommand\C{\mathbb{C}}

\newcommand\1{\mathbf{1}}

\newcommand\DD{\mathfrak{D}}
\newcommand\NN{\mathfrak{N}}
\newcommand\XX{\mathfrak{X}}
\newcommand\YY{\mathfrak{Y}}

\def\smooth{s}

\newcommand\pmin[1][\smooth]{\pdmnMinusOne/\allowbreak(\dmnMinusOne+#1)}

\def\dmn{{n+1}}
\def\pdmn{{(n+1)}}
\def\dmnMinusOne{n}

\def\pdmnMinusOne{n}

\def\dmn{d}
\def\pdmn{d}
\def\dmnMinusOne{{d-1}}

\def\pdmnMinusOne{{(d-1)}}

\usepackage{tikz}
\usetikzlibrary{calc,intersections,through}

\begin{document}

\title[Perturbation of well posedness]{Perturbation of well posedness for higher order elliptic systems with rough coefficients}

\author{Ariel Barton}
\address{Ariel Barton, Department of Mathematical Sciences,
			309 SCEN,
			University of Ar\-kan\-sas,
			Fayetteville, AR 72701}
\email{aeb019@uark.edu}


\begin{abstract} 
In this paper we study boundary value problems for higher order elliptic differential operators in divergence form. We consider the two closely related topics of inhomogeneous problems and problems with boundary data in fractional smoothness spaces. 

We establish $L^\infty$ perturbative results concerning well posed\-ness of inhomogeneous problems with boundary data in fractional smoothness spaces. 

Combined with earlier known results, this allows us to establish new well posedness results for second order operators whose coefficients are close to being real and $t$-independent and for fourth-order operators close to the biharmonic operator.
\end{abstract}


\maketitle

\setcounter{tocdepth}{2}
\tableofcontents
\setcounter{tocdepth}{3}

% Check for out-of-order references
%\def\mylabel#1{\expandafter\gdef\csname laabeel#1\endcsname{Labeled!}}
%\def\ref#1{\expandafter \ifx \csname laabeel#1\endcsname \relax
%	\typeout{Out of order reference to #1 on line \the\inputlineno!}%
%	\else
%	%\typeout{Correctly ordered reference to #1 on line \the\inputlineno!}
%	\fi 99}
%\let\label\mylabel
%\makeatletter \let \label@in@display \mylabel \makeatother

\section{Introduction}

In this paper we will consider the theory of boundary value problems for elliptic differential operators $L$ of the form
\begin{equation}\label{eqn:divergence}
(L\vec u)_j =\sum_{k=1}^N \sum_{\abs{\alpha}=\abs{\beta}=m}\partial^\alpha (A^{jk}_{\alpha\beta}\,\partial^\beta u_k)\end{equation}
of arbitrary even order~$2m$,
for bounded measurable coefficients~$\mat A$. We will require $\mat A$ to satisfy certain positive definiteness assumptions (see the bounds \eqref{eqn:elliptic:everywhere} and~\eqref{eqn:elliptic:domain} below). 
%, but aside from a brief discussion of a special case in Section~\ref{sec:t-independent}, we will impose no further regularity assumptions on our coefficients. 
In the case of rough coefficients, it is appropriate to formulate the operator~$L$ in the weak sense; we say that $L\vec u=\Div_m\arr H$ in $\Omega$ if
\begin{equation}\label{eqn:weak}
\langle \nabla^m\vec\varphi, \mat A\nabla^m\vec u \rangle_\Omega
=\langle \nabla^m\vec\varphi, \arr H\rangle_\Omega
\quad\text{for all $\vec\varphi\in C^\infty_0(\Omega)$}\end{equation}
where $\langle\,\cdot\,,\,\cdot\,\rangle$ denotes the standard inner product on~$L^2(\Omega)$. %(See formula~\eqref{eqn:rangle} below.)

We are interested in the Dirichlet problem
\begin{equation}
\label{eqn:Dirichlet:introduction}
L\vec u=\Div_m\arr H \text{ in }\Omega,\quad \Tr_{m-1}^\Omega\vec u=\arr f , 
\quad
\doublebar{\vec u}_\XX\leq C \doublebar{\arr f}_\DD+ C \doublebar{\arr H}_\YY\end{equation}
for some appropriate function spaces $\XX$, $\DD$ and~$\YY$. Here $\Tr_{m-1}^\Omega \vec u=\Trace^\Omega \nabla^{m-1}\vec u$, where $\Trace$ is the standard boundary trace operator of Sobolev spaces; see \cite[Definition~2.4]{Bar16pB}.

We are also interested in the Neumann problem. %We will use the (somewhat involved) formulation of Neumann boundary data described in Section~\ref{sec:data} below; 
It turns out that even formulating the Neumann problem in the case of higher order equations is a difficult matter; see \cite{CohG85,Ver05,Agr07,MitM13A,BarHM15p} for some varied formulations and \cite{Ver10,BarM16B,BarHM15p} for a discussion of related issues. Following \cite{BarHM15p}, we will let the Neumann boundary values $\M_{\mat A,\arr H}^\Omega \vec u$ of a solution $\vec u$ to $L\vec u =\Div_m\arr H$ be given by
\begin{equation}\label{dfn:Neumann}
\langle\nabla^m\vec \varphi,\mat A\nabla^m \vec u-\arr H\rangle_\Omega=\langle \nabla^{m-1}\vec \varphi, \M_{\mat A,\arr H}^\Omega \vec u\rangle_{\partial\Omega} 
\text{ for all $\vec\varphi\in C^\infty_0(\R^\dmn)$.}
\end{equation}
Observe that by the weak formulation~\eqref{eqn:weak} of~$L\vec u$ above, if $\partial\Omega$ is connected then the expression $\langle\nabla^m\vec \varphi,\mat A\nabla^m \vec u-\arr H\rangle_\Omega$ depends only on $\Tr_{m-1}^\Omega\vec\varphi$. Thus, formula~\eqref{eqn:Neumann:introduction} defines $\M_{\mat A,\arr H}^\Omega \vec u$ as a linear operator on $\{\nabla^{m-1}\vec\varphi:\vec\varphi\in C^\infty_0(\R^\dmn)\}$. Furthermore, $\langle\nabla^m\vec \varphi,\mat A\nabla^m \vec u-\arr H\rangle_\Omega$ depends only on the values of $\vec \varphi$ near~$\partial\Omega$, and not on the values of $\vec \varphi$ in the interior of~$\Omega$, and so it is reasonable to regard $\M_{\mat A,\arr H}^\Omega \vec u$ as boundary values of~$\vec u$.

We are then interested in the Neumann problem
\begin{equation}
\label{eqn:Neumann:introduction}
L\vec u=\Div_m\arr H \text{ in }\Omega,\quad \M_{\mat A,\arr H}^\Omega \vec u=\arr g, 
\quad
\doublebar{\vec u}_\XX\leq C \doublebar{\arr g}_\NN+ C \doublebar{\arr H}_\YY\end{equation}
for some appropriate function spaces $\XX$, $\NN$ and~$\YY$.

\subsection{The history of the problem and function spaces}
\label{sec:introduction:history}

Consider the two special cases of the Dirichlet problem
\begin{align}
\label{eqn:Dirichlet:inside:introduction}
L\vec u&=\Div_m\arr H \text{ in }\Omega,
&\Tr_{m-1}^\Omega\vec u&=0, 
&\doublebar{\vec u}_\XX&\leq C \doublebar{\arr H}_\YY,
\\
\label{eqn:Dirichlet:boundary:introduction}
L\vec u&=0 \text{ in }\Omega,
&\Tr_{m-1}^\Omega\vec u&=\arr f, 
&\doublebar{\vec u}_\XX&\leq C \doublebar{\arr f}_\DD
.\end{align}
%To establish well posedness of the problem~\eqref{eqn:Dirichlet:boundary:introduction}, it is clearly necessary to require that $\DD\subseteq \{\nabla^{m-1} \vec F \big\vert_{\partial\Omega}:\vec F\in\XX\}$, that is, that for every $\arr f\in\DD$ there is some extension $\vec F$ of $\arr f$ in~$\XX$. To obtain
When studying the problem~\eqref{eqn:Dirichlet:boundary:introduction},
it is often appropriate to choose the function spaces $\XX$ and $\DD$ such that $\DD=\{\Trace \nabla^{m-1} \vec F:\vec F\in\XX\}$. 
Conversely, when studying the problem~\eqref{eqn:Dirichlet:inside:introduction} it is appropriate to choose $\XX=\{\vec F:\mat A\nabla^m\vec F\in\YY\}$. 

With these choices of $\XX$ and~$\DD$, it is possible to reduce the problem \eqref{eqn:Dirichlet:boundary:introduction} to the problem~\eqref{eqn:Dirichlet:inside:introduction}: if we let $\vec F$ be an extension of~$\arr f$, and let $\vec v$ solve the problem~\eqref{eqn:Dirichlet:inside:introduction} with $\arr H=\mat A\nabla^m\vec F$, then $u=\vec F-\vec v$ solves the problem~\eqref{eqn:Dirichlet:boundary:introduction}. 
This technique was used, for example, in \cite{MazMS10} and \cite[Theorem~6.33]{MitM13A}. See also Lemma~\ref{lem:zero:boundary} below.

Conversely, it is often possible to solve $L\vec u=\Div_m\arr H$ in $\R^\dmn$ (see, for example, Section~\ref{sec:newton:bounded} below); under these circumstances, solutions to the problem~\eqref{eqn:Dirichlet:boundary:introduction} may be used to correct the boundary values and solve the problem~\eqref{eqn:Dirichlet:inside:introduction} or the full Dirichlet problem~\eqref{eqn:Dirichlet:introduction}. This technique has been used many times in the literature; see, for example, \cite{JerK95,AdoP98,MitM13B,BarM16A} and \cite[Theorems~6.34 and~6.36]{MitM13A}, or Lemma~\ref{lem:zero:interior} below.

For a number of operators of order~$2m$ with smooth or constant coefficients, the Dirichlet problem \eqref{eqn:Dirichlet:inside:introduction} has been shown to be well-posed in the case where $\Omega$ is a Lipschitz domain, $\XX$ is the Bessel potential space $L^p_{m-1+\smooth+1/p}(\Omega)$, and $\YY=L^p_{\smooth+1/p-1}(\Omega)$, for $0<\smooth<1$ and for certain values of~$p$ depending on $L$, $\smooth$ and~$\Omega$. In particular, 
in \cite{JerK95}, well-posedness was established for $L=\Delta$ and certain $p$ with $1<p<\infty$; some extensions to the case $p\leq 1$ were established in \cite{MayMit04A}. In \cite{AdoP98,MitMW11,MitM13B}, similar results were established for the biharmonic operator $L=\Delta^2$, and in \cite{MitM13A} results were established for general constant-coefficient operators. In the case of operators with variable Lipschitz continuous coefficients, some well-posedness results were established in \cite{Agr07}. %(Some of these results also extend from Bessel potential spaces to more general Besov or Triebel-Lizorkin spaces.)

If solutions $\vec u$ lie in the space $\XX=L^p_{m-1+\smooth+1/p}(\Omega)$, then the appropriate space $\DD$ of Dirichlet boundary values % $\Trace \nabla^{m-1} \vec u$ 
to extend to the problem~\eqref{eqn:Dirichlet:introduction} is the space of Whitney arrays $W\!A^p_{m-1,\smooth}(\partial\Omega)$. This is  a subspace of the Besov space $(B^{p,p}_\smooth(\partial\Omega))^r$, where $r$ is the number of multiindices of length $m-1$; if $m=1$ then $W\!A^p_{0,\smooth}(\partial\Omega)=B^{p,p}_\smooth(\partial\Omega)$, but if $m\geq 2$ then $W\!A^p_{m-1,\smooth}(\partial\Omega)$ is a proper subspace. (This reflects the fact that, if $m-1\geq 1$, then $\nabla^{m-1} u$ is an array of partial derivatives and thus must satisfy appropriate compatibility conditions.) 

%Many of the above results were proven by first solving the Dirichlet problem~\eqref{eqn:Dirichlet:boundary:introduction} with Besov boundary data and then extending to the full Dirichlet problem~\eqref{eqn:Dirichlet:introduction}. 
The parameter $\smooth$ measures smoothness; thus, we emphasize that in the above results, the Dirichlet boundary data $\Tr_{m-1}^\Omega\vec u$ always has between zero and one degree of smoothness.

The Neumann problem \eqref{eqn:Neumann:introduction} has also been studied. In the case of the harmonic operator $L=\Delta$ (\cite{FabMM98,Zan00,MayMit04A}), biharmonic operator $L=\Delta^2$ (\cite{MitM13B}),  and general constant coefficent operators (\cite{MitM13A}), well-posedness has been established in Lipschitz domains, again for $\XX=L^p_{m-1+\smooth+1/p}(\Omega)$ and $\YY=L^p_{\smooth+1/p-1}(\Omega)$, $0<\smooth<1$, and certain values of~$p$. In this case, the appropriate space of boundary data is $\NN={N\!A^p_{m-1,\smooth-1}(\partial\Omega)}$,  a quotient space of the negative smoothness space $(B^{p,p}_{\smooth-1}(\partial\Omega))^r$. %Again, this result is valid for $0<\smooth<1$ and for certain values of~$p$ depending on $L$, $\smooth$ and~$\Omega$.
%This was established in \cite{FabMM98,Zan00,MayMit04A} in the case $L=\Delta$, in \cite{MitM13B} in the case $L=\Delta^2$, and in \cite{MitM13A} for general constant coefficient operators; 
See also \cite{Agr07} for the case of operators with Lipschitz continuous coefficients.

\begin{rmk} In both the Dirichlet and Neumann problems discussed above, boundary values are expected to lie in fractional smoothness spaces.
We may also consider the problem \eqref{eqn:Dirichlet:boundary:introduction}, or the similar Neumann problem
\begin{align}
\label{eqn:Neumann:boundary:introduction}
%\langle\nabla^m\varphi,\mat A\nabla^m u\rangle_\Omega=\langle \nabla^{m-1}\varphi, \arr g\rangle_{\partial\Omega} \text{ for all $\vec\varphi\in C^\infty_0(\R^\dmn)$},
L\vec u&=0 \text{ in }\Omega,
&\M_{\mat A,0}^\Omega\vec u&=\arr g, 
&\doublebar{\vec u}_\XX&\leq C \doublebar{\arr g}_\NN
,\end{align}
with boundary data in integer smoothness spaces (that is, Lebesgue spaces $L^p(\partial\Omega)$ or Sobolev spaces $\dot W^p_{\pm 1}(\partial\Omega)$). However, this requires spaces of solutions $\XX$ for which the corresponding problem \eqref{eqn:Dirichlet:inside:introduction} is ill-posed (even in particularly nice cases, such as the case of harmonic functions $L=\Delta$ in the upper half-space~$\R^\dmn_+$). Thus, the theory of inhomogeneous problems ($L\vec u=\Div_m\smash{\arr H}$ rather than $L\vec u=0$) is deeply and inextricably connected to the theory of boundary data in fractional smoothness spaces.

Although we will not consider boundary data in integer smoothness spaces, we mention some of the known results. The Dirichlet problem for the biharmonic operator $\Delta^2$ or polyharmonic operator~$\Delta^m$, $m\geq 3$, with data in $L^p(\partial\Omega)$, was investigated in \cite{SelS81,CohG83,DahKV86,Ver87,Ver90,She06A}, and with data in the Sobolev space $W^p_1(\partial\Omega)$ in \cite{Ver90,PipV92,MitM10,KilS11A}. The $L^p$ or $W^p_1$-Dirichlet problems for more general higher order constant coefficient operators were investigated in \cite{PipV95A,Ver96,She06B,KilS11B}. The $L^p$-Neumann problem has been investigated for the biharmonic operator in \cite{CohG85,Ver05,She07A,MitM13A}. Very little is known in the case of higher order variable coefficient operators; the $L^2$-Neumann problem for self-adjoint $t$-independent coefficients in the half-space $\Omega=\R^\dmn_+$ was shown to be well posed in the recent preprint \cite{BarHM17pC}, and the Dirichlet problem for fourth-order operators of a form other than \eqref{eqn:divergence} with $L^2$ boundary data was solved in \cite{BarM13}. See the author's survey paper with  Mayboroda \cite{BarM16B} for a more extensive discussion of these results. We omit discussion of the extensive literature concerning second order boundary value problems (the case $m=1$) with data in integer smoothness spaces.
\end{rmk}

We are interested in boundary value problems for operators of the form \eqref{eqn:weak} with rough coefficients~$\mat A$. We still consider boundary data in Besov spaces. However, the space $\XX=L^p_{m-1+\smooth+1/p}(\Omega)$ is not an appropriate space in which to seek solutions, because this space requires that the gradient $\nabla^m\vec u$ of a solution $\vec u$ display $\smooth+1/p-1$ degrees of smoothness, and if $\mat A$ is rough then $\nabla^m u$ may be rough as well. See \cite[Chapter~10]{BarM16A}. Another possible solution space $\XX$ is suggested by the theory for constant coefficients. In \cite{JerK95} and~\cite{AdoP98}, it was established that if $\Delta^m u=0$ in~$\Omega$, for $m=1$ or $m=2$, then for appropriate $p$ and $\smooth$ we have that $u\in L^p_{m-1+\smooth+1/p}(\Omega)$ if and only if $u\in W^p_{m-1}(\Omega)$ and
\begin{equation*}\int_\Omega \abs{\nabla^m u(x)}^p\dist(x,\partial\Omega)^{p-1-p\smooth}\,dx<\infty.\end{equation*}
%See also \cite{DahKPV97}.
We may thus seek to control the above norm of our solutions, rather than the $L^p_{m-1+\smooth+1/p}(\Omega)$-norm.

In \cite{MazMS10}, Maz'ya, Mitrea and Shaposhnikova established well-posedness of the Dirichlet problem~\eqref{eqn:Dirichlet:introduction} for operators~$L$ with variable $VMO$ coefficients, with $\DD=W\!A^p_{m-1,\smooth}(\partial\Omega)$ as usual, but with the norm on solutions given by
\begin{equation}\label{eqn:norm:MazMS10}
\doublebar{\vec u}_{W^{p,\smooth}_m(\Omega)} = \biggl(\sum_{k=0}^m \int_\Omega \abs{\nabla^k \vec u(x)}^p \dist(x,\partial\Omega)^{p-1-p\smooth}\,dx\biggr)^{1/p}
.\end{equation}

In \cite{BarM16A},  Mayboroda and the author of the present paper investigated the Dirichlet and Neumann problems \eqref{eqn:Dirichlet:introduction} and~\eqref{eqn:Neumann:introduction} in the case $m=N=1$, for coefficients constant in the vertical direction but merely bounded measurable in the horizontal directions, in the domain above a Lipschitz graph. We established well-posedness for certain $\smooth$ and~$p$ in the case of real symmetric coefficients (the Neumann problem) or general real coefficients (the Dirichlet problem), with $\DD=\dot B^{p,p}_\smooth(\partial\Omega)$ and $\NN=\dot B^{p,p}_{\smooth-1}(\partial\Omega)$ as usual. %(In the case $m=1$ we may take $\DD$ and $\NN$ to be the full Besov spaces; the additional complication of Whitney arrays is not necessary.) 
We used a somewhat different choice of solution space~$\XX$; specifically, we let $\XX=\dot W^{p,\smooth}_{1,av}(\Omega)$, where $\dot W^{p,\smooth}_{m,av}(\Omega)$ is the set of (equivalence classes of) functions $u$ for which the $\dot W^{p,\smooth}_{m,av}(\Omega)$-norm given by
\begin{align}\label{eqn:W:norm:2}
\doublebar{u}_{\dot W^{p,\smooth}_{m,av}(\Omega)}&=\doublebar{\nabla^m u}_{ L^{p,\smooth}_{av}(\Omega)},
\\
\label{eqn:L:norm:2}
\doublebar{\arr H}_{L^{p,\smooth}_{av}(\Omega)}
&= 
\biggl(\int_\Omega \biggl(\fint_{B(x,\dist(x,\allowbreak \partial\Omega)/2)} \abs{\arr H}^2 \biggr)^{p/2}  \dist(x,\allowbreak \partial\Omega)^{p-1-p\smooth}\,dx\biggl)^{1/p}
\end{align}
is finite. (Two functions are equivalent if their difference has norm zero; if $\Omega$ is open and connected then two functions are equivalent if they differ by a polynomial of degree at most $m-1$.)
Here $\fint$ denotes the averaged integral $\fint_B H = \frac{1}{\abs{B}}\int_B H$. This norm was also used in \cite{AmeA16p}, where somewhat more general second order operators (in particular, the case $N\geq 1$) was considered.

We chose to use homogeneous norms (that is, to bound only $\nabla^m u$ and not any of the lower order derivatives) because we were working in unbounded domains, and homogeneous norms are in many ways more convenient in that context. It is also possible to consider homogeneous norms in bounded Lipschitz domains. In particular, if $\partial\Omega$ is connected then $\Tr_{m-1}^\Omega\vec u$ determines the lower-order derivatives up to adding polynomials, and so we can recover the inhomogeneous results. However, if $\partial\Omega$ is disconnected, then $\Tr_{m-1}^\Omega\vec u$ does not determine the lower-order derivatives, and so throughout this paper we will consider only domains with connected boundary.

The $L^{p,\smooth}_{av}$-norm of \cite{BarM16A} involves $L^2$ averages over Whitney balls. These averages are useful both in the case of $p$ large and in the case $p<1$. 

Finiteness of the $W^{p,\smooth}_m$-norm \eqref{eqn:norm:MazMS10} requires that the gradient $\nabla^m \vec u$ of a solution be locally $p$th-power integrable. This is a reasonable assumption, even for $p$ large, if the coefficients are constant, or even merely~$VMO$. However, for rough coefficients, the best we can expect is for $\nabla^m\vec u$ to be locally square-integrable, or at best $q$th-power integrable for $q<2+\varepsilon$ and $\varepsilon$ possibly very small. (In the second-order case, this expectation comes from the Caccioppoli inequality and Meyers's reverse H\"older inequality; both may be generalized to the higher order case, as in \cite{Cam80, AusQ00, Bar16}.) The technique of bounding $L^2$ averages of gradients on Whitney balls, rather than the gradients themselves, is common in the theory of elliptic differential equations; see, for example, the modified nontangential maximal function introduced in \cite{KenP93} and used extensively in the literature.

Conversely, if $\smooth>0$ and $p\geq1$, then finiteness of the $W^{p,\smooth}_m$-norm  \eqref{eqn:norm:MazMS10} ensures that $\nabla^m \vec u$ is locally integrable up to the boundary. This useful property ensures that the Dirichlet and Neumann boundary values of $\vec u$ are meaningful. However, if $p<1$ then finiteness of the $W^{p,\smooth}_m$-norm \eqref{eqn:norm:MazMS10} does not ensure local integrability, and so it is not clear that the trace operator is well-defined on such spaces. However, if $\smooth>0$ and $p>\pmin$, then finiteness of the $\dot W^{p,\smooth}_{m,av}$-norm \eqref{eqn:W:norm:2} does ensure local integrability; see \cite[Theorem~6.1]{BarM16A} or \cite[Lemma~3.7]{Bar16pB}. Thus, using the averaged norm allows us to consider at least some values of $p<1$.

We remark that the requirement $p>\pmin$ appears also in \cite{MayMit04A,MitM13A}, and for similar reasons: the condition $u\in L^{p}_{m-1+\smooth+1/p}(\Omega)$ ensures local integrability of $\nabla^m u$ for precisely the given range of $p$. %(A stronger requirement, still allowing some values of $p<1$, appears in \cite{MitMW11}, as well as .)

In this paper, we will investigate the Dirichlet problem
\begin{equation}
\label{eqn:Dirichlet:p:smooth}
\left\{
\begin{aligned}
L\vec u&=\Div_m\arr H\text{ in }\Omega,\\ 
\Tr_{m-1}^\Omega\vec u&=\arr f, \\ 
\doublebar{\vec u}_{\dot W^{p,\smooth}_{m,av}(\Omega)} &\leq C\doublebar{\arr H}_{L^{p,\smooth}_{av}(\Omega)}
+C\doublebar{\arr f}_{\dot W\!A^p_{m-1,\smooth}(\partial\Omega)}
\end{aligned}\right.\end{equation}
and the Neumann problem
\begin{equation}
\label{eqn:Neumann:p:smooth}
\left\{\begin{aligned}
L\vec u&=\Div_m\arr H\text{ in }\Omega,\\ 
\M_{\mat A,\arr H}^\Omega\vec u&=\arr g, \\ 
\doublebar{\vec u}_{\dot W^{p,\smooth}_{m,av}(\Omega)} &\leq C\doublebar{\arr H}_{L^{p,\smooth}_{av}(\Omega)}
+C\doublebar{\arr g}_{\dot N\!A^p_{m-1,\smooth-1}(\partial\Omega)}
\end{aligned}\right.
\end{equation}
where ${\dot W^{p,\smooth}_{m,av}(\Omega)}$ and ${L^{p,\smooth}_{av}(\Omega)}$ are given by formulas~\eqref{eqn:W:norm:2} and~\eqref{eqn:L:norm:2}, and the boundary spaces $\dot W\!A^p_{m-1,\smooth}(\partial\Omega)$ and $\dot N\!A^p_{m-1,\smooth-1}(\partial\Omega)$ are as defined in \cite[Section~2.2]{Bar16pB}. These are the homogeneous counterparts to the spaces mentioned above; the main result of \cite{Bar16pB} is that they are the spaces of Dirichlet and Neumann traces of ${\dot W^{p,\smooth}_{m,av}(\Omega)}$ functions.

We say that these problems are well posed if, for every $\arr H\in L^{p,\smooth}_{av}(\Omega)$ and every $\arr f\in {\dot W\!A^p_{m-1,\smooth}(\partial\Omega)}$ or $\arr g\in {\dot N\!A^p_{m-1,\smooth-1}(\partial\Omega)}$, there is exactly one  $\vec u \in {\dot W^{p,\smooth}_{m,av}(\Omega)}$ that satisfies $L\vec u=\Div_m\arr H$ and $\Tr_{m-1}^\Omega\vec u=\arr f$ or $\M_{\mat A,\arr H}^\Omega\vec u=\arr g$, and if that $\vec u$ satisfies the given quantitative bound $\doublebar{\vec u}_{\dot W^{p,\smooth}_{m,av}(\Omega)} \leq C\doublebar{\arr H}_{L^{p,\smooth}_{av}(\Omega)}
+C\doublebar{\arr f}_{\dot W\!A^p_{m-1,\smooth}(\partial\Omega)}$ or $\smash{\doublebar{\vec u}_{\dot W^{p,\smooth}_{m,av}(\Omega)} \leq C\doublebar{\arr H}_{L^{p,\smooth}_{av}(\Omega)}
+C\doublebar{\arr g}_{\dot N\!A^p_{m-1,\smooth-1}(\partial\Omega)}}$.

%We remark that the main result of \cite{Bar16pB} is that if $\pmin<p\leq\infty$, then 
%\begin{equation}\label{eqn:Dirichlet:trace}
%\dot W\!A^p_{m-1,\smooth}(\partial\Omega)=\{\Tr_{m-1}^\Omega \vec F:\vec F\in \dot W^{p,\smooth}_{m,av}(\Omega)\},
%\end{equation}
%and if $1<p\leq \infty$ then
%\begin{equation}\label{eqn:Neumann:trace}
%\dot N\!A^p_{m-1,\smooth-1}(\partial\Omega)=\{\M_m^\Omega \arr G:\arr G\in L^{p,\smooth}_{av}(\Omega),\>\Div_m\arr G=0\}.
%\end{equation}
%We conjecture that this second relation is also true for $\pmin<p\leq \infty$ (and can prove it in some special cases, see \cite[Theorem~3.5]{Bar16pB}).

\subsection{The main results of this paper}

The most general result of this paper is the following theorem.

\begin{thm}\label{thm:neighborhood} Let $\Omega\subset\R^\dmn$ be a Lipschitz domain with connected boundary. Let $L$ be an elliptic system of the form~\eqref{eqn:divergence}, whose coefficients $\mat A$ satisfy the ellipticity condition~\eqref{eqn:elliptic:domain} and the uniform boundedness condition~\eqref{eqn:elliptic:bounded}.

Then there is some $\varepsilon>0$, depending only on~$m$, $\dmn$, the Lipschitz character of~$\Omega$, and the constants $\lambda$ and $\Lambda$ in formulas~\eqref{eqn:elliptic:bounded} and~\eqref{eqn:elliptic:domain}, such that if $\abs{p-2}<\varepsilon$ and $\abs{\smooth-1/2}<\varepsilon$, then the Neumann problem~\eqref{eqn:Neumann:p:smooth} is well posed.
\end{thm}

We remark that this theorem is a well posedness result valid for \emph{all} bounded elliptic coefficient matrices~$\mat A$; we impose no smoothness assumptions on the coefficients.
We will prove this theorem in Section~\ref{sec:neighborhood}.

The $p=2$, $\smooth=1/2$ case of Theorem~\ref{thm:neighborhood} follows from the Lax-Milgram lemma by a straightforward and well known argument. 
An equivalent result for the Dirichlet problem was proven in \cite[Section~8.1]{MazMS10}; in the case $m=1$, see also \cite[Theorem~1]{Mey63} and \cite[Theorem~5.1]{BreM13}.
If $\smooth=1-1/p$ (with no restrictions on $m$, and for either the Dirichlet or Neumann problems), then the result was established by Brewster, D.~Mitrea, I.~Mitrea, and M.~Mitrea in \cite{BreMMM14}, in somewhat more general domains. If $L$ has constant coefficients, then a very similar theorem (using Bessel potential spaces and more general Besov spaces) was established by I.~Mitrea and M.~Mitrea in \cite{MitM13A} using the method of layer potentials.

%In this paper we will establish some further results. These results imply well posedness of boundary value problems for certain coefficients $\mat B$ and exponents $p$, $\smooth$, {given} well posedness for related coefficents $\mat A$ or related exponents $q$,~$\sigma$. In the next section we shall use known results from the literature, in combination with our Theorems \ref{thm:perturb} and Lemmas \ref{lem:BVP:duality} and~\ref{lem:interpolation}, to produce new well posedness results for some particular coefficients and exponents.

Our second main result is a perturbative result.
Theorem~\ref{thm:perturb} states that, if boundary value problems for some operator $L$ are well posed, then they are also well posed for any other operator $M$ whose coefficients are sufficiently close (in $L^\infty(\R^\dmn)$) to those of~$L$. 
We will prove Theorem~\ref{thm:perturb} in Section~\ref{sec:perturb}. In Section~\ref{sec:perturb:independent}, we will discuss the history of such perturbation results. In Section~\ref{sec:A0:introduction}, we will combine Theorem~\ref{thm:perturb} with known results from the literature to establish new well posedness results.

\begin{thm}\label{thm:perturb}
Let $\Omega\subset\R^\dmn$ be a Lipschitz domain with connected boundary, let $0<\smooth<1$, and let $\pmin<p<\infty$. Then there is some constant $C_1$ depending only on $p$, $\smooth$, the dimension $\dmn$, and the Lipschitz character of~$\Omega$, with the following significance.

Let $L$ and $M$ be operators of the form \eqref{eqn:divergence} acting on functions defined in $\R^\dmn$, with the same values of $m$ and~$N$, associated to bounded coefficients $\mat A$ and~$\mat B$. Let $\varepsilon= \doublebar{\mat A-\mat B}_{L^\infty(\R^\dmn)}$.

Suppose that %$\mat A$ satisfies the ellipticity condition~\eqref{eqn:elliptic:everywhere} for some $\lambda>0$, and that 
there is some constant $C_0$ such that  the Dirichlet problem
\begin{equation}
\label{eqn:Dirichlet:0}
%\left\{\begin{aligned}
{L} \vec u = \Div_m \arr \Phi \text{ in }\Omega,
\quad \Tr_{m-1}^\Omega\vec u = 0,
\quad
\doublebar{\vec u}_{\dot W^{p,\smooth}_{m,av}(\Omega)} \leq C_0 \doublebar{\arr \Phi}_{L^{p,\smooth}_{av}(\Omega)}
%\end{aligned}\right.
\end{equation}
is well posed.
If $\varepsilon<1/C_1C_0$, then the Dirichlet problem
\begin{equation}
\label{eqn:Dirichlet:1}
\left\{\begin{aligned}
{M} \vec u &= \Div_m \arr H \text{ in }\Omega,
\\ \Tr_{m-1}^\Omega \vec u&=\arr f,
\\
\doublebar{\vec u}_{\dot W^{p,\smooth}_{m,av}(\Omega)} &\leq C_2\doublebar{\arr H}_{L^{p,\smooth}_{av}(\Omega)}
+C_2C_3\doublebar{\arr f}_{\dot W\!A^p_{m-1,\smooth}(\partial\Omega)}
\end{aligned}\right.
\end{equation}
is well posed. Here $C_3$ depends only on $p$, $\smooth$, the dimension $\dmn$ and the Lipschitz character of~$\Omega$, and  
\begin{equation*}C_2=\frac{C_0}{1-C_0\varepsilon}\text{ if }p\geq 1,\qquad
C_2=\biggl(\frac{C_0^p}{1-C_0^p\varepsilon^p}\biggr)^{1/p}\text{ if }p\leq 1.\end{equation*}

Similarly, suppose that %$\mat A$ satisfies the ellipticity condition~\eqref{eqn:elliptic:domain} for some $\lambda>0$, and that 
there is some constant $C_0$ such that the Neumann problem
\begin{equation}
\label{eqn:Neumann:0}
%\left\{\begin{aligned}
{L} \vec u = \Div_m \arr \Phi \text{ in }\Omega,
\quad \M_{\mat A,\arr \Phi}^\Omega\vec u = 0,
\quad
\doublebar{\vec u}_{\dot W^{p,\smooth}_{m,av}(\Omega)} \leq  C_0 \doublebar{\arr \Phi}_{L^{p,\smooth}_{av}(\Omega)}
%\end{aligned}\right.
\end{equation}
is well posed.
If $\varepsilon<1/C_1C_0$, then the Neumann problem
\begin{equation}
\label{eqn:Neumann:1}
\left\{
\begin{aligned}
{M} \vec u &= \Div_m \arr H \text{ in }\Omega,
\\ \M_{\mat B,\arr H}^\Omega \vec u&=\arr g,
\\
\doublebar{\vec u}_{\dot W^{p,\smooth}_{m,av}(\Omega)} &\leq  C_2 \doublebar{\arr H}_{L^{p,\smooth}_{av}(\Omega)}
+C_2C_3 \doublebar{\arr g}_{\dot N\!A^p_{m-1,\smooth-1}(\partial\Omega)}
\end{aligned}\right.
\end{equation}
is also well posed.
\end{thm}

In applying Theorem~\ref{thm:perturb}, especially in analyzing a range of $p$ and $\smooth$, the following two lemmas are often helpful. Lemma~\ref{lem:BVP:duality} is a duality result; Lemma~\ref{lem:interpolation} is an interpolation result. Both will be used in Section~\ref{sec:A0:introduction}.

\begin{lem}\label{lem:BVP:duality}
Let $\Omega$ be a Lipschitz domain with connected boundary and let $L$ be an operator of the form~\eqref{eqn:divergence} associated to bounded coefficients~$\mat A$. Let $0<\smooth<1$ and let $1\leq p< \infty$.

Let $\smooth'=1-\smooth$, and let $p'$ be the extended real number that satisfies $1/p+1/p'=1$.
Let $(A^*)^{jk}_{\alpha\beta}=\overline{A^{kj}_{\beta\alpha}}$, and let $L^*$ be the operator of the form~\eqref{eqn:divergence} associated to the coefficients $\mat A^*$.

If there is a constant $C_0$ such that the Dirichlet problem 
\begin{equation*}
L\vec u=\Div_m\arr H,\quad \Tr_{m-1}^\Omega\vec u=\arr f,\quad \doublebar{\vec u}_{\dot W^{p,\smooth}_{m,av}(\Omega)}\leq C_0\doublebar{\arr H}_{L^{p,\smooth}_{av}(\Omega)}
+ C_0\doublebar{\arr f}_{\dot W\!A^p_{m-1,\smooth}(\partial\Omega)}
\end{equation*}
is well posed, then there is a constant $C_1$ such that the problem
\begin{equation*}
L^*\vec u=\Div_m\arr \Phi,\>\>\>\> \Tr_{m-1}^\Omega\vec u=\arr \varphi,\>\>\>\> \doublebar{\vec u}_{\dot W^{p',\smooth'}_{m,av}(\Omega)}\leq C_1\doublebar{\arr \Phi}_{L^{p',\smooth'}_{av}(\Omega)}
+ C_1\doublebar{\arr \varphi}_{\dot W\!A^{p'}_{m-1,\smooth'}(\partial\Omega)}
\end{equation*}
is well posed.

Similarly, if there is a constant $C_0$ such that the Neumann problem
\begin{equation*}
L\vec u=\Div_m\arr H,\>\>\>\> \M_{\mat A,\arr H}^\Omega\vec u=\arr g,\>\>\>\> \doublebar{\vec u}_{\dot W^{p,\smooth}_{m,av}(\Omega)}\leq C_0\doublebar{\arr H}_{L^{p,\smooth}_{av}(\Omega)}
+ C_0\doublebar{\arr g}_{\dot N\!A^p_{m-1,\smooth-1}(\partial\Omega)}
\end{equation*}
is well posed, then there is a $C_1$ such that the problem
\begin{equation*}
L^*\vec u=\Div_m\arr \Phi,\>\>\>\> \M_{\mat A^*,\arr \Phi}^\Omega\vec u=\arr \varphi,\>\>\>\> \doublebar{\vec u}_{\dot W^{p',\smooth'}_{m,av}(\Omega)}\leq C_1\doublebar{\arr \Phi}_{L^{p',\smooth'}_{av}(\Omega)}
+ C_1\doublebar{\arr \varphi}_{\dot N\!A^{p'}_{m-1,\smooth'-1}(\partial\Omega)}
\end{equation*}
is well posed.

\end{lem}
Lemma~\ref{lem:BVP:duality} will be proven in Section~\ref{sec:duality}, as the $p\geq 1$, $p'\geq 1$ case of Theorems~\ref{thm:exist:unique} and~\ref{thm:unique:exist}.

\begin{lem}\label{lem:interpolation}
Let $\Omega$ be a Lipschitz domain with connected boundary and let $L$ be an operator of the form~\eqref{eqn:divergence}. Let $0<\smooth_0<1$ and $0<\smooth_1<1$. 
Let $p_0$ and $p_1$ satisfy $\pmin[\smooth_j]<p_j<\infty$ for $j=0$, $1$.

If $0\leq \sigma\leq 1$, then let $\smooth_\sigma=(1-\sigma)\smooth_0+\sigma\smooth_1$ and let $1/p_\sigma=(1-\sigma)/p_0+\sigma/p_1$.

Suppose that the Dirichlet problems
\begin{gather}
\label{eqn:interpolation:Dirichlet:0}
L\vec u=\Div_m\arr\Phi,\quad \Tr_{m-1}^\Omega\vec u=0,\quad \doublebar{\vec u}_{\dot W^{p_0,\smooth_0}_{m,av}(\Omega)}\leq C_0\doublebar{\arr\Phi}_{L^{p_0,\smooth_0}_{av}(\Omega)}
,\\
\label{eqn:interpolation:Dirichlet:1}
L\vec u=\Div_m\arr\Psi,\quad \Tr_{m-1}^\Omega\vec u=0,\quad \doublebar{\vec u}_{\dot W^{p_1,\smooth_1}_{m,av}(\Omega)}\leq C_1\doublebar{\arr\Psi}_{L^{p_1,\smooth_1}_{av}(\Omega)}
\end{gather}
are both well posed. Suppose furthermore that they are compatibly well posed in the sense that, if $\arr\Phi\in {L^{p_0,\smooth_0}_{av}(\Omega)}\cap{L^{p_1,\smooth_1}_{av}(\Omega)}$, then there is a single solution $\vec u\in {\dot W^{p_0,\smooth_0}_{m,av}(\Omega)}\cap{\dot W^{p_1,\smooth_1}_{m,av}(\Omega)}$ to both problems.

Then for every $0<\sigma<1$, there is some $C>0$ depending on $\sigma$, $p_j$, $\smooth_j$, $C_j$ and~$\Omega$, such that for every $\arr H\in {L^{p_\sigma,\smooth_\sigma}_{av}(\Omega)}$ and $\arr f\in{\dot W\!A^{p_\sigma}_{m-1,\smooth_\sigma}(\partial\Omega)}$, there exists at least one solution to the Dirichlet problem
\begin{multline}\label{eqn:interpolation:Dirichlet}
L\vec u=\Div_m\arr H,\quad \Tr_{m-1}^\Omega\vec u=\arr f,\\ \doublebar{\vec u}_{\dot W^{p_\sigma,\smooth_\sigma}_{m,av}(\Omega)}\leq C\doublebar{\arr H}_{L^{p_\sigma,\smooth_\sigma}_{av}(\Omega)}
+ C\doublebar{\arr f}_{\dot W\!A^{p_\sigma}_{m-1,\smooth_\sigma}(\partial\Omega)}
.\end{multline}
If $1<p_0<\infty$ and $1<p_1<\infty$, 
then there is at most one solution to the Dirichlet problem~\eqref{eqn:interpolation:Dirichlet} and so the problem is well posed.

Similarly,
if the Neumann problems
\begin{gather}
\label{eqn:interpolation:Neumann:0}
L\vec u=\Div_m\arr\Phi,\quad \M_{\mat A,\arr \Phi}^\Omega\vec u=0,\quad \doublebar{\vec u}_{\dot W^{p_0,\smooth_0}_{m,av}(\Omega)}\leq C_0\doublebar{\arr\Phi}_{L^{p_0,\smooth_0}_{av}(\Omega)}
,\\
\label{eqn:interpolation:Neumann:1}
L\vec u=\Div_m\arr\Psi,\quad \M_{\mat A,\arr \Psi}^\Omega\vec u=0,\quad \doublebar{\vec u}_{\dot W^{p_1,\smooth_1}_{m,av}(\Omega)}\leq C_1\doublebar{\arr\Psi}_{L^{p_1,\smooth_1}_{av}(\Omega)}
\end{gather}
are both well posed and are compatibly well posed, then for every $0<\sigma<1$ the Neumann problem
\begin{multline}\label{eqn:interpolation:Neumann}
L\vec u=\Div_m\arr H,\quad \M_{\mat A,\arr H}^\Omega\vec u=\arr g,\\ \doublebar{\vec u}_{\dot W^{p_\sigma,\smooth_\sigma}_{m,av}(\Omega)}\leq C\doublebar{\arr H}_{L^{p_\sigma,\smooth_\sigma}_{av}(\Omega)}
+ C\doublebar{\arr g}_{\dot N\!A^{p_\sigma}_{m-1,\smooth_\sigma-1}(\partial\Omega)}
\end{multline}
has solutions, and if $1<p_0<\infty$ and $1<p_1<\infty$ then the problem is well posed.
\end{lem}
This lemma will be proven at the end of Section~\ref{sec:interpolation}. 
We remark that $\{(\smooth_\sigma,1/p_\sigma):0<\sigma<1\}$ is the straight line segment connecting the points $(\smooth_0,1/p_0)$ and $(\smooth_1,1/p_1)$. 

By Corollary~\ref{cor:compatible} below, if 
$\pdmnMinusOne/p_0-\smooth_0=\pdmnMinusOne/p_1-\smooth_1$, or 
if $\Omega$ is bounded, $0<\smooth_0<\smooth_1<1$, and  $\pdmnMinusOne/p_1-\smooth_1\leq \pdmnMinusOne/p_0 -\smooth_0$, then the problems \eqref{eqn:interpolation:Dirichlet:0} and~\eqref{eqn:interpolation:Dirichlet:1} or \eqref{eqn:interpolation:Neumann:0} and~\eqref{eqn:interpolation:Neumann:1} are compatible in the above sense; furthermore, by Corollary~\ref{cor:unique:extrapolate}, solutions to the problem~\eqref{eqn:interpolation:Dirichlet} or~\eqref{eqn:interpolation:Neumann} are unique and thus the problems are well posed.

The compatibility condition is not trivial; the main result of \cite{Axe10} is an example of a second order operator $L$ such that the Dirichlet problems
\begin{align*}Lu=0 \text{ in }\R^2_+,\quad \Trace u &= f, \quad \doublebar{u}_{T^p_\infty}\leq \doublebar{f}_{L^p(\partial\R^2_+)}
,\\
Lv=0 \text{ in }\R^2_+,\quad \Trace v &= f, \quad \doublebar{u}_{\dot W^2_1(\R^2_+)}\leq \doublebar{f}_{\dot B^{2,2}_{1/2}(\partial\R^2_+)}\end{align*}
are both well posed, but for which $u\neq v$ for some $f\in {L^p(\partial\R^2_+)}\cap{\dot B^{2,2}_{1/2}(\partial\R^2_+)}$. Here $T^p_\infty$ is the tent space defined in \cite{CoiMS85}.

\subsection{Historical remarks on $L^\infty$ perturbation} 
\label{sec:perturb:independent}

Perturbation results such as Theorem~\ref{thm:perturb} have been of  interest in recent years. We mention one particular class of coefficients that has received a great deal of study.
Suppose that $m=1$, that $\Omega=\{(x',t):x'\in\R^\dmnMinusOne,\>t>\psi(x')\}$ is the domain above a Lipschitz graph, and that
$\mat A$ is $t$-independent in the sense that
\begin{equation}\label{eqn:t-independent}
\mat A(x',t)=\mat A(x',s)=\mat A(x')\qquad\text{for all $x'\in\R^\dmnMinusOne$ and all $s$, $t\in\R$.}
\end{equation}
Then well-posedness of the Dirichlet problem \eqref{eqn:Dirichlet:boundary:introduction} for $L=\Div\mat A\nabla$, with $\DD=L^2(\Omega)$ and $\XX = \widetilde T^p_\infty$, implies well-posedness of the Dirichlet problem $M=\Div\mat B\nabla$, for $t$-independent $\mat B$ sufficiently close to $\mat A$. This was established in full generality in \cite{AusAM10}, and some previous versions were established in \cite{FabJK84,AusAH08,AlfAAHK11}. 

Here $\widetilde T^p_\infty$ is a ``nontangential space''  %given by
%\begin{equation*}\doublebar{u}_{\widetilde T^p_\infty}=\biggl(\int_{\R^\dmnMinusOne} \sup_{\abs{x'-y'}<t} \abs{u(y',t)}^p\,dx'\biggr)^{1/p}.\end{equation*} This space is 
appropriate for studying boundary data in~$L^p(\partial\Omega)$. (The space $\widetilde T^p_\infty$ is a generalization, essentially introduced in \cite{KenP93}, of the tent space $T^p_\infty$ defined in \cite{CoiMS85}. See, for example, \cite{HofMayMou15} for a precise definition of~$\widetilde T^p_\infty$.)

Similar results are valid for the Neumann problem~\eqref{eqn:Neumann:boundary:introduction} with $\NN=L^2(\partial\Omega)$ and $\XX=\nabla^{-1}(\widetilde T^2_\infty)$, where $\nabla^{-1}(\widetilde T^2_\infty)$ is the space of functions whose gradients lie in a nontangential space, and the Dirichlet problem \eqref{eqn:Dirichlet:boundary:introduction}  with $\DD=\dot W_1^2(\partial\Omega)$ and $\XX = \nabla^{-1}(\widetilde T^2_\infty)$, often called the ``Dirichlet regularity'' problem. Indeed, the stability result established in \cite{AlfAAHK11} required well-posedness of all three boundary value problems for~$\mat A$ to establish any well-posedness results for~$\mat B$.

Some similar stability results are available for boundary data in $L^p(\partial\Omega)$ for more general~$p$;
in particular, such stability results follow from the boundedness of layer potentials for $t$-independent coefficients established in \cite{HofMitMor15} and the well known equivalence between invertibility of layer potentials and well-posedness of boundary value problems. (See \cite{Ver84,BarM13,BarM16A,HofKMP15B,Bar17pA}.) 
%in particular, in \cite{Bar13}, it was established that if the ambient dimension $\dmn=2$, if $\mat B$ and $\mat A$ are both $t$-independent and close in~$L^\infty$, and if $\mat A$ is real-valued, then we have well-posedness of the $(\widetilde T^p_\infty,\allowbreak 0,\allowbreak \allowbreak L^p(\partial\Omega),\allowbreak \allowbreak \mat B,\allowbreak \Omega)$-Dirichlet problem, $(\nabla^{-1}(\widetilde T^q_\infty),\allowbreak 0,\allowbreak \allowbreak L^q(\partial\Omega),\allowbreak \allowbreak \mat A,\allowbreak \Omega)$-Neumann problem, and $(\nabla^{-1}(\widetilde T^q_\infty),0,\allowbreak \dot W^q_1(\partial\Omega),\allowbreak \mat A,\Omega)$-Dirichlet problem for any Lipschitz domain~$\Omega$ with connected boundary and any $p<\infty$ large enough and $q>1$ small enough. (Well-posedness of these problems for real-valued $\mat A$ was established in \cite{KenKPT00} and in \cite{KenR09,Rul07}.) In \cite{HofMitMor15} it was shown that these results are true for $\Omega=\R^\dmn_+$ in arbitrary dimensions provided $\mat A$ is real symmetric; if $\mat A$ is merely real then the results for the Dirichlet problems are still known to be true (see \cite{HofMitMor15} combined with \cite{HofKMP15B,HofKMP15A}) and the 

A higher order stability result, for boundary data in $L^2(\partial\R^\dmn_+)$, was established in \cite[Theorem~1.12]{BarHM17pC}.

Thus far, perturbation results for $\mat B$ \emph{not} independent of~$t=x_\dmn$ have been limited to Carleson-measure perturbation rather than $L^\infty$ perturbation; that is, well posedness results for $\mat A$ extend to well posedness results for $\mat B$ provided 
\begin{equation*}\sup_{x\in\partial\Omega,\>r>0} \biggl(\frac{1}{r^\dmnMinusOne} \int_{B(x,r)\cap\Omega} 
\Bigl(\sup_{B(y,\dist(y,\partial\Omega)/2)}\abs{\mat B-\mat A} ^2\Bigr) \frac{1}{\dist(y,\partial\Omega)}\,dy\biggr)^{1/2}\end{equation*}
is small. %(Here $B(y,\Omega)=B(y,\dist(y,\partial\Omega)/2)$ is the Whitney ball centered at~$y$.) 
Notice that this is a stronger condition than smallness of $\doublebar{\mat B-\mat A}_{L^\infty(\R^\dmn)}$. See the papers \cite{Dah86a, Fef89, FefKP91, Fef93, KenP93, KenP95, DinPP07, DinR10, AusA11, AusR12, HofMayMou15} for such Carleson perturbation results. We remark that for the most part, the known results for Carleson measure perturbation concern well-posedness of problems with boundary data in integer smoothess spaces.

Our Theorem~\ref{thm:perturb} allows for $L^\infty$ perturbation of coefficients that are not $t$-independent; however, it also concerns only boundary value problems with boundary data in fractional smoothness spaces, rather than the Lebesgue and Sobolev spaces mentioned above.

\subsection{New well posedness results derived from Theorems~\ref{thm:perturb}}
\label{sec:A0:introduction}

In this section we review some known well posedness results from the literature, and we discuss how these well posedness results combine with Theorems~\ref{thm:perturb} to yield new well posedness results.

\subsubsection{Perturbation of second order operators with real \texorpdfstring{$t$}{t}-independent coefficients}

In \cite{BarM16A}, the following well posedness results were established. 
\begin{thm}[{\cite[Section~9.3]{BarM16A}}]\label{thm:BarM16A}
Let $L=\Div \mat A\nabla$ be an elliptic operator of the form \eqref{eqn:divergence}, with $m=N=1$, acting on functions defined on open sets in~$\R^\dmn$, $\dmn\geq 2$. Suppose that $\mat A$ has real coefficients and is $t$-independent in the sense of formula~\eqref{eqn:t-independent}. Let $\Omega=\{(x',t):x'\in\R^\dmnMinusOne,\>t>\psi(x')\}$ for some Lipschitz function~$\psi$.

Then there is some $\kappa>0$ depending only on the dimension $\dmn$, the constants $\lambda$ and $\Lambda$ in formulas~\eqref{eqn:elliptic:bounded} and~\eqref{eqn:elliptic:everywhere}, and on $M=\doublebar{\nabla\psi}_{L^\infty(\R^\dmnMinusOne)}$, such that,
if 
\begin{equation}
\label{eqn:exponents:BarM16A}
0<\smooth<1,\quad 0<p\leq \infty, \quad \smooth-\kappa< \frac{1}{p}<\min\biggl(\smooth+\kappa, \frac{\dmn-2+\smooth+\kappa}{\dmnMinusOne}\biggr) \end{equation}
then the Dirichlet problem \eqref{eqn:Dirichlet:p:smooth} (with $m=1$) is well posed. %If $\dmn=2$ then the Neumann problem~\eqref{eqn:Neumann:p:smooth} is well posed.

If in addition $\mat A$ is symmetric, then the Dirichlet problem is well posed whenever
\begin{equation}\label{eqn:exponents:BarM16A:symmetric}
0< \smooth< 1,\quad 0<p\leq \infty, \quad \frac{\smooth-\kappa}{2}< \frac{1}{p}<\begin{cases}
\frac{1+\smooth+\kappa}{2}, & 0<\smooth<1-\kappa,\\
\frac{\dmn-2+\smooth+\kappa}{\dmnMinusOne}, & 1-\kappa\leq \smooth<1.
\end{cases}
\end{equation}
Furthermore, the Neumann problem~\eqref{eqn:Neumann:p:smooth}
is well posed for the same range of indices.

Finally, for any $p_0$, $\smooth_0$ and $p_1$, $\smooth_1$ satisfying the given conditions, these problems are compatibly well posed in the sense of Lemma~\ref{lem:interpolation}.%that if $\arr H\in L^{p_0,\smooth_0}_{av}(\Omega)\cap L^{p_1,\smooth_1}_{av}(\Omega)$ and either $\arr f\in \dot W\!A^{p_0}_{m-1,\smooth_0}(\Omega)\cap \dot W\!A^{p_1}_{m-1,\smooth_1}(\Omega)$ or $\arr g\in \dot N\!A^{p_0}_{m-1,\smooth_0-1}(\Omega)\cap \dot N\!A^{p_1}_{m-1,\smooth_1-1}(\Omega)$, for $p_j$, $\smooth_j$ satisfying the conditions given above, then there is a single solution $u\in \dot W^{p_0,\smooth_0}_{1,av}(\Omega)\cap \dot W^{p_1,\smooth_1}_{1,av}(\Omega)$ to the problem~\eqref{eqn:BarM16A:Dirichlet} or~\eqref{eqn:BarM16A:Neumann}.
\end{thm}

\begin{figure}
\begin{center}

\def\alph{0.35}
\def\enn{3}
\def\del{0.1}

\begin{tikzpicture}[scale=2]

% Gray background
%\fill [white!80!black] (0,0) -- (1,0) -- (1,1) -- (0,1) -- cycle;

% Region of well posedness
\fill [white!50!black] (0,0) -- (\alph,0) -- (1,1-\alph) -- (1,1+\alph/\enn) -- (1-\alph,1) -- (0,\alph) -- cycle;
\draw [thick,black] (0,0) -- (1,1);

% Axes
\draw [->] (-1/3,0)--(1.3,0) node [below] {$\smooth$};
\draw [->] (0,-1/3)--(0,1.3) node [left] {$1/p$};

%\node at (0,1) {$\circ$}; \node at (0,1) [left] {$1$};
\node at (1,1) {$\circ$}; \node at (1,1) [right] {$(1,1)$};

\end{tikzpicture}
\begin{tikzpicture}[scale=2]

% Gray background
%\fill [white!80!black] (0,0) -- (1,0) -- (1,1) -- (0,1) -- cycle;

% Region of well posedness
\fill [white!50!black] 
(0,0) -- (\alph,0) -- (1,1/2-\alph/2) -- (1,1+\alph/\enn) -- (1-\alph,1) -- (0,1/2+\alph/2) -- cycle;
\fill [black] (0,0) -- (1,1/2) -- (1,1) -- (0,1/2) -- cycle;

% Axes
\draw [->] (-1/3,0)--(1.3,0) node [below] {$\smooth$};
\draw [->] (0,-1/3)--(0,1.3) node [left] {$1/p$};

\node at (0,1/2) {$\circ$}; \node at (0,1/2) [left] {$1/2$};
\node at (1,1/2) {$\circ$}; \node at (1,1/2) [right] {$(1,1/2)$};
\node at (1,1) {$\circ$}; \node at (1,1) [right] {$(1,1)$};
%\node at (1-\alph,1) {$\circ$}; \node at (1-\alph,1) [above] {$(1-\kappa,1)\>\>\>\>$};
%\node at (1,1+\alph/\enn) {$\circ$}; \node at (1,1+\alph/\enn) [above right] {$(1,1+\kappa/\pdmnMinusOne)$};

\end{tikzpicture}
\end{center}
\caption{Theorem~\ref{thm:BarM16A}.
If $\mat A$ is real, $t$-independent, and if $L$ is a decoupled system of $N$ independent differential equations, then we have well posedness of the Dirichlet problem for all values of $(\smooth,1/p)$ shown. We have well posedness for all such $\mat A$ and all $(\smooth,1/p)$ in the black region; the size of the grey region depends on $\Omega$ and~$\mat A$. The right hand side indicates the region of well posedness if in addition $\mat A$ is symmetric or nearly symmetric; in this case the Neumann problem is well posed as well.
}\label{fig:BarM16A}
\end{figure}
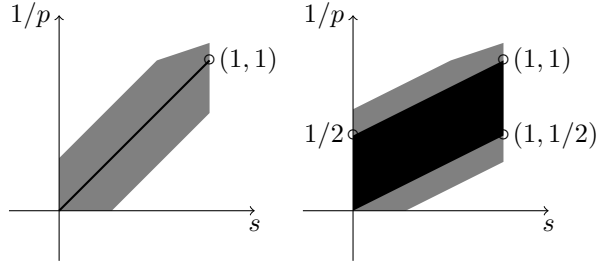

The acceptable values of $\smooth$ and $1/p$ are shown in Figure~\ref{fig:BarM16A}.

\begin{rmk}\label{rmk:BarM16A:2} If $\dmn=2$, a straightforward argument (see \cite{Pip97,KenR09,Bar13}) shows that the Dirichlet problem~\eqref{eqn:Dirichlet:p:smooth} and the Neumann problem~\eqref{eqn:Neumann:p:smooth} with $\mat A$ replaced by $(1/\det \mat A)\mat A^t$ are equivalent; thus, the Neumann problem is also well posed if $\dmn=2$ and $p$ and $\smooth$ satisfy the condition~\eqref{eqn:exponents:BarM16A}. 

If in addition $\mat A$ is symmetric, it follows from known results that  there is some $\kappa>0$ such that the Dirichlet and Neumann problems are well posed whenever
\begin{equation}\label{eqn:exponents:BarM16A:2}
0<\smooth<1,\quad 0<p\leq \infty, \quad -\frac{1}{2}-\kappa<\frac{1}{p}-\smooth<\frac{1}{2}+\kappa
.\end{equation}
See Section~\ref{sec:BarM16A:2}.
\end{rmk}

We may establish well posedness for more general second order systems using Theorem~\ref{thm:perturb}. 

\begin{thm} 
Fix some $\Lambda>\lambda>0$ and some $M>0$. Let $\smooth$ and $p$ satisfy the conditions~\eqref{eqn:exponents:BarM16A:symmetric} (if $\dmn\geq 3$) or~\eqref{eqn:exponents:BarM16A:2} (if $\dmn=2$).

Then there is some 
$\varepsilon>0$ depending on $\Lambda$, $\lambda$, $M$, $p$, $\smooth$, and the dimension $\dmn$ so that, if $\Omega$ is as in Theorem~\ref{thm:BarM16A}, if $L$ is an elliptic system of the form~\eqref{eqn:divergence} with $m=1$ associated to coefficients $\mat A$ that satisfy the ellipticity conditions~\eqref{eqn:elliptic:bounded} and~\eqref{eqn:elliptic:everywhere}, and if 
\begin{multline}
\label{eqn:BarM16A:perturb}
\smash{\sup_{{{j,\alpha,\beta,x',t}}} \abs{\im A^{jj}_{\alpha\beta}(x',t)}
+\sup_{{{j,\alpha,\beta,x',t,s}}} \abs{A^{jj}_{\alpha\beta}(x',t)-A^{jj}_{\alpha\beta}(x',s)}}
\\+\sup_{\substack{{j,k,\alpha,\beta,x',t}\\j\neq k}}  \abs{A^{jk}_{\alpha\beta}(x',t)}
<\varepsilon
\end{multline}
and
\begin{equation}
\label{eqn:BarM16A:perturb:symmetric}
\sup_{{{j,\alpha,\beta,x',t}}} \abs{A^{jj}_{\alpha\beta}(x',t)-A^{jj}_{\beta\alpha}(x',t)}
<\varepsilon\end{equation}
then the Dirichlet problem
\begin{equation}\label{eqn:Dirichlet:BarM16A}
\left\{\begin{aligned}L\vec u&=\Div \vec H\text{ in }\Omega,
\\ \vec u&=\vec f\text{ on }\partial\Omega,
\\ \doublebar{\vec u}_{\dot W^{p,\smooth}_{1,av}(\Omega)}&\leq C \doublebar{\vec f}_{\dot B^{p,p}_\smooth(\partial\Omega)}+C\doublebar{\vec H}_{L^{p,\smooth}_{av}(\Omega)}
\end{aligned}\right.\end{equation}
and the Neumann problem
\begin{equation}\label{eqn:Neumann:BarM16A}
\left\{\begin{aligned}L\vec u&=\Div \vec H\text{ in }\Omega,
\\ \nu\cdot \mat A\nabla \vec u=\mathop{\vec {\mathrm{M}}}\nolimits^\Omega_{\mat A}\vec u&=\vec g\text{ on }\partial\Omega,
\\ \doublebar{\vec u}_{\dot W^{p,\smooth}_{1,av}(\Omega)}&\leq C \doublebar{\vec g}_{\dot B^{p,p}_{\smooth-1}(\partial\Omega)} +C\doublebar{\vec H}_{L^{p,\smooth}_{av}(\Omega)}
\end{aligned}\right.\end{equation}
are well posed.

If $p$ and $\smooth$ satisfy the stronger condition~\eqref{eqn:exponents:BarM16A}, then
then the Dirichlet problem~\eqref{eqn:Dirichlet:BarM16A} is well posed even if the condition~\eqref{eqn:BarM16A:perturb:symmetric} is not satisfied.
\end{thm}

Observe that the size of the acceptable perturbation $\varepsilon$ depends on $p$ and~$\smooth$.
By perturbing at finitely many points and applying Lemmas~\ref{lem:interpolation} and~\ref{lem:BVP:duality}, we can  construct large regions in the $(\smooth,1/p)$-plane such that boundary value problems are well posed in the given regions.

\begin{thm} \label{thm:BarM16A:perturb}
Fix some $\Lambda>\lambda>0$ and some $M>0$. Let $\kappa$ be as in Theorem~\ref{thm:BarM16A}; notice that $0<\kappa\leq 1$ and that $\kappa$ depends only on~$\lambda$, $\Lambda$, $M$ and the dimension $\dmn$.
Fix some $\delta$ with $0<\delta<\kappa/2$. 

Then there is some $\varepsilon>0$ depending on $\Lambda$, $\lambda$, $M$, $\delta$, and the dimension $\dmn$ so that, if $\Omega$ is as in Theorem~\ref{thm:BarM16A}, if $L$ is an elliptic system of the form~\eqref{eqn:divergence} with $m=1$ associated to coefficients $\mat A$ that satisfy the ellipticity conditions~\eqref{eqn:elliptic:bounded} and~\eqref{eqn:elliptic:everywhere}, and if $\mat A$ satisfies the condition~\eqref{eqn:BarM16A:perturb},
then the Dirichlet problem
\eqref{eqn:Dirichlet:BarM16A}
is well posed whenever %$\delta\leq\smooth\leq1-\delta$ and $1/p=\smooth$.
\begin{equation*}\delta\leq\smooth\leq1-\delta, \quad \max(0,\smooth-\kappa+\delta)\leq \frac{1}{p}\leq\min\biggl(\smooth+\kappa-\delta, \frac{\dmn-2+\smooth+\kappa-\delta}{\dmnMinusOne}\biggr) .\end{equation*}

If in addition the condition~\eqref{eqn:BarM16A:perturb:symmetric} is valid,
then the Dirichlet problem~\eqref{eqn:Dirichlet:BarM16A} and the Neumann problem~\eqref{eqn:Neumann:BarM16A}
are well posed whenever $\dmn\geq 3$ and
\begin{equation*}\delta\leq\smooth\leq1-\delta,
\quad \max\biggl(0,\frac{\smooth}{2}-\kappa+\delta\biggr)\leq \frac{1}{p}\leq\min\biggl(\frac{1+\smooth}{2}+\kappa-\delta, \frac{\dmn-2+\smooth+\kappa-\delta}{\dmnMinusOne}\biggr)\end{equation*}
or $\dmn=2$ and
\begin{equation*}
\delta\leq \smooth\leq 1-\delta,\quad 0<p\leq \infty,\quad -\frac{1}{2}-(\kappa-\delta)\leq \frac{1}{p}-\smooth\leq \frac{1}{2}+\kappa-\delta
.\end{equation*}
\end{thm}

\begin{figure}
\begin{center}

\def\alph{0.35}
\def\enn{3}
\def\del{0.1}

\begin{tikzpicture}[scale=2]

% Gray background
%\fill [white!80!black] (0,0) -- (1,0) -- (1,1) -- (0,1) -- cycle;

% Region of well posedness
\fill [white!80!black] (0,0) -- (\alph,0) -- (1,1-\alph) -- (1,1+\alph/\enn) -- (1-\alph,1) -- (0,\alph) -- cycle;
\fill [white!50!black] (\del,0) -- (\alph-\del,0) -- (1-\del,1-\alph) -- (1-\del,1+\alph/\enn-2*\del/\enn) -- (1-\alph+\del,1) -- (\del,\alph);
\draw [thick,black] (\del,\del) -- (1-\del,1-\del);

% Axes
\draw [->] (-1/3,0)--(1.3,0) node [below] {$\smooth$};
\draw [->] (0,-1/3)--(0,1.3) node [left] {$1/p$};

%\node at (0,1) {$\circ$}; \node at (0,1) [left] {$1$};
%\node at (1,0) {$\circ$}; \node at (1,0) [below] {$1$};

\end{tikzpicture}
\begin{tikzpicture}[scale=2]

% Gray background
%\fill [white!80!black] (0,0) -- (1,0) -- (1,1) -- (0,1) -- cycle;

% Region of well posedness
\fill [white!80!black] (0,0) -- (\alph,0) -- (1,1/2-\alph/2) -- (1,1+\alph/\enn) -- (1-\alph,1) -- (0,1/2+\alph/2) -- cycle;
\fill [white!50!black] (\del,0) -- (\alph-\del,0) -- (1-\del,1/2-\alph/2) -- (1-\del,1+\alph/\enn-2*\del/\enn) -- (1-\alph+\del,1) -- (\del,1/2+\alph/2);
\fill [black] (\del,\del/2) -- (1-\del,1/2-\del/2) -- (1-\del,1-\del/2) -- (\del,1/2+\del/2) -- cycle;

% Axes
\draw [->] (-1/3,0)--(1.3,0) node [below] {$\smooth$};
\draw [->] (0,-1/3)--(0,1.3) node [left] {$1/p$};

%\node at (0,1) {$\circ$}; \node at (0,1) [left] {$1$};
%\node at (1,0) {$\circ$}; \node at (1,0) [below] {$1$};

\end{tikzpicture}

\end{center}
\caption{Theorem~\ref{thm:BarM16A:perturb}.
If $\mat A$ is near to satisfying the conditions of Theorem~\ref{thm:BarM16A}, then we have well posedness of the Dirichlet problem for all indicated values of $(\smooth,1/p)$.
}\label{fig:BarM16A:perturb}
\end{figure}
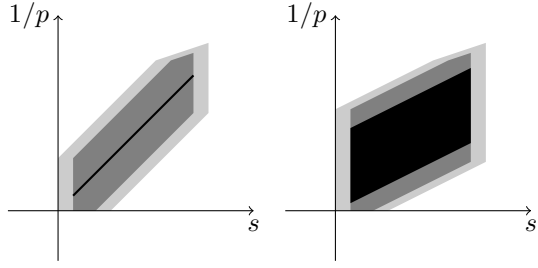

The acceptable values of $p$ and $\smooth$ are shown in Figure~\ref{fig:BarM16A:perturb}.

\subsubsection{Perturbations of the biharmonic operator}

In \cite{MitM13B}, I.~Mit\-rea and M.~Mit\-rea established well posedness of boundary value problems for the biharmonic operator $\Delta^2$ in bounded Lipschitz domains. Their results may be shown to imply the following. See Section~\ref{sec:biharmonic}. 

\begin{thm}[{\cite{MitM13B}}]\label{thm:biharmonic:introduction}
Let $\Omega\subset\R^\dmn$ be a bounded Lipschitz domain with connected boundary. Let $-1/\pdmnMinusOne<\rho<1$. Then there is some $\kappa>0$ depending on~$\Omega$ and $\rho$ such that if $\dmn\geq 4$ and
% $0<\smooth<1$, $0<p\leq\infty$ and
\begin{align}\label{eqn:biharmonic:stripe}
0<\smooth<1,\quad 1< p< \infty, \quad
\frac{1}{2}-\frac{1}{\dmnMinusOne}-\kappa<\frac{1}{p}-\frac{\smooth}{\dmnMinusOne} < \frac{1}{2}+\kappa,
\end{align}
or if $\dmn=2$ or $\dmn=3$ and
\begin{align}
\label{eqn:biharmonic:stripe:2}
0<\smooth<1,\quad 1< p< \infty,\quad
0<\frac{1}{p}-\biggl(\frac{1-\kappa}{2}\biggr) \smooth < \frac{1+\kappa}{2},
\end{align}
then the biharmonic Dirichlet problem
\begin{equation}\label{eqn:Dirichlet:biharmonic}\left\{\begin{aligned} \Delta^2 u &= \Div_2\arr H \quad\text{in }\Omega,\\
\nabla u &= \arr f \quad\text{on }\partial\Omega,\\
\doublebar{u}_{\dot W^{p,\smooth}_{2,av}(\Omega)} 
& \leq
C \doublebar{\arr f}_{\dot W\!A^p_{1,\smooth}(\partial\Omega)}
+C \doublebar{\arr H}_{L^{p,\smooth}_{av}(\Omega)}
\end{aligned}\right.
\end{equation}
and the biharmonic Neumann problem
\begin{equation}\label{eqn:Neumann:biharmonic}\left\{\begin{aligned} \Delta^2 u &= \Div_2 \arr H \quad\text{in }\Omega,\\
\M_{\mat A_\rho,\arr H}^\Omega u &= \arr g ,\\
\doublebar{u}_{\dot W^{p,\smooth}_{2,av}(\Omega)} 
& \leq
C \doublebar{\arr g}_{\dot N\!A^p_{1,\smooth-1}(\partial\Omega)}
+C \doublebar{\arr H}_{L^{p,\smooth}_{av}(\Omega)}
\end{aligned}\right.
\end{equation}
are both well posed.
\end{thm}

\begin{figure}
\begin{center}

\def\alph{0.15}
\def\enn{3}

\begin{tikzpicture}[scale=2]

% Gray background
%\fill [white!80!black] (0,0) -- (1,0) -- (1,1) -- (0,1) -- cycle;

% Region of well posedness
\fill [white!50!black] (0,1/2-1/\enn-\alph) -- (1,1/2-\alph)--(1,1/2+1/\enn+\alph) -- (0,1/2+\alph) -- cycle;
\fill [black] (0,1/2-1/\enn) -- (1,1/2)--(1,1/2+1/\enn) -- (0,1/2) -- cycle;

% Axes
\draw [->] (-1/3,0)--(1.3,0) node [below] {$\smooth$};
\draw [->] (0,-1/3)--(0,1.3) node [left] {$1/p$};

\node at (0,1/2) {$\circ$}; \node at (0,1/2) [left] {$\frac12$};
\node at (1,1/2) {$\circ$}; \node at (1,1/2) [right] {$(1,\frac12)$};
\node at (1,1/2+1/\enn) {$\circ$}; \node at (1,1/2+1/\enn) [right] {$(1,\frac12+\frac1\dmnMinusOne)$};
\node at (0,1/2-1/\enn) {$\circ$}; \node at (0,1/2-1/\enn) [left] {$\frac12-\frac1\dmnMinusOne$};

\end{tikzpicture}
\def\alph{0.25}
\begin{tikzpicture}[scale=2]

% Gray background
%\fill [white!80!black] (0,0) -- (1,0) -- (1,1) -- (0,1) -- cycle;

% Region of well posedness
\fill [white!50!black] (0,0) -- (1,1/2-\alph)--(1,1) -- (0,1/2+\alph) -- cycle;
\fill [black] (0,0) -- (1,1/2)--(1,1) -- (0,1/2) -- cycle;

% Axes
\draw [->] (-1/3,0)--(1.3,0) node [below] {$\smooth$};
\draw [->] (0,-1/3)--(0,1.3) node [left] {$1/p$};

\node at (0,1/2) {$\circ$}; \node at (0,1/2) [left] {$\frac12$};
\node at (1,1) {$\circ$}; \node at (1,1) [right] {$(1,1)$};
\node at (1,1/2) {$\circ$}; \node at (1,1/2) [right] {$(1,\frac12)$};

\end{tikzpicture}
\end{center}
\caption{Theorem~\ref{thm:biharmonic:introduction}.
The Dirichlet and Neumann problems~\eqref{eqn:Dirichlet:biharmonic} and~\eqref{eqn:Neumann:biharmonic} are well posed provided $(\smooth,1/p)$ lie in the region shown on the left (for $\dmn\geq4$) or on the right (for $\dmn=2$ or $\dmn=3$).
}\label{fig:MitM13B}
\end{figure}
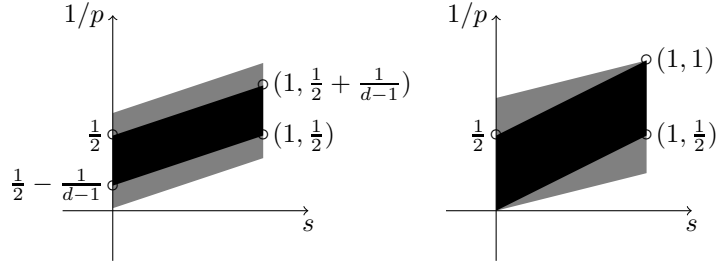

The acceptable values of $\smooth$ and $1/p$ are shown in Figure~\ref{fig:MitM13B}. 
Here
$\mat A_\rho$ is the symmetric constant coefficient matrix such that 
\begin{equation}
\label{eqn:biharmonic:coefficients}
\langle \nabla^2 \psi(x),\mat A_\rho \nabla^2 \varphi(x)\rangle = 
\rho\, \overline{\Delta \psi(x)}\,\Delta\varphi(x)
+ (1-\rho) \sum_{j=1}^\dmn \sum_{k=1}^\dmn \overline{\partial_j \partial_k \psi(x)}\,\partial_j\partial_k \varphi(x).
\end{equation}

We remark that \cite{MitMW11} contained some well posedness results in the case $p\leq 1$ for the Dirichlet problem  if $\dmn=3$.  In a forthcoming paper, we intend to consider the case $p\leq 1$ extensively; we will apply comparable results therein.

Using Theorem~\ref{thm:perturb}, we may derive new well posedness results for operators whose coefficients are close to those of the biharmonic operator.

\begin{thm}\label{thm:biharmonic:perturb}
Let $N\geq 1$ be an integer, and for each $1\leq j\leq N$, let $\rho_j\in \R$; in the case of the Neumann problem we additionally require $-1/\pdmnMinusOne<\rho_j<1$. Let $\Omega\subset\R^\dmn$ be a bounded simply connected Lipschitz domain,  and let $\kappa_j$ be as in Theorem~\ref{thm:biharmonic:introduction}. Let $\kappa=\min_j\kappa_j$. 
Let $0<\delta<\kappa$.

Let $L$ be an operator of the form~\eqref{eqn:divergence}, with $m=2$, associated to coefficients~$\mat A$.
Then there is some $\varepsilon>0$ such that, if $L$ is an operator of the form~\eqref{eqn:divergence} with $m=2$, and if
\begin{equation*}\sup_{j,\alpha,\beta,x}\abs{A^{jj}_{\alpha\beta}(x)-(A_{\rho_j})_{\alpha\beta}} +\sup_{\substack{{j,k,\alpha,\beta,x}\\j\neq k}}\abs{A^{jk}_{\alpha\beta}(x)}<\varepsilon\end{equation*}
then the Dirichlet problem~\eqref{eqn:Dirichlet:p:smooth}
and the Neumann problem~\eqref{eqn:Neumann:p:smooth}, with $m=2$,
are well posed whenever $\delta\leq\smooth\leq 1-\delta$, $1/(1-\delta)\leq p\leq 1/\delta$ and
\begin{gather*}
\dmn\geq 4 \quad\text{and}\quad\frac{1}{2}-\frac{1}{\dmnMinusOne}-(\kappa-\delta) \leq \frac{1}{p}-\frac{\smooth}{\dmnMinusOne} \leq \frac{1}{2}+(\kappa-\delta),
\\
\dmn=2\text{ or }\dmn=3\text{ and }
0\leq \frac{1}{p}-\biggl(\frac{1-(\kappa-\delta)}{2}\biggr) \smooth \leq  \frac{1+(\kappa-\delta)}{2}
.\end{gather*}
\end{thm}

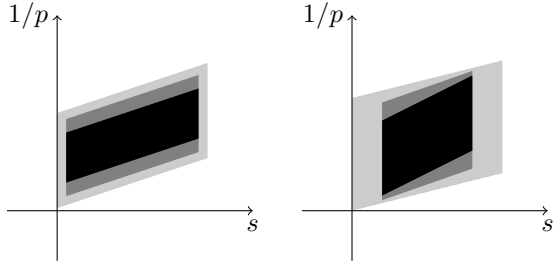
\begin{figure}
\begin{center}

\def\alph{0.15}
\def\del{0.06}
\def\enn{3}

\begin{tikzpicture}[scale=2]

% Gray background
%\fill [white!80!black] (0,0) -- (1,0) -- (1,1) -- (0,1) -- cycle;
\fill [white!80!black] (0,1/2-1/\enn-\alph) -- (1,1/2-\alph)--(1,1/2+1/\enn+\alph) -- (0,1/2+\alph) -- cycle;

% Region of well posedness
\fill [white!50!black] 
(\del,1/2-1/\enn-\alph+\del+\del/\enn) -- 
(1-\del,1/2-\alph+\del-\del/\enn)--
(1-\del,1/2+1/\enn-\del/\enn+\alph-\del) -- 
(\del,1/2+\del/\enn+\alph-\del) -- cycle;

\fill [black] 
(\del,1/2-1/\enn+\del/\enn) -- 
(1-\del,1/2-\del/\enn)--
(1-\del,1/2+1/\enn-\del/\enn) -- 
(\del,1/2+\del/\enn) -- cycle;

% Axes
\draw [->] (-1/3,0)--(1.3,0) node [below] {$\smooth$};
\draw [->] (0,-1/3)--(0,1.3) node [left] {$1/p$};

%\node at (0,1) {$\circ$}; \node at (0,1) [left] {$1$};
%\node at (1,0) {$\circ$}; \node at (1,0) [below] {$1$};

\end{tikzpicture}
\def\alph{0.25}
\def\del{0.2}
\begin{tikzpicture}[scale=2]

% Gray background
\fill [white!80!black] (0,0) -- (1,1/2-\alph)--(1,1) -- (0,1/2+\alph) -- cycle;
%\fill [white!80!black] (0,0) -- (1,1/2)--(1,1) -- (0,1/2) -- cycle;

% Region of well posedness
\fill [white!50!black] 
(\del,\del/2-\alph*\del+\del*\del/2)
--(1-\del,1/2-\alph+\alph*\del-\del*\del/2)
--(1-\del,1-\del/2+\alph*\del-\del*\del/2)
--(\del,1/2+\alph-\alph*\del+\del*\del/2)
--cycle
;
\fill [black] 
(\del,\del/2)
--(1-\del,1/2-\del/2)
--(1-\del,1-\del/2)
--(\del,1/2+\del/2)
--cycle;

% Axes
\draw [->] (-1/3,0)--(1.3,0) node [below] {$\smooth$};
\draw [->] (0,-1/3)--(0,1.3) node [left] {$1/p$};

%\node at (0,1) {$\circ$}; \node at (0,1) [left] {$1$};
%\node at (1,0) {$\circ$}; \node at (1,0) [below] {$1$};

\end{tikzpicture}
\end{center}
\caption{Theorem~\ref{thm:biharmonic:perturb}. If $L$ is close to the biharmonic operator, then the Dirichlet and Neumann problems~\eqref{eqn:Dirichlet:biharmonic} and~\eqref{eqn:Neumann:biharmonic} are well posed provided $(\smooth,1/p)$ lie in the region shown on the left (for $\dmn\geq4$) or on the right (for $\dmn=2$ or $\dmn=3$).
}\label{fig:MitM13B:perturb}
\end{figure}

The acceptable values of $\smooth$ and $1/p$ are shown in Figure~\ref{fig:MitM13B:perturb}.

\subsection{Outline of the paper}

The paper is organized as follows. In Section~\ref{sec:dfn} we will define our terminology. In Section~\ref{sec:L} we will discuss some properties of the function spaces $L^{p,\smooth}_{av}(\Omega)$ and $\dot W^{p,\smooth}_{m,av}(\Omega)$; many of these properties were established in \cite{Bar16pB} but some are new to the present paper. In Sections~\ref{sec:perturb} and \ref{sec:neighborhood} we will prove Theorems~\ref{thm:perturb} and \ref{thm:neighborhood}, respectively. Finally, in Section~\ref{sec:A0} we will resolve some differences between well posedness results as stated in the literature, and the well posedness results required by Theorem~\ref{thm:perturb}; the results of Section~\ref{sec:A0} were used in Section~\ref{sec:A0:introduction} above.

\subsection*{Acknowledgements}

The author would like to thank Steve Hofmann and Svitlana Mayboroda for many useful discussions concerning the theory of higher-order elliptic equations. The author would also like to thank the American Institute of Mathematics for hosting the SQuaRE workshop on ``Singular integral operators and solvability of boundary problems for elliptic equations with rough coefficients,'' at which many of these discussions occurred and the Mathematical Sciences Research Institute for hosting a Program on Harmonic Analysis, during which substantial parts of this paper were written.

\section{Definitions}
\label{sec:dfn}

Throughout we work with a divergence-form elliptic system of $N$ partial differential equations of order~$2m$ in dimension $\dmn$.
The notation of multiindices, function spaces, and Lipschitz domains used in this paper will be that of \cite[Section~2]{Bar16pB}.

%If $\arr H$ is an array of functions defined in~$\Omega$ and indexed by multiindices~$\alpha$ with $\abs{\alpha}=m$, then $\Div_m\arr H$ is the distribution given by
%\begin{equation}
%\label{eqn:Div}
%\bigl\langle \vec\varphi, \Div_m\arr H\bigr\rangle_\Omega
%=
%(-1)^m\bigl\langle\nabla^m\vec\varphi, \arr H\bigr\rangle_\Omega
%\end{equation}
%for all smooth test functions~$\vec\varphi$ supported in~$\Omega$.
%In particular, if the right-hand side is zero for all $\vec\varphi$ then we say that $\Div_m\arr H=0$.

If $U\subset\R^\dmn$ is a measurable set, we let $\1_U$ be the characteristic function of~$U$. If $F$ is a function defined on~$U$, we let $\mathcal{E}^0_U F$ be the extension of~$F$ to $\R^\dmn$ by zero, that is,
\begin{equation*}\mathcal{E}^0_U F(x)= \begin{cases} F(x), & x\in U, \\ 0, &x\notin U.\end{cases}\end{equation*}
If $F$ is defined on some $V\supsetneq U$, we let $F\big\vert_U$ be the restriction of~$F$ to~$U$.

We now introduce some notation and standard bounds for elliptic operators.
Let $\mat A = \bigl( A^{jk}_{\alpha\beta} \bigr)$ be measurable coefficients defined on $\R^\dmn$, indexed by integers $1\leq j\leq N$, $1\leq k\leq N$ and by multtiindices $\alpha$, $\beta$ with $\abs{\alpha}=\abs{\beta}=m$. If $\arr H$ is an array, then $\mat A\arr H$ is the array given by 
\begin{equation*}(\mat A\arr H)_{j,\alpha} = \sum_{k=1}^N
\sum_{\abs{\beta}=m}
A^{jk}_{\alpha\beta} H_{k,\beta}.\end{equation*}

Throughout we consider coefficients that satisfy the bound
%\begin{align}
%\label{eqn:elliptic:bounded}
%\abs[big]{\bigl\langle\nabla^m\vec g,\mat A \nabla^m\vec f\bigr\rangle_{\R^\dmn}}
%&\leq 
%	\Lambda \doublebar{\nabla^m\vec f}_{L^2(\R^\dmn)} \doublebar{\nabla^m\vec g}_{L^2(\R^\dmn)}
%\end{align}
%for some $\Lambda>\lambda>0$, and for all $\vec f$, $\vec g$, $\vec \varphi\in \dot W^2_m(\R^\dmn)$. %Occasionally we will need to replace the bound \eqref{eqn:elliptic:bounded} with the stronger bound
\begin{align}
\label{eqn:elliptic:bounded}
\doublebar{\mat A}_{L^\infty(\R^\dmn)}
&\leq 
	\Lambda
\end{align}
and the G\r{a}rding inequality
\begin{align}
\label{eqn:elliptic:everywhere}
\re {\bigl\langle\nabla^m\vec \varphi,\mat A \nabla^m\vec \varphi\bigr\rangle_{\R^\dmn}} 
&\geq 
	\lambda\doublebar{\nabla^m\vec\varphi}_{L^2(\R^\dmn)}^2
	\quad\text{for all $\vec\varphi\in\dot W^2_m(\R^\dmn)$}
\end{align}
for some $\Lambda>\lambda>0$. When studying the Neumann problem in a domain~$\Omega$, we will often require that $\mat A$ satisfy the local G\r{a}rding inequality
\begin{align}
\label{eqn:elliptic:domain}
\re {\bigl\langle\nabla^m\vec \varphi,\mat A \nabla^m\vec \varphi\bigr\rangle_\Omega} 
&\geq 
	\lambda\doublebar{\nabla^m\vec\varphi}_{L^2(\Omega)}^2
	\quad\text{for all $\vec\varphi\in\dot W^2_m(\Omega)$}
.\end{align}

Here the inner product $\bigl\langle \,\cdot\, , \,\cdot\,\bigr\rangle$ is given by
\begin{equation*}
\bigl\langle \arr F,\arr G\bigr\rangle = 
\sum_{\abs{\gamma}=m}
\overline{F_{\gamma}}\, G_{\gamma},
\>\>\>
\bigl\langle \arr F,\arr G\bigr\rangle_\Omega = 
\sum_{\abs{\gamma}=m}
\int_{\Omega} \overline{F_{\gamma}}\, G_{\gamma},
\>\>\>
\bigl\langle \arr F,\arr G\bigr\rangle_{\partial\Omega} = 
\sum_{\abs{\gamma}=m}
\int_{\partial\Omega} \overline{F_{\gamma}}\, G_{\gamma}\,d\sigma\end{equation*}
where $\sigma$ denotes surface measure. (In this paper we will consider  only domains  with rectifiable boundary.)
The norm $\abs{\mat A}$ of $\mat A$ in formula~\eqref{eqn:elliptic:bounded} is the operator norm, that is, $\abs{\mat A}=\sup_{{\arr H}\neq 0} \abs{\mat A\arr H}/\abs{\arr H}$, where $\abs{\arr H}^2=\langle \arr H,\arr H\rangle$.

We let $L$ be the operator of the form~\eqref{eqn:divergence} associated with the coefficients~$\mat A$. %; notice that the weak formulation~\eqref{eqn:Div} of the divergence implies that $L\vec u$ is given by formula~\eqref{eqn:weak}. %; that is, we say that $L\vec u=\Div_m \arr H$ in~$\Omega$ if, for every $\vec\varphi$ smooth and compactly supported in~$\Omega$, we have that
%%\begin{equation}
%%\label{eqn:L}
%%\bigl\langle\nabla^m\vec\varphi, \mat A \nabla^m\vec u\bigr\rangle_\Omega
%%=
%%\bigl\langle\nabla^m\vec\varphi, \arr H\bigr\rangle_\Omega,
%%\end{equation}
%%that is, we have that
%\begin{equation*}
%\sum_{j=1}^N \sum_{k=1}^N
%\sum_{\abs{\alpha}=\abs{\beta}=m}
%\int_{\Omega}\partial^\alpha \bar \varphi_j\, A^{jk}_{\alpha\beta}\,\partial^\beta u_k
%=
%\sum_{j=1}^N
%\sum_{\abs{\alpha}=m}
%\int_{\Omega} \partial^\alpha \bar \varphi_j \, H_{j,\alpha}
%.
%\end{equation*}
%In particular, if the left-hand side is zero for all such~$\vec\varphi$ then we say that $L\vec u=0$.

Throughout the paper we will let $C$ denote a constant whose value may change from line to line, but which depends only on the dimension $\dmn$, the ellipticity constants $\lambda$ and $\Lambda$ in the bounds~\eqref{eqn:elliptic:bounded}, \eqref{eqn:elliptic:everywhere} and \eqref{eqn:elliptic:domain}, and the Lipschitz character $(M, n, c_0)$ of any relevant domains. Any other dependencies will be indicated explicitly. We say that $ A\approx B$ if, for some such constant $C$, $ A\leq CB$ and $B\leq CA$.

\section{Properties of function spaces}
\label{sec:L}

In order to investigate boundary value problems with solutions in the spaces $\dot W^{p,\smooth}_{m,av}(\Omega)$, we will need a number of properties of the spaces $\dot W^{p,\smooth}_{m,av}(\Omega)$ and $L^{p,\smooth}_{av}(\Omega)$. 

%The main results of the author's paper \cite{Bar16pB} are trace and extension theorems for the spaces $\dot W^{p,\smooth}_{m,av}(\Omega)$ and $L^{p,\smooth}_{av}(\Omega)$. 
%The paper \cite{Bar16pB} begins with some preliminary results, many of which we will also need; for the reader's convenience, we will now summarize the important results of \cite[Section~3]{Bar16pB}.

Let $\Omega$ be an open set in $\R^\dmn$, and let $\mathcal{G}$ be a grid of Whitney cubes; then $\Omega=\cup_{Q\in\mathcal{G}} Q$, the cubes in~$\mathcal{G}$ have pairwise-disjoint interiors, and if $Q\in\mathcal{G}$ then the side-length $\ell(Q)$ of~$Q$ satisfies $\ell(Q)\approx\dist(Q,\partial\Omega)$.
If $0<p<\infty$ and $\arr H\in L^{p,\smooth}_{av}(\Omega)$, then
\begin{equation}
\label{eqn:L:norm:whitney}
\doublebar{\arr H}_{L^{p,\smooth}_{av}(\Omega)}
\approx \biggl( \sum_{Q\in\mathcal{G}} \biggl(\fint_Q \abs{\arr H}^2\biggr)^{p/2} \ell(Q)^{\dmnMinusOne+p-p\smooth}\biggr)^{1/p}
\end{equation}
where the comparability constants depend on~$p$, $\smooth$, and the comparability constants for Whitney cubes in the relation $\ell(Q)\approx\dist(Q,\partial\Omega)$. (This equivalence is still valid in the case $p=\infty$ if we replace the sum over cubes by an appropriate supremum.) This implies that we may replace the balls $B(x,\dist(x,\partial\Omega)/2)$ in the definition~\eqref{eqn:L:norm:2} of $L^{p,\smooth}_{av}(\Omega)$ by balls $B(x,a\dist(x,\partial\Omega))$ for any $0<a<1$, and produce an equivalent norm.

We have two important consequences. 
First, 
\begin{equation}
\label{eqn:L2:2:half}
L^{2,1/2}_{av}(\Omega)=L^2(\Omega)\quad\text{ and so }\quad\dot W^{2,1/2}_{m,av}(\Omega)=\dot W^2_m(\Omega).\end{equation}
Second, if $1\leq p < \infty$, then we have the duality relation
\begin{equation}\label{eqn:L:dual}
(L^{p,\smooth}_{av}(\Omega))^*=L^{p',1-\smooth}_{av}(\Omega)
\end{equation}
where $1/p+1/p'=1$. (As usual $L^{1,1-\smooth}_{av}(\Omega)$ is not the dual space to $L^{\infty,\smooth}_{av}(\Omega)$.) 

We have the following result showing that $L^{p,\smooth}_{av}(\Omega)$-functions are locally integrable up to the boundary.
\begin{lem}[{\cite[Lemma~3.7]{Bar16pB}}]\label{lem:L:L1}
%Let $\smooth\in\R$ and let $0<p\leq\infty$, $1\leq q\leq 2$. Suppose that either $0<p\leq q$ and $1/q> (\dmnMinusOne+p-p\smooth)/\pdmn p$, or that $0<q\leq p$ and $1/q>1-\smooth$. (In particular, if $q=1$ and $0<\smooth<1$ then then these conditions are valid whenever $\pmin<p\leq\infty$.)
Suppose that $\Omega\subset\R^\dmn$ is a Lipschitz domain, and that
$\smooth>0$ and $\pmin<p\leq \infty$. 
If $\arr H\in L_{av}^{p,\smooth}(\Omega)$, if $x_0\in\partial\Omega$, and if $R>0$, then
\begin{equation}%\label{eqn:L:L1:1}
\doublebar{\arr H}_{L^1(B(x_0,R)\cap\Omega)} 
\leq C\doublebar{\arr H}_{L_{av}^{p,\smooth}(\Omega)} R^{\dmnMinusOne+\smooth-\pdmnMinusOne/p} 
.\end{equation}

%Furthermore, let $V=\{(x',t):t>\psi(x')\}$ be a Lipschitz graph domain. If $Q\subset \R^\dmnMinusOne$ is a cube, let 
%\begin{align*}
%T(Q)&=\{(x',t):x'\in Q,\>\psi(x')<t<\psi(x')+c_0\ell(Q)\}
%\end{align*}
%for some constant $c_0>0$.
%Let $\mathcal{D}$ be the grid of dyadic cubes in $\R^\dmnMinusOne$; that is,
%\begin{equation*}\mathcal{D} = \{(2^k n_1,2^k n_1+2^k )\times(2^k n_2,2^k n_2+2^k )\times\dots\times (2^k n_\dmn,2^k n_\dmn+2^k ):n_j,k\in\Z\}.\end{equation*}
%Then
%\begin{equation}%\label{eqn:L:L1:2}
%\sum_{Q\in\mathcal{D}} \biggl(\int_{T(Q)} \abs{\arr H}^q\biggr)^{p/q} \ell(Q)^{\pdmn/q-1-p\smooth-p\pdmnMinusOne}
%\approx
%\doublebar{\arr H}_{L_{av}^{p,\smooth,q}(V)}^p
%\end{equation}
%where
%\begin{equation}%\label{eqn:L:average:1}
%\doublebar{\arr H}_{L^{p,\smooth,q}_{av}(\Omega)}  = \int_\Omega \biggl(\fint_{B(x,\dist(x,\partial\Omega)/2)} \abs{\arr H}^q\biggr)^{p/q}\dist(x,\partial\Omega)^{p-1-p\smooth}\,dx
%.\end{equation}
\end{lem}
%Observe that $\doublebar{\arr H}_{L^{p,\smooth}_{av}(\Omega)}= \doublebar{\arr H}_{L^{p,\smooth,2}_{av}(\Omega)}$.

Conversely, bounded compactly supported functions are contained in $L^{p,\smooth}_{av}(\Omega)$.
\begin{lem}\label{lem:L:L-infinity}
Suppose that $\Omega\subset\R^\dmn$ is a Lipschitz domain, and that 
$\smooth<1$ and $0<p\leq\infty$. 
If $\arr H\in L^\infty(\Omega)$, if $x_0\in\partial\Omega$, and if $R>0$, then $\1_{B(x_0,R)}\arr H$ is in ${L^{p,\smooth}_{av}(\Omega)}$, with
\begin{equation*}%\label{eqn:L:L1:1}
\doublebar{\1_{B(x_0,R)}\arr H}_{L^{p,\smooth}_{av}(\Omega)} 
\leq C\doublebar{\arr H}_{L^\infty(\Omega)} R^{1-\smooth+\pdmnMinusOne/p} 
.\end{equation*}
\end{lem}
\begin{proof}
By the definition~\eqref{eqn:L:norm:2} of ${L^{p,\smooth}_{av}(\Omega)}$,
\[\doublebar{\1_{B(x_0,R)}\arr H}_{L^{p,\smooth}_{av}(\Omega)} 
\leq
\doublebar{\arr H}_{L^\infty(\Omega)}\biggl(\int_{\Omega\cap B(x_0,(3/2)R)} \dist(x,\partial\Omega)^{p-1-p\smooth}\,dx\biggr)^{1/p}
.\]
If $\Omega$ is a Lipschitz domain, then the integral clearly converges, and so the proof is complete.
%If $\Omega=\R^\dmn_+$, this may be verified immediately by using a dyadic Whitney decomposition and applying the norm~\eqref{eqn:L:norm:whitney}. If $\Omega$ is a Lipschitz graph domain, the result follows by using a change of variables. 
%
%There remains the case where $\Omega$ is a domain with compact boundary. 
%We may control the ${L^{p,\smooth}_{av}(\Omega)}$ norm of~$\arr H$ near $\partial\Omega$ using the bound for Lipschitz graph domains. If $R$ is sufficiently small (compared with the natural length scale $r=r_\Omega$ of Definition~\ref{dfn:domain}), this completes the proof. 
%
%If $R>r_\Omega/C$, then we may control the ${L^{p,\smooth}_{av}(\Omega)}$ norm of~$\arr H$ far from $\partial\Omega$ by using the norm~\eqref{eqn:L:norm:whitney} and the observation that there are at most $C(1+r_\Omega/2^j)^\dmn$ dyadic Whitney cubes of side-length $2^j$. 
\end{proof}

\subsection{Embedding results, compatibility and uniqueness}

In this section we will show that $L^{r,\omega}_{av}(\Omega)\subset L^{q,\sigma}_{av}(\Omega)$ for appropriate values of  $q$, $r$, $\sigma$, $\omega$ and~$\Omega$. Furthermore, we will state some useful consequences of this embedding result.

\begin{lem}\label{lem:embedding} Let $\Omega$ be an open set with $\Omega\subset\R^\dmn$. Let $0<\sigma<\omega<1$ and let $\pmin[\sigma]<q\leq \infty$.

Suppose that $ r$ is such that one of the following conditions is true.
\begin{itemize}
\item $\frac{\dmnMinusOne}{q}-{\sigma} = \frac{\dmnMinusOne}{r}-{\omega}$.
\item $\Omega$ is bounded, and 
$0\leq \frac{\dmnMinusOne}{ r} \leq \frac{\dmnMinusOne}{q}+{ \omega-\sigma}$.
\end{itemize}

Then $L^{r,\omega}_{av}(\Omega)\subsetneq L^{q,\sigma}_{av}(\Omega)$, and if $\arr\Psi\in L^{r,\omega}_{av}(\Omega)$ then 
\begin{equation*}\doublebar{\arr\Psi}_{L^{q,\sigma}_{av}(\Omega)}\leq C(r,q,\omega,\sigma) \diam \Omega^{\pdmnMinusOne/q-\pdmnMinusOne/r+\omega-\sigma}\doublebar{\arr\Psi}_{L^{r,\omega}_{av}(\Omega)}.\end{equation*}
\end{lem}
If $\Omega$ is unbounded, then the indeterminate form $\diam \Omega^{\pdmnMinusOne/q-\pdmnMinusOne/r+\omega-\sigma}=\infty^0$ in the above expression is taken to be~$1$.

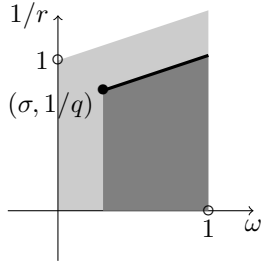
\begin{figure}
\begin{center}
\begin{tikzpicture}[scale=2]

\def\eps{0.3}
\def\enn{3}
\def\P{0.8}
\def\Th{0.3}

% Gray background
\fill [white!80!black] (0,0) -- (1,0) -- (1,1+1/\enn) -- (0,1) -- cycle;

% Axes
\draw [->] (-1/3,0)--(1.3,0) node [below] {$\omega$};
\draw [->] (0,-1/3)--(0,1.3) node [left] {$1/r$};

\fill [white!50!black] (\Th,0) -- (1,0) -- (1, \P + 1/\enn-\Th/\enn) -- (\Th,\P) -- cycle;

\draw [very thick] (\Th,\P) -- (1, \P + 1/\enn-\Th/\enn);

\node at (\Th,\P) {$\bullet$}; \node [left] at (\Th,\P-0.1) {$(\sigma,1/q)$};

\node at (0,1) {$\circ$}; \node at (0,1) [left] {$1$};
\node at (1,0) {$\circ$}; \node at (1,0) [below] {$1$};
%\node at (1, \P + 1/\enn-\Th/\enn) [above] {slope $1/\pdmnMinusOne\qquad$};

\end{tikzpicture}
\end{center}
\caption{For a given value of $(\sigma,1/q)$, the acceptable values of $(1/\omega,1/ r)$ for Lemma~\ref{lem:embedding}. The black line has slope $1/\pdmnMinusOne$.}
\end{figure}

\begin{proof}[Proof of Lemma~\ref{lem:embedding}]
Begin with the case where $0<q< r$. In this case $\pdmnMinusOne/r< \pdmnMinusOne/q < \pdmnMinusOne/q + \omega-\sigma$ and so $\Omega$ must be bounded.
Recall that
\begin{align*}\doublebar{\arr \Psi}_{L^{ q, \sigma}_{av}(\Omega)}^q
&=
\int_\Omega (\arr \Psi)_W(x)^{ q} \dist(x,\partial\Omega)^{ q-1-q\sigma}\,dx
\end{align*}
where $(\arr\Psi)_W(x)=\bigl(\fint_{B(x,\dist(x,\partial\Omega)/2)} \abs{\arr\Psi}^2\bigr)^{1/2}$. By H\"older's inequality, if $\arr \Psi\in L^{r,\omega}_{av}(\Omega)$ then
\begin{align*}\doublebar{\arr \Psi}_{L^{ q, \sigma}_{av}(\Omega)}^q
&\leq
	\biggl(\int_\Omega {\arr \Psi}_W(x)^{r} \dist(x,\partial\Omega)^{r-1-r\omega}\,dx\biggr)^{q/r}
	\\&\qquad\times
	\biggl(\int_\Omega \dist(x,\partial\Omega)^{-1+(\omega-\sigma)qr/(r-q)}\,dx\biggr)^{1-q/r}
.\end{align*}
Because $\omega>\sigma$ and $r>q>0$, the second integral converges and we may derive the desired inequality.

We now consider the case $r\leq q$.
Let $\mathcal{G}$ be as in formula~\eqref{eqn:L:norm:whitney}.
Then
\begin{equation*}
\doublebar{\arr \Psi}_{L^{ q, \sigma}_{av}(\Omega)}
\approx
\biggl(\sum_{Q\in\mathcal{G}}\biggl(\fint_Q \abs{\arr \Psi}^2\biggr)^{q/2}\ell(Q)^{\dmnMinusOne+q-q\sigma}\biggr)^{1/q}
.\end{equation*}

%We begin with the case ${\pdmnMinusOne}/{q}-{\sigma} = {\pdmnMinusOne}/{r}-{\omega}$. We have that
%\begin{align*}\doublebar{\arr \Psi}_{L_{av}^{q,\sigma}(\Omega)}
%&\approx 
%%\biggl(\sum_{Q\in\mathcal{G}} \biggl(\biggl(\fint_Q \abs{\arr \Psi}^2\biggr)^{1/2} \ell(Q)^{\pdmnMinusOne/q+1-\sigma}\biggr)^q
%%\biggr)^{1/q}
%%\\&=
%\biggl(\sum_{Q\in\mathcal{G}} \biggl(\biggl(\fint_Q \abs{\arr \Psi}^2\biggr)^{1/2} \ell(Q)^{\pdmnMinusOne/r+1-\omega}\biggr)^q
%\biggr)^{1/q}
%.\end{align*}
%Observe that in this case $r<q$; thus, we may use basic properties of 
%sequence spaces to bound an $\ell^q$ norm by an $\ell^r$ norm. This yields the result that
%\begin{equation}
%\label{eqn:embedding:L}
%\doublebar{\arr \Psi}_{L_{av}^{q,\sigma}(\Omega)}
%\leq C(r,q,\omega,\sigma)\doublebar{\arr \Psi}_{L_{av}^{r,\omega}(\Omega)}
%\end{equation}
%provided  $0<r<q\leq \infty$ and ${\pdmnMinusOne}/{q}-{\sigma} = {\pdmnMinusOne}/{r}-{\omega}$.
%Notice that if we use the norm \eqref{eqn:L:norm:whitney} rather than the norm~\eqref{eqn:L:norm:2}, then the constant $C(r,q,\omega,\sigma)$ is equal to~$1$. 
%
%Now suppose that $\Omega$ is a bounded open set and $1/r <1/q+(\omega-\sigma)/\pdmnMinusOne$. We will consider several cases.
%

Because $\omega-\sigma+\pdmnMinusOne/q-\pdmnMinusOne/r\geq 0$, we have that if $\diam \Omega<\infty$ then
\begin{equation*}
\doublebar{\arr \Psi}_{L^{ q, \sigma}_{av}(\Omega)}^q
\lesssim\sum_{Q\in\mathcal{G}}\biggl(\fint_Q \abs{\arr \Psi}^2\biggr)^{q/2}\ell(Q)^{\pdmnMinusOne(q/r)+q-q\omega}
\diam \Omega^{\pdmnMinusOne(1-q/r)+q\omega-q\sigma}
.\end{equation*}
If $\diam\Omega=\infty$, then we consider only the case ${\pdmnMinusOne(1-q/r)+q\omega-q\sigma}=0$, and so the above formula is valid if we take $\diam \Omega^{\pdmnMinusOne(1-q/r)+q\omega-q\sigma}=1$.

Rewriting, we see that
\begin{multline*}
\doublebar{\arr \Psi}_{L^{ q, \sigma}_{av}(\Omega)}
\\\lesssim
\biggl(\sum_{Q\in\mathcal{G}}\biggl(\biggl(\fint_Q \abs{\arr \Psi}^2\biggr)^{1/2}\ell(Q)^{\pdmnMinusOne/r+1-\omega}\biggr)^q
\biggr)^{1/q}
\diam \Omega^{\pdmnMinusOne(1/q-1/r)+\omega-\sigma}
.\end{multline*}
If $r\leq q$, then we may bound the norm in the sequence space $\ell^q$ by the norm in~$\ell^r$. This completes the proof.
\end{proof}

This embedding result has two useful corollaries. The first allows us to extrapolate uniqueness of solutions; the second is a compatibility condition of the type required by Lemma~\ref{lem:interpolation}.

\begin{cor}\label{cor:unique:extrapolate} Let $L$ be an elliptic operator of order~$2m$. Let $\Omega$, $q$, $\sigma$, $r$ and $\omega$ satisfy the conditions of Lemma~\ref{lem:embedding}.

Suppose that the only solution to the problem 
\begin{equation}\label{eqn:Dirichlet:extrapolate:unique}
L\vec v=0\text{ in }\Omega,\quad \Tr_{m-1}^\Omega\vec v=0,\quad \vec v\in \dot W^{q,\sigma}_{m,av}(\Omega)\end{equation}
is $\vec v=0$ (as an element of $\dot W^{q,\sigma}_{m,av}(\Omega)$; that is, $\vec v$ is the equivalence class of functions $\{\vec V: \nabla^m\vec V=0$ in $\Omega\}$.)

Then for any $\arr H\in L^{r,\omega}_{av}(\Omega)$ and any $\arr f\in \dot W\!A^{ r}_{m-1,\omega}(\partial\Omega)$, there is at most one $\vec u\in  \dot W^{r,\omega}_{m,av}(\Omega)$ that satisfies
\begin{equation*}%\label{eqn:Dirichlet:extrapolate:unique:2}
L\vec u=\Div_m\arr H\text{ in }\Omega,\quad \Tr_{m-1}^\Omega\vec u=\arr f,\quad \vec u\in \dot W^{ r, \omega}_{m,av}(\Omega).\end{equation*}

A similar statement is valid for the Neumann problem.
\end{cor}

\begin{cor}\label{cor:compatible}
Let $L$ be an elliptic operator of order~$2m$. Let $\Omega$, $q$, $\sigma$, $r$ and $\omega$ satisfy the conditions of Lemma~\ref{lem:embedding}.

Suppose that the problem~\eqref{eqn:Dirichlet:extrapolate:unique} has only the trivial solution.

Suppose that $\arr H\in L^{r,\omega}_{av}(\Omega)\cap L^{q,\sigma}_{av}(\partial\Omega)$ and $\arr f\in \dot W\!A^r_{m-1,\omega}(\partial\Omega)\cap \dot W\!A^q_{m-1,\sigma}(\partial\Omega)$. Let $\vec u\in \dot W^{q,\sigma}_{m,av}(\Omega)$ and $\vec w\in \dot W^{r,\omega}_{m,av}(\Omega)$ satisfy $L\vec u= L\vec w=\Div_m\arr H$ in~$\Omega$ and $\Tr_{m-1}^\Omega\vec u=\Tr_{m-1}^\Omega\vec w=\arr f$. 

Then $\nabla^m \vec u=\nabla^m \vec w$ in~$\Omega$.

A similar result is valid for the Neumann problem.
\end{cor}

%\begin{proof} By Lemma~\ref{lem:embedding}, $\vec w\in \dot W^{ r, \omega}_{m,av}(\Omega)\subset \dot W^{q, \sigma}_{m,av}(\Omega)$, and so $\vec w-\vec u\in \dot W^{q, \sigma}_{m,av}(\Omega)$ and $L(\vec w-\vec u)=0$, $\Tr_{m-1}^\Omega(\vec w-\vec u)=0$. Thus $\vec w-\vec u=0$.
%\end{proof}

\section{\texorpdfstring{$L^\infty$}{Bounded} perturbation and well posedness}
\label{sec:perturb}

In this section we will prove Theorem~\ref{thm:perturb}. We will also prove Lemma~\ref{lem:BVP:duality}.

%Recall if the Dirichlet or Neumann problem for an operator ${L}=\Div_m\mat A\nabla^m$ is well posed in a Lipschitz domain~$\Omega$, then the corresponding boundary value problem for the operator ${M}=\Div_m\mat B\nabla^m$ is well posed provided $\doublebar{\mat A-\mat B}_{L^\infty(\Omega)}$ is small enough. 

We will begin (Lemma~\ref{lem:zero:boundary}) by reducing to the case of homogeneous boundary values.
Theorem~\ref{thm:exist:perturb} will establish that if solutions to ${L}\vec u=\Div_m\arr\Phi$ exist, then so must solutions to ${M}\vec u=\Div_m\arr H$. In Section~\ref{sec:duality} we will prove a generalization of Lemma~\ref{lem:BVP:duality}; specifically, we %Theorems~\ref{thm:exist:unique} and~\ref{thm:unique:exist} 
will show that uniqueness of solutions to  the Dirichlet or Neumann problem $L\vec u=\Div_m\arr H$, for data $\arr H\in L^{p,\smooth}_{av}(\Omega)$, is equivalent to existence of solutions to $L^*\vec u=\Div_m\arr \Phi$, for data $\arr \Phi\in L^{p',\smooth'}_{av}(\Omega)$. In Section~\ref{sec:perturb:full} we will combine these results to establish that uniqueness of solutions, like existence of solutions, is stable under $L^\infty$ perturbation.

\subsection{Reduction to the case of homogeneous boundary values}
\label{sec:extension}

In this subsection we will prove the following lemma. 

\begin{lem}\label{lem:zero:boundary} Let $\Omega$ be a Lipschitz domain with connected boundary. Let $0<\smooth<1$, let $\pmin<p\leq \infty$, and suppose that for every $\arr \Phi \in L^{p,\smooth}_{av}(\Omega)$ there exists a solution $\vec u$ to the Dirichlet problem with homogeneous boundary data
\begin{equation}\label{eqn:Dirichlet:interior}
L  \vec u = \Div_m \arr \Phi \text{ in }\Omega,
\quad \Tr_{m-1}^\Omega\vec u = 0,
\quad
\doublebar{\vec u}_{\dot W^{p,\smooth}_{m,av}(\Omega)} \leq C \doublebar{\arr \Phi }_{L^{p,\smooth}_{av}(\Omega)}
\end{equation}
or the Neumann problem
\begin{equation}\label{eqn:Neumann:interior}
L \vec u = \Div_m \arr \Phi  \text{ in }\Omega,
\quad \M_{\mat A,\arr \Phi}^\Omega\vec u = 0,
\quad
\doublebar{\vec u}_{\dot W^{p,\smooth}_{m,av}(\Omega)} \leq C \doublebar{\arr \Phi }_{L^{p,\smooth}_{av}(\Omega)}
.\end{equation}
Then for each $\arr H\in  L^{p,\smooth}_{av}(\Omega)$ and for each $\arr f\in {\dot W\!A^p_{m-1,\smooth}(\partial\Omega)}$ or $\arr g\in {\dot N\!A^p_{m-1,\smooth-1}(\partial\Omega)}$, respectively, there is a solution to the full Dirichlet problem
\begin{multline}
\label{eqn:Dirichlet:full}
%\left\{\begin{aligned}
L \vec u = \Div_m \arr H \text{ in }\Omega,
\quad \Tr_{m-1}^\Omega\vec u = \arr f,
\\
\doublebar{\vec u}_{\dot W^{p,\smooth}_{m,av}(\Omega)} \leq C \doublebar{\arr H}_{L^{p,\smooth}_{av}(\Omega)}
+C\doublebar{\arr f}_{\dot W\!A^p_{m-1,\smooth}(\partial\Omega)}
.%\end{aligned}\right.
\end{multline}
or the full Neumann problem
\begin{multline}
\label{eqn:Neumann:full}
L \vec u = \Div_m \arr H \text{ in }\Omega,
\quad \M_{\mat A,\arr H}^\Omega\vec u = \arr g,
\\
\doublebar{\vec u}_{\dot W^{p,\smooth}_{m,av}(\Omega)} \leq C \doublebar{\arr H}_{L^{p,\smooth}_{av}(\Omega)}
+C\doublebar{\arr g}_{\dot N\!A^p_{m-1,\smooth-1}(\partial\Omega)}
.\end{multline}

If solutions to the problems~\eqref{eqn:Dirichlet:interior} or~\eqref{eqn:Neumann:interior} are unique, then so are solutions to the problems~\eqref{eqn:Dirichlet:full} or~\eqref{eqn:Neumann:full}, respectively.
\end{lem}

\begin{proof}
The uniqueness follows by linearity; we need only establish existence.

Begin with the Dirichlet case. Let $\vec F$ satisfy $\Tr_{m-1}^\Omega\vec F=\arr f$; by \cite[Theorem~4.1]{Bar16pB}, there exists some such $\vec F$ that in addition satisfies $\doublebar{\vec F}_{\dot W^{p,\smooth}_{m,av}(\Omega)}\leq C\doublebar{\arr f}_{\dot W\!A^p_{m-1,\smooth}(\partial\Omega)}$.  Let $\arr\Phi = \arr H - \mat A\nabla^m\vec F$, and let $\vec v$ be the solution to the Dirichlet problem~\eqref{eqn:Dirichlet:interior} with data $\arr\Phi$.
Let $\vec u=\vec v+\vec F$. Then 
\begin{gather*}\Tr_{m-1}^\Omega \vec u=\Tr_{m-1}^\Omega \vec v+\Tr_{m-1}^\Omega \vec F = 0+\arr f
,\\
L\vec u = L\vec v+L\vec F = \Div_m \arr \Phi + \Div_m \mat A\nabla^m\vec F = \Div_m\arr H \text{ in }\Omega
,\\
\begin{aligned}
\doublebar{\vec u}_{\dot W^{p,\smooth}_{m,av}(\Omega)}
&\leq 
c_p\doublebar{\vec v}_{\dot W^{p,\smooth}_{m,av}(\Omega)}
+c_p\doublebar{\vec F}_{\dot W^{p,\smooth}_{m,av}(\Omega)}
\\&\leq 
C\doublebar{\arr H}_{L^{p,\smooth}_{av}(\Omega)}
+C\doublebar{\arr f}_{\dot W\!A^p_{m-1,\smooth}(\partial\Omega)}
\end{aligned}
\end{gather*}
as desired. (If $p\geq 1$ then $c_p=1$.)

The Neumann case is similar. Let $\arr G$ be the extension of $\arr g$ given by \cite[Theorem~6.1]{Bar16pB}. Let $\arr \Phi=\arr H+\arr G$ and let $\vec u$ be the solution to the Neumann problem~\eqref{eqn:Neumann:interior} with data $\arr\Phi$.

Then $\M_{\mat A,\arr\Phi}^\Omega \vec u=0$. 
If $\vec\varphi\in C^\infty_0(\R^\dmn)$ is a smooth testing function, then 
by the definition~\eqref{dfn:Neumann} of Neumann boundary values,
\begin{align*}
\langle \nabla^m\vec\varphi, \mat A\nabla^m \vec u\rangle_\Omega
&= \langle \nabla^m\vec\varphi, \arr \Phi\rangle_\Omega
= \langle \nabla^m\vec\varphi, \arr H\rangle_\Omega
+
\langle \nabla^m\vec\varphi, \arr G\rangle_\Omega
\end{align*}
and by \cite[Theorem~6.1]{Bar16pB}, $\langle \nabla^m\vec\varphi, \arr G\rangle_\Omega=\langle \Tr_{m-1}^\Omega\vec\varphi, \arr g\rangle_{\partial\Omega}$.
In particular, $L\vec u=\Div_m\arr H$ in~$\Omega$, and
\begin{equation*}\langle \Tr_{m-1}^\Omega\vec\varphi,\M_{\mat A,\arr H}^\Omega \vec u\rangle_{\partial\Omega}
=
\langle \nabla^m\vec\varphi, \mat A\nabla^m \vec u\rangle_\Omega
-
\langle \nabla^m\vec\varphi, \arr H\rangle_\Omega
=
\langle \Tr_{m-1}^\Omega\vec\varphi, \arr g\rangle_{\partial\Omega}
\end{equation*}
and so $\M_{\mat A,\arr H}^\Omega\vec u=\arr g$. Furthermore,
\begin{equation*}\doublebar{\vec u}_{\dot W^{p,\smooth}_{m,av}(\Omega)}
\leq 
C_1\doublebar{\arr \Phi}_{L^{p,\smooth}_{av}(\Omega)}
\leq
C_2\doublebar{\arr H}_{L^{p,\smooth}_{av}(\Omega)}
+C_2\doublebar{\arr g}_{\dot N\!A^p_{m-1,\smooth-1}(\partial\Omega)}
\end{equation*}
as desired.
\end{proof}

\subsection{Perturbation of existence}

In this section we will prove the following theorem. This theorem provides the existence component of Theorem~\ref{thm:perturb}. %We remark that we also wish to derive some results for compatibility of perturbed solutions (in the sense of Lemma~\ref{lem:interpolation}); thus, the notation of the 

\begin{thm} \label{thm:exist:perturb}
Suppose that $L$ is a differential operator of the form~\eqref{eqn:divergence}, of order $2m$ and acting on $\C^N$-valued functions, associated to bounded coefficients~$\mat A$.
Let $M$ be another operator of order $2m$, also acting on $\C^N$-valued functions, and associated to the coefficients~$\mat B$. Let $\doublebar{\mat A-\mat B}_{L^\infty}=\varepsilon$. 

Let $0<\smooth<1$ and let $\pmin<p\leq \infty$. Let $\Omega\subset\R^\dmn$ be a Lipschitz domain.
%Suppose that $\mat A$ satisfies the ellipticity condition \eqref{eqn:elliptic:everywhere} for some $\lambda>0$. 
Suppose that for every $\arr \Phi\in L^{p,\smooth}_{av}(\Omega)$ there exists a solution $\vec u$ to the Dirichlet problem~\eqref{eqn:Dirichlet:0}.
If $\varepsilon<1/C_0$, where $C_0$ is as in the problem~\eqref{eqn:Dirichlet:0}, then for each $\arr H\in L^{p,\smooth}_{av}(\Omega)$ there exists a solution $\vec u$ to the Dirichlet problem~\eqref{eqn:Dirichlet:1} with $\arr f=0$.

Similarly, if %$\mat A$ satisfies the ellipticity condition \eqref{eqn:elliptic:domain} for some $\lambda>0$, and if 
for every $\arr \Phi\in L^{p,\smooth}_{av}(\Omega)$ there exists a solution $\vec u$ to the Neumann problem~\eqref{eqn:Neumann:0},
then whenever  $\varepsilon<1/C_0$, we have that for each $\arr H\in L^{p,\smooth}_{av}(\Omega)$ there exists a solution $\vec u$ to the Neumann problem~\eqref{eqn:Neumann:1} with $\arr g=0$.

%Finally, suppose that $\mat A$ satisfies the ellipticity condition \eqref{eqn:elliptic:everywhere} for some $\lambda>0$, and that the Newton potential $\vec \Pi^{L}$ given by formula~\eqref{dfn:newton} extends to an operator defined on all of $L^{p,\smooth}_{av}(\Omega)$ that satisfies the bound 
%\begin{equation*}\doublebar{\vec\Pi^{L}\arr H}_{\dot W^{p,\smooth}_{m,av}(\R^\dmn\setminus\partial\Omega)}\leq C_0\doublebar{\arr H}_{\dot W^{p,\smooth}_{m,av}(\Omega)}.\end{equation*}
%If $\varepsilon<\min(\lambda,1/C_0)$, then $\vec\Pi^{M}$ also extends to an operator on $L^{p,\smooth}_{av}(\R^\dmn\setminus\partial\Omega)$, and we have the bound
%\begin{equation*}\doublebar{\vec\Pi^{M}\arr H}_{\dot W^{p,\smooth}_{m,av}(\R^\dmn\setminus\partial\Omega)}\leq C(C_0,p,\smooth)\doublebar{\arr H}_{\dot W^{p,\smooth}_{m,av}(\Omega)}.\end{equation*}

Finally, let $0<\smooth_j<1$ and $\pmin[\smooth_j]<p_j\leq \infty$ for $j=0$, $1$. Suppose that $\mat A$ satisfies the condition \eqref{eqn:elliptic:everywhere} or \eqref{eqn:elliptic:domain}, and that  the Dirichlet problems~\textup{(\ref{eqn:interpolation:Dirichlet:0}--\ref{eqn:interpolation:Dirichlet:1})} or Neumann problems~\textup{(\ref{eqn:interpolation:Neumann:0}--\ref{eqn:interpolation:Neumann:1})}, respectively, are compatibly well posed in the sense of Lemma~\ref{lem:interpolation}. If $\varepsilon<\min(1/C_0,1/C_1)$, then the perturbed solutions are compatible; that is, for each $\arr H\in L^{p_0,\smooth_0}_{av}(\Omega)\cap L^{p_1,\smooth_1}_{av}(\Omega)$ there exists a function $\vec u$ with 
\begin{equation*}
\left\{
\begin{gathered}
M  \vec u = \Div_m \arr H \text{ in }\Omega,
\quad \Tr_{m-1}^\Omega \vec u=0 \text{ or }\M_{\mat B,\arr H}\vec u=0,
\\
\doublebar{\vec u}_{\dot W^{p_0,\smooth_0}_{m,av}(\Omega)} \leq C(C_0,p,\varepsilon)\doublebar{\arr H}_{L^{p_0,\smooth_0}_{av}(\Omega)}
,\\
\doublebar{\vec u}_{\dot W^{p_1,\smooth_1}_{m,av}(\Omega)} \leq C(C_1,p,\varepsilon)\doublebar{\arr H}_{L^{p_1,\smooth_1}_{av}(\Omega)}.
\end{gathered}\right.
\end{equation*}

Here
\begin{equation*}C(c,p,\varepsilon)=\frac{c}{1-c\varepsilon}\text{ if }p\geq 1,\qquad
C(c,p,\varepsilon)=\biggl(\frac{c^p}{1-c^p\varepsilon^p}\biggr)^{1/p}\text{ if }p\leq 1.\end{equation*}
\end{thm}

\begin{proof}
Choose some $\arr H\in L^{p,\smooth}_{av}(\Omega)$. Let $\vec u_0$ be a solution to the Dirichlet problem~\eqref{eqn:Dirichlet:0} or the Neumann problem~\eqref{eqn:Neumann:0}
with data~$\arr H$. For each $j\geq 0$, let $\vec u_{j+1}$ be a solution to the given problem with data $(\mat A-\mat B)\nabla^m \vec u_{j}$. 

We then have that
\begin{equation*}
\doublebar{\nabla^m\vec u_0}_{L^{p,\smooth}_{av}({\Omega})} \leq C_0\doublebar{\arr H}_{L^{p,\smooth}_{av}({\Omega})},
\qquad
\doublebar{\nabla^m\vec u_{j+1}}_{L^{p,\smooth}_{av}({\Omega})} \leq C_0\varepsilon\doublebar{\nabla^m\vec u_j}_{L^{p,\smooth}_{av}({\Omega})}.\end{equation*}
We have that $\dot W^{p,\smooth}_{m,av}({\Omega})$ is a quasi-Banach space and thus is complete. Let $\vec u=\sum_{j=0}^\infty \vec u_j$.
If $p\geq 1$, then we have that
\begin{equation*}\doublebar{\vec u}_{\dot W^{p,\smooth}_{m,av}({\Omega})}
\leq \sum_{j=0}^\infty \doublebar{\vec u_j}_{\dot W^{p,\smooth}_{m,av}({\Omega})}
\leq \sum_{j=0}^\infty C_0(C_0\varepsilon)^j\doublebar{\arr H}_{L^{p,\smooth}_{av}(\Omega)}
\leq \frac{C_0}{1-C_0\varepsilon}\doublebar{\arr H}_{L^{p,\smooth}_{av}(\Omega)}.\end{equation*}
If $p\leq 1$, then $\dot W^{p,\smooth}_{m,av}({\Omega})$ is a quasi-Banach space and satisfies the $p$-norm inequality
\begin{equation*}
%\label{eqn:norm:p}
\doublebar{\vec u+\vec v}_{\dot W^{p,\smooth}_{m,av}({\Omega})}^p \leq 
\doublebar{\vec u}_{\dot W^{p,\smooth}_{m,av}({\Omega})}^p 
+\doublebar{\vec v}_{\dot W^{p,\smooth}_{m,av}({\Omega})}^p \end{equation*}
and so we have that
\begin{equation*}\doublebar{\vec u}_{\dot W^{p,\smooth}_{m,av}({\Omega})}^p
\leq \sum_{j=0}^\infty \doublebar{\vec u_j}_{\dot W^{p,\smooth}_{m,av}({\Omega})}^p
\leq \frac{C_0^p}{1-C_0^p\varepsilon^p}\doublebar{\arr H}_{L^{p,\smooth}_{av}(\Omega)}^p.\end{equation*}

Let $\vec\varphi$ be smooth and compactly supported in~$\R^\dmn$; if we seek to establish well posedness of the Dirichlet problem \eqref{eqn:Dirichlet:1}, we further require that $\vec\varphi$ be supported in~${\Omega}$.

By Lemma~\ref{lem:L:L1}, we have that $\langle \nabla^m\vec\varphi,\arr\Psi\rangle_{\Omega}$ represents an absolutely convergent integral whenever $\arr\Psi\in L^{p,\smooth}_{m,av}(\Omega)$, and so the following computations are valid.

By bilinearity of the inner product,
%\begin{align*}
%\langle \nabla^m\vec\varphi, \mat B\nabla^m \vec u\rangle_{\Omega}
%&=
%	\sum_{j=0}^\infty \langle \nabla^m\vec\varphi, \mat B\nabla^m \vec u_j\rangle_{\Omega}
%\\&=
%	\sum_{j=0}^\infty \langle \nabla^m\vec\varphi, (\mat B-\mat A)\nabla^m \vec u_j\rangle_{\Omega}
%	+
%	\sum_{j=0}^\infty \langle \nabla^m\vec\varphi, \mat A\nabla^m \vec u_j\rangle_{\Omega}
%.\end{align*}
%Reindexing the sum, we have that
\begin{align*}
\langle \nabla^m\vec\varphi, \mat B\nabla^m \vec u\rangle_{\Omega}
&=
	\langle \nabla^m\vec\varphi, \mat A\nabla^m \vec u_0\rangle_{\Omega}
	\\&\qquad+
	\sum_{j=0}^\infty \langle \nabla^m\vec\varphi, (\mat B-\mat A)\nabla^m \vec u_j\rangle_{\Omega}
	+
	\langle \nabla^m\vec\varphi, \mat A\nabla^m \vec u_{j+1}\rangle_{\Omega}
.\end{align*}
By definition of $\vec u_j$, we have that
\begin{align*}
\langle \nabla^m\vec\varphi, \mat B\nabla^m \vec u\rangle_{\Omega}
&=
	\langle \nabla^m\vec\varphi, \arr H\rangle_{\Omega}
	\\&\qquad+
	\sum_{j=0}^\infty \langle \nabla^m\vec\varphi, (\mat B-\mat A)\nabla^m \vec u_j\rangle_{\Omega}
	+
	\langle \nabla^m\vec\varphi, (\mat A-\mat B)\nabla^m \vec u_j\rangle_{\Omega}
\\&=
	\langle \nabla^m\vec\varphi, \arr H\rangle_{\Omega}
.\end{align*}
Recall from the definition~\eqref{dfn:Neumann} of $\M_{\mat A,\arr H}\vec u$ that $\vec u$ is a solution to the Neumann problem~\eqref{eqn:Neumann:1} if and only if $\doublebar{\vec u}_{\dot W^{p,\smooth}_{m,av}(\Omega)} \leq  C(C_0,p,\varepsilon) \doublebar{\arr \Phi}_{L^{p,\smooth}_{av}(\Omega)}$ and 
\begin{equation*}\langle \nabla^m \vec\varphi,\mat B\nabla^m \vec u\rangle_\Omega = \langle \nabla^m \vec\varphi,\arr H\rangle_\Omega\quad\text{for all $\vec\varphi\in C^\infty_0(\R^\dmn)$.}\end{equation*}

Thus, if $\vec u_j$ was a solution to the Neumann problem~\eqref{eqn:Neumann:0} then $\vec u$ is a solution to the Neumann problem~\eqref{eqn:Neumann:1}.

If $\vec u_j$ was a solution to the Dirichlet problem~\eqref{eqn:Dirichlet:0} then $M\vec u=\Div_m\arr H$ in ${\Omega}$. Furthermore, $\Tr_{m-1}^{\Omega} \vec u_j=0$ for each~$j$ and so $\Tr_{m-1}^{\Omega} \vec u=0$ as well; thus, $\vec u$ is a solution to the Dirichlet problem~\eqref{eqn:Dirichlet:1}, as desired.

Finally, if the Dirichlet problems~\textup{(\ref{eqn:interpolation:Dirichlet:0}--\ref{eqn:interpolation:Dirichlet:1})} or Neumann problems~\textup{(\ref{eqn:interpolation:Neumann:0}--\ref{eqn:interpolation:Neumann:1})} are compatibly well posed, we may choose $\vec u_j\in \dot W^{p_0,\smooth_0}_{m,av}(\Omega)\cap \dot W^{p_1,\smooth_1}_{m,av}(\Omega)$ and so $\vec u\in \dot W^{p_0,\smooth_0}_{m,av}(\Omega)\cap \dot W^{p_1,\smooth_1}_{m,av}(\Omega)$, with the desired bounds.
\end{proof}

\subsection{Duality}
\label{sec:duality}

We have shown that if a boundary value problem for $L$ is well posed, then solutions to the corresponding problem for $M$ exist. %(In fact, we have proven a slightly stronger result, as we never used uniqueness of solutions to the problem for~$L$.) 
We must now show that if a boundary value problem for $L$ is well posed, then solutions to the corresponding problem for $M$ are unique. 

We will do this by using duality results to relate uniqueness of solutions for $L$ to existence of solutions for $L^*$; we may then use Theorem~\ref{thm:exist:perturb} to produce perturbative results.

We remark that Theorems~\ref{thm:exist:unique} and~\ref{thm:unique:exist} are a generalization of Lemma~\ref{lem:BVP:duality}; they include some results for the case $p<1$.

\begin{thm}\label{thm:exist:unique} Suppose that $L$ is a differential operator of the form~\eqref{eqn:divergence}, of order $2m$ and acting on $\C^N$-valued functions, associated to bounded coefficients~$\mat A$.  Let $\Omega\subset\R^\dmn$ be a Lipschitz domain with connected boundary. Let $(A^*)^{jk}_{\alpha\beta}=\overline{A^{kj}_{\beta\alpha}}$, and let $L^*$ be the differential operator associated to~$\mat A^*$.

Let $0<\smooth<1$ and $\pmin<p\leq \infty$. If $p<\infty$, then let $\smooth'$ be the real number and $p'$ be the extended real number that satisfy
\begin{equation*}\frac{1}{p'} = \max\biggl(0,1-\frac{1}{p}\biggr),\qquad \smooth'=(1-\smooth)+\pdmnMinusOne\max\biggl(\frac{1}{p}-1,0\biggr).
\end{equation*}
If $p=\infty$, let $\smooth'$ be any number with $1-\smooth\leq \smooth'<1$ and let $p'$ satisfy
\begin{equation*}1+\frac{\smooth'-(1-\smooth)}{\pdmnMinusOne}=\frac{1}{p'}.\end{equation*}

%Suppose that $\mat A$ satisfies the ellipticity condition \eqref{eqn:elliptic:everywhere} for some $\lambda>0$. 
Suppose that for every $\arr H\in L^{p,\smooth}_{av}(\Omega)$ there exists at least one solution $\vec u$ to the Dirichlet problem 
\begin{equation}
\label{eqn:Dirichlet:p}
L \vec u = \Div_m \arr H \text{ in }\Omega,
\quad \Tr_{m-1}^\Omega\vec u = 0,
\quad {\vec u}\in{\dot W^{p,\smooth}_{m,av}(\Omega)}.
\end{equation}
Then for every $\arr \Phi\in L^{p',\smooth'}_{av}(\Omega)$ there is at most one solution to the Dirichlet problem
\begin{equation}
\label{eqn:Dirichlet:p'}
L^* \vec v = \Div_m \arr \Phi \text{ in }\Omega,
\quad \Tr_{m-1}^\Omega\vec v = 0,
\quad {\vec v}\in{\dot W^{p',\smooth'}_{m,av}(\Omega)}
.\end{equation}
Furthermore, if $p'\geq1$ and there is a constant $C_0$ such that there is at least one solution to the Dirichlet problem~\eqref{eqn:Dirichlet:p} that satisfies $\doublebar{\vec u}_{\dot W^{p,\smooth}_{m,av}(\Omega)} \leq C_0 \doublebar{\arr H}_{L^{p,\smooth}_{av}(\Omega)}$, then the  solution to the Dirichlet problem~\eqref{eqn:Dirichlet:p'}, if it exists, must satisfy \begin{equation*}\doublebar{\vec v}_{\dot W^{p',\smooth'}_{m,av}(\Omega)} \leq C_1 C_0 \doublebar{\arr \Phi}_{L^{p',\smooth'}_{av}(\Omega)}\end{equation*}
where $C_1$ is such that 
\begin{equation*}\abs{\langle \arr F,\arr G\rangle_\Omega} \leq \sqrt{C_1} \doublebar{\arr F}_{L^{p',\smooth'}_{av}(\Omega)} \doublebar{\arr G}_{L^{p,\smooth}_{av}(\Omega)},
\quad
\doublebar{\arr F}_{L^{p',\smooth'}_{av}(\Omega)}\leq \sqrt{C_1} \sup_{\arr G\neq 0} \frac{\abs{\langle \arr F,\arr G\rangle_\Omega}}{\doublebar{\arr G}_{L^{p,\smooth}_{av}(\Omega)}}
.\end{equation*}
In particular, if we use the norm~\eqref{eqn:L:norm:whitney} in $L^{p,\smooth}_{av}(\Omega)$, then $C_1=1$.

%Similarly, suppose $\mat A$ satisfies the ellipticity condition \eqref{eqn:elliptic:domain} for some $\lambda>0$.
Suppose that for every $\arr H\in L^{p,\smooth}_{av}(\Omega)$ there exists at least one solution $\vec u$ to the Neumann problem 
\begin{equation}
\label{eqn:Neumann:p}
L \vec u = \Div_m \arr H \text{ in }\Omega,
\quad \M_{\mat A,\arr H}^\Omega\vec u = 0,
\quad {\vec u}\in{\dot W^{p,\smooth}_{m,av}(\Omega)}
.\end{equation}
Then for every $\arr \Phi\in L^{p',\smooth'}_{av}(\Omega)$ there is at most one solution to the Neumann problem
\begin{equation}
\label{eqn:Neumann:p'}
L^* \vec v = \Div_m \arr \Phi \text{ in }\Omega,
\quad \M_{\mat A^*,\arr\Phi}^\Omega\vec v = 0,
\quad {\vec v}\in{\dot W^{p',\smooth'}_{m,av}(\Omega)}
.\end{equation}
Furthermore, if $p'\geq 1$ and there is a constant $C_0<\infty$ such that there is at least one solution to the Neumann problem~\eqref{eqn:Neumann:p} that satisfies $\doublebar{\vec u}_{\dot W^{p,\smooth}_{m,av}(\Omega)} \leq C_0 \doublebar{\arr H}_{L^{p,\smooth}_{av}(\Omega)}$, then every solution to the Neumann problem~\eqref{eqn:Neumann:p'}, if it exists, must satisfy $\doublebar{\vec v}_{\dot W^{p',\smooth'}_{m,av}(\Omega)} \leq C_1 C_0 \doublebar{\arr \Phi}_{L^{p',\smooth'}_{av}(\Omega)}$.

\end{thm}

\begin{proof} %We present the argument only for the Dirichlet problem; the argument for the Neumann problem is similar. 

Suppose that $\vec v$ and $\vec w$ are two solutions to the Dirichlet problem~\eqref{eqn:Dirichlet:p'} or Neumann problem~\eqref{eqn:Neumann:p'}. To show that $\nabla^m \vec v=\nabla^m\vec w$ it suffices to show that $\langle \arr H, \nabla^m \vec v\rangle_\Omega=\langle \arr H, \nabla^m \vec w\rangle_\Omega$ for all $\arr H$ bounded and compactly supported in~$\Omega$.

Choose some such $\arr H$. Observe that $\arr H\in L^{p,\smooth}_{av}(\Omega)$.
Let $\vec u$ be a solution to the Dirichlet problem~\eqref{eqn:Dirichlet:p} or Neumann problem~\eqref{eqn:Neumann:p} with data~$\arr H$; by assumption, at least one such $\vec u$ exists. If $C_0<\infty$, we require $\doublebar{\vec u}_{\dot W^{p,\smooth}_{m,av}(\Omega)} \leq C_0 \doublebar{\arr H}_{L^{p,\smooth}_{av}(\Omega)}$.
Then 
\begin{equation*}\langle \mat A\nabla^m \vec u,\nabla^m\vec\varphi\rangle_\Omega = \langle \arr H,\nabla^m\vec\varphi\rangle_\Omega\end{equation*}
for all $\vec\varphi$ smooth and compactly supported in~$\Omega$ (the Dirichlet problem) or $\R^\dmn$ (the Neumann problem). By density of smooth functions (see \cite[Theorem~3.15]{Bar16pB}), this is true for all $\vec\varphi\in\dot W^{p',\smooth'}_{m,av}(\Omega)$ (with $\Tr_{m-1}^\Omega\vec\varphi=0$ in the case of the Dirichlet problem, without restriction in the case of the Neumann problem). In particular, it is true for $\vec \varphi=\vec v$. Thus,
\begin{equation*}\langle \arr H,\nabla^m\vec v\rangle_\Omega
=\langle \mat A\nabla^m \vec u,\nabla^m\vec v\rangle_\Omega
=\langle \nabla^m \vec u,\mat A^*\nabla^m\vec v\rangle_\Omega.\end{equation*}
But because $\vec v$ is a solution to the problem~\eqref{eqn:Dirichlet:p'}, we have that
\begin{equation*}\langle \nabla^m \vec u,\mat A^*\nabla^m\vec v\rangle_\Omega
=\langle \nabla^m \vec u,\arr \Phi\rangle_\Omega.\end{equation*}
Similarly 
\begin{equation*}\langle \arr H,\nabla^m\vec w\rangle_\Omega = \langle \nabla^m \vec u,\arr \Phi\rangle_\Omega \end{equation*}
and so we have that
\begin{equation*}\langle \arr H,\nabla^m\vec v\rangle_\Omega = \langle \arr H,\nabla^m\vec w\rangle_\Omega \end{equation*}
as desired.

Furthermore, if $C_0<\infty$ then 
\begin{align*}
\abs{\langle \arr H,\nabla^m\vec v\rangle_\Omega}
&=
	\abs{\langle \nabla^m \vec u,\arr \Phi\rangle_\Omega}
\leq 
	\sqrt{C_1}\doublebar{\vec u}_{\dot W^{p,\smooth}_{m,av}(\Omega)} \doublebar{\arr \Phi}_{L^{p',\smooth'}_{av}(\Omega)}
\\&\leq 
	\sqrt{C_1}C_0 \doublebar{\arr H}_{L^{p,\smooth}_{av}(\Omega)}\doublebar{\arr \Phi}_{L^{p',\smooth'}_{av}(\Omega)}
.\end{align*}
If $p'\geq 1$ then
\begin{equation*}\doublebar{\vec v}_{\dot W^{p',\smooth'}_{m,av}(\Omega)}
\leq \sqrt{C_1} \sup_{\arr H\in L^{p,\smooth}_{av}(\Omega)} \frac{\abs{\langle \arr H,\nabla^m\vec v\rangle_\Omega}}{\doublebar{\arr H}_{L^{p,\smooth}_{av}(\Omega)}}
\leq C_1 C_0\doublebar{\arr \Phi}_{L^{p',\smooth'}_{av}(\Omega)}\end{equation*}
as desired.
\end{proof}

We now prove the converse. We remark that the converse is somewhat more delicate and that more must be assumed. 

Specifically, first, existence of solutions for $p=\infty$ implies uniqueness of solutions for a range of $p\leq 1$; uniqueness of solutions for $p=\infty$ is not at present known to imply existence results even for $p=1$.
%if $p=\infty$ then existence of solutions to the $L^{\infty,\smooth}_{av}$-boundary value problem for~$L$ implied uniqueness of solutions to the $L^{p',\smooth'}_{av}$-boundary value problem for~$L^*$ for a range of $p'\leq 1$. In the converse case, uniqueness of solutions to the $L^{\infty,\smooth}_{av}$-boundary value problem for~$L$ implies existence of solutions to the $L^{p',\smooth'}_{av}$-Dirichlet problem for~$L^*$ only for $p'=1$.

Second, if $p<1$ then we will need to assume both uniqueness and existence of solutions to the $L^{p,\smooth}_{av}$-boundary value problem for~$L$ to derive existence of solutions to the $L^{p',\smooth'}_{av}$-boundary value problem for~$L^*$.

Finally, recall that in Theorem~\ref{thm:exist:unique} we were able to derive uniqueness of solutions to the problem \eqref{eqn:Dirichlet:p'} or~\eqref{eqn:Neumann:p'} given only existence of solutions to the boundary value problem \eqref{eqn:Dirichlet:p} or~\eqref{eqn:Neumann:p}. The stronger condition $\doublebar{\vec u}_{\dot W^{p,\smooth}_{m,av}(\Omega)} \leq C_0 \doublebar{\arr H}_{L^{p,\smooth}_{av}(\Omega)}$ was not necessary for mere uniqueness (although in the case $p'\geq 1$ it did yield a stronger result). In Theorem~\ref{thm:unique:exist}, we will need to assume the condition $\doublebar{\vec u}_{\dot W^{p,\smooth}_{m,av}(\Omega)} \leq C_0 \doublebar{\arr H}_{L^{p,\smooth}_{av}(\Omega)}$ to establish existence of solutions to the problem \eqref{eqn:Dirichlet:p'} or~\eqref{eqn:Neumann:p'}; nothing will be proven given only uniqueness of solutions to the boundary value problem \eqref{eqn:Dirichlet:p} or~\eqref{eqn:Neumann:p}.

\begin{thm} \label{thm:unique:exist} Let $L$, $\mat A$, $\Omega$ be as in Theorem~\ref{thm:exist:unique}. %We further require that $\Omega$ be connected. 
Let $0<\smooth<1$, $\pmin<p<\infty$, and let 
\begin{equation*}\frac{1}{p'} = \max\biggl(0,1-\frac{1}{p}\biggr),\qquad \smooth'=(1-\smooth)+\pdmnMinusOne\max\biggl(\frac{1}{p}-1,0\biggr).
\end{equation*}

Suppose that there is some constant $C_0<\infty$ such that, if $\vec u$ is a solution to the Dirichlet problem~\eqref{eqn:Dirichlet:p} with data~$\arr H$, then $\doublebar{\vec u}_{\dot W^{p,\smooth}_{m,av}(\Omega)}\leq C_0\doublebar{\arr H}_{L^{p,\smooth}_{av}(\Omega)}$. If $p<1$, suppose in addition that a (necessarily unique) solution to the Dirichlet problem~\eqref{eqn:Dirichlet:p} exists for all $\arr H\in L^{p,\smooth}_{av}(\Omega)$.

Then for all $\arr\Phi\in L^{p',\smooth'}_{av}(\Omega)$ there is at least one solution $\vec v$ to the Dirichlet problem~\eqref{eqn:Dirichlet:p'}; furthermore, at least one such solution satisfies $\doublebar{\vec v}_{\dot W^{p',\smooth'}_{m,av}(\Omega)}\leq C_1C_0\doublebar{\arr \Phi}_{L^{p',\smooth'}_{av}(\Omega)}$, where $C_1$ is as in Theorem~\ref{thm:exist:unique}.

A similar result is valid for the Neumann problem.
\end{thm}

% Injectivity of T does not imply surjectivity of T^*.
% Let Tf = (f_1, (1/2) f_2, (1/3) f_3, ... )
% T is self-adjoint, but T^{-1} f = (f_1, 2 f_2, 3 f_3, ... )
% which is not necessarily in \ell^2, so $T$ may not be onto.

\begin{proof}
Let $E$ denote the space of all $\arr H\in L^{p,\smooth}_{av}(\Omega)$ such that a solution to the problem~\eqref{eqn:Dirichlet:p} or~\eqref{eqn:Neumann:p} exists. By assumption, if $p<1$ then $E=L^{p,\smooth}_{av}(\Omega)$. 

Let $T:E\mapsto L^{p,\smooth}_{av}(\Omega)$ be given by $T\arr H=\nabla^m \vec u$, where $\vec u$ is the solution to the problem~\eqref{eqn:Dirichlet:p} or~\eqref{eqn:Neumann:p} with data~$\arr H$. If $p<1$ then $T$ is defined on $L^{p,\smooth}_{av}(\Omega)$ by assumption. If $p\geq 1$, then by the Hahn-Banach theorem we may extend $T$ to a bounded linear operator on all of $L^{p,\smooth}_{av}(\Omega)$.

Observe that there are two subspaces of $E$ for which we may easily evaluate $T$:
\begin{itemize}
\item If $\arr H= \mat A\nabla^m\vec\varphi$ for some $\vec\varphi\in \dot W^{p,\smooth}_{m,av}(\Omega)$ (with $ \Tr_{m-1}^\Omega\vec\varphi=0$ in the case of the Dirichlet problem), then $\arr H\in E$ and $T\arr H=\vec\varphi$.
\item If $\langle \nabla^m\vec \varphi,\arr H\rangle_\Omega=0$, for all $\vec\varphi \in C^\infty_0(\Omega)$ in the case of the Dirichlet problem, or all $\vec\varphi\in C^\infty_0(\R^\dmn)$ in the case of the Neumann problem, then $\arr H\in E$ and $T\arr H=0$. 
\end{itemize}

We now bound the adjoint $T^*$ to~$T$.
If $1\leq p<\infty$, then by formula~\eqref{eqn:L:dual} we have that $L^{p',\smooth'}_{av}(\Omega)$ is the dual space to $L^{p,\smooth}_{av}(\Omega)$, and so $T^*$ is a bounded linear operator 
$L^{p',\smooth'}_{av}(\Omega)\mapsto L^{p',\smooth'}_{av}(\Omega)$.

If $\pmin<p<1$, let $\mathcal{G}$ be a grid of dyadic Whitney cubes, as in the norm~\eqref{eqn:L:norm:whitney}. Let $\widetilde \smooth = \smooth-\pdmnMinusOne(1/p-1)$.
Let $\arr H\in L^{p,\smooth}_{av}(\Omega)$, so $T\arr H\in L^{p,\smooth}_{av}(\Omega)$.
By Lemma~\ref{lem:embedding}, $T\arr H\in L^{1,\widetilde\smooth}_{av}(\Omega)$. Furthermore,
\begin{equation*}\doublebar{T\arr H}_{L^{1,\widetilde\smooth}_{av}(\Omega)}
\leq \sum_{Q\in \mathcal{G}} \doublebar{T(\1_Q\arr H)}_{L^{1,\widetilde\smooth}_{av}(\Omega)}.\end{equation*}
Another application of Lemma~\ref{lem:embedding} yields that
\begin{equation*}\doublebar{T\arr H}_{L^{1,\widetilde\smooth}_{av}(\Omega)}
\leq C\sum_{Q\in \mathcal{G}} \doublebar{T(\1_Q\arr H)}_{L^{p,\smooth}_{av}(\Omega)}.\end{equation*}
By boundedness of~$T$, we have that
\begin{equation*}\doublebar{T(\1_Q\arr H)}_{L^{p,\smooth}_{av}(\Omega)} 
\leq C \doublebar{\1_Q\arr H}_{L^{p,\smooth}_{av}(\Omega)}
.\end{equation*}
Because $\1_Q\arr H$ is supported in~$Q$,
\begin{equation*}\doublebar{\1_Q\arr H}_{L^{p,\smooth}_{av}(\Omega)} 
\approx\doublebar{\1_Q\arr H}_{L^{1,\widetilde\smooth}_{av}(\Omega)}
.\end{equation*}
Thus
\begin{equation*}\doublebar{T\arr H}_{L^{1,\widetilde\smooth}_{av}(\Omega)}
\leq C\sum_{Q\in \mathcal{G}} \doublebar{\1_Q\arr H}_{L^{1,\widetilde\smooth}_{av}(\Omega)}
\approx \doublebar{\arr H}_{L^{1,\widetilde\smooth}_{av}(\Omega)}.
\end{equation*}
Thus, $T$ extends by density to a bounded operator on ${L^{1,\widetilde\smooth}_{av}(\Omega)}$. Observe that $1-\widetilde\smooth=\smooth'$ and so boundedness of $T^*$ on $L^{\infty,\smooth'}_{av}(\Omega)$ follows from the results for $p=1$.

Thus, if $\pmin<p< \infty$ then $T^*$ is a bounded operator on $L^{p',\smooth'}_{av}(\Omega)$. It is elementary to show that $\doublebar{T^*}\leq C_1C_0$. It suffices to show that if $\arr\Phi\in  L^{p',\smooth'}_{av}(\Omega)$ then $T^*\arr\Phi=\nabla^m \vec v$ for some $\vec v\in L^{p',\smooth'}_{av}(\Omega)$, and that $\vec v$ is a solution to problem~\eqref{eqn:Dirichlet:p'} or~\eqref{eqn:Neumann:p'}.

Suppose first that $p>1$. 
Let $W=\{\nabla^m \vec v: \vec v\in \dot W^{p',\smooth'}_{m,av}(\Omega)\}$ if we seek to solve the Neumann problem, or $W=\{\nabla^m \vec v: \vec v\in \dot W^{p',\smooth'}_{m,av}(\Omega),\>\Tr_{m-1}^\Omega\vec v=0\}$ if we seek to solve the Dirichlet problem. Then $W$ is a closed subspace of $L^{p',\smooth'}_{av}(\Omega)$. 

Suppose that $T^*\arr\Phi\notin W$ for some $\arr\Phi\in L^{p',\smooth'}_{av}(\Omega)$. Because $W$ is closed, there is some $\varepsilon>0$ such that $\doublebar{T^*\arr\Phi-\nabla^m \vec v}_{L^{\smash{p',\smooth'}\vphantom{p,\smooth}}_{av}(\Omega)}\geq \varepsilon$ for every $\nabla^m\vec v\in W$.

If $p>1$, then $p'<\infty$ and so the dual space of $\smash{L^{p',\smooth'}_{av}(\Omega)}$ is ${L^{p,\smooth}_{av}(\Omega)}$.
It is a standard result in functional analysis (see, for example, \cite[Theorem~4.8.3]{Fri82}) that there is some $\arr H\in {L^{p,\smooth}_{av}(\Omega)}$ such that $\langle \arr H, T^*\arr\Phi\rangle_\Omega = 1$ and $\langle \arr H, \arr w\rangle_\Omega=0$ for every $\arr w\in W$. Recalling the definition of~$W$, we have that $\langle \arr H, \nabla^m\vec v\rangle_\Omega=0$ for all $\vec v\in \dot W^{p',\smooth'}_{m,av}(\Omega)$ (possibly with the additional assumption $\Tr_{m-1}^\Omega\vec v=0$).

But if $\langle \arr H,\nabla^m \vec v\rangle_\Omega=0$ for every such~$\vec v$, then in particular $\langle \arr H,\nabla^m \vec \varphi\rangle_\Omega=0$ for any $\varphi$ smooth and compactly supported (in $\Omega$ or $\R^\dmn$); thus  $T\arr H=0$. But then $\langle \arr H, T^*\arr \Phi\rangle_\Omega = 0$, contradicting our assumption; thus, $T^*\arr \Phi\in W$. % and so $T^*\arr \Phi=\nabla^m\vec v$ for some $\vec v\in \dot W^{p',\smooth'}_{m,av}(\Omega)$, as desired. Notice that in the case of the Dirichlet problem we also have that $\Tr_{m-1}^\Omega\vec v=0$.

If $p\leq 1$ and so $p'=\infty$, then $T^*\arr \Phi\in L^{\infty,\smooth'}_{av}(\Omega)$.
By Lemma~\ref{lem:L:L1}, if $\arr F\in {L^{\infty,\smooth'}_{av}(\Omega)}$ then
\begin{equation*}
\doublebar{\arr F}_{L^1(\Omega;\, dx/(1+\abs{x}^\dmn))}
= \int_\Omega \abs{\arr F(x)} \,\frac{1}{1+\abs{x}^\dmn}\,dx
\leq  C \doublebar{\arr F}_{L^{\infty,\smooth'}_{av}(\Omega)}\end{equation*}
whenever $\smooth'<1$.

Let $W=\{\nabla^m \vec v:\nabla^m\vec v\in {L^1(\Omega;\, dx/(1+\abs{x}^\dmn))}\}$ or $W=\{\nabla^m \vec v:\nabla^m\vec v\in {L^1(\Omega;\, dx/(1+\abs{x}^\dmn))},\>\Tr_{m-1}^\Omega\vec v=0\}$. Because ${L^1(\Omega;\, dx/(1+\abs{x}^\dmn))} \subset L^1_{loc}(\overline\Omega)$, we have that $\Tr_{m-1}^\Omega\vec v$ is meaningful.

As before, if $T^*\arr \Phi\notin W$ then there is some $\arr H$ with $\esssup_\Omega \abs{\arr H(x)} (1+\abs{x}^\dmn)<\infty$ such that $\langle \arr H,T^*\arr \Phi\rangle_\Omega=1$ and $\langle \arr H,\nabla^m \vec v\rangle_\Omega=0$ for all $\vec v$ smooth and compactly supported. 
By Lemma~\ref{lem:L:L-infinity}, if $p>\pmin$ then $\arr H \in L^{p,\smooth}_{av}(\Omega)$, and so $\langle \arr H,T^*\arr \Phi\rangle_\Omega=\langle T\arr H,\arr \Phi\rangle_\Omega$. We may derive a contradiction as before.

Thus, in either case, $T^*\arr \Phi=\nabla^m\vec v$ for some $\vec v\in \dot W^{p',\smooth'}_{m,av}(\Omega)$, as desired. Notice that in the case of the Dirichlet problem we also have that $\Tr_{m-1}^\Omega\vec v=0$.

Finally, suppose that $\vec\varphi\in C^\infty_0(\Omega)$ (the Dirichlet problem) or $\vec\varphi\in C^\infty_0(\R^\dmn)$ (the Neumann problem).
Then
\begin{equation*}
\langle \nabla^m \vec\varphi,\mat A^*\nabla^m \vec v\rangle_\Omega
=
\langle \mat A \nabla^m \vec\varphi,T^*\arr\Phi\rangle_\Omega
=
\langle T(\mat A \nabla^m \vec\varphi),\arr\Phi\rangle_\Omega
=
\langle  \nabla^m \vec\varphi,\arr\Phi\rangle_\Omega
\end{equation*}
and so $\vec v$ is a solution to the Dirichlet problem~\eqref{eqn:Dirichlet:p'} or the Neumann problem problem~\eqref{eqn:Neumann:p'}, as desired.
\end{proof}

\subsection{Perturbation of full well posedness}
\label{sec:perturb:full}

We now prove Theorem~\ref{thm:perturb}.

By Theorem~\ref{thm:exist:perturb}, there is at least one solution to the problem \eqref{eqn:Dirichlet:1} or~\eqref{eqn:Neumann:1}. Furthermore, by Theorems~\ref{thm:exist:unique} and~\ref{thm:unique:exist}, if $p<\infty$, if $p'$ and $\smooth'$ are as in Theorem~\ref{thm:unique:exist}, if and $\arr \Phi\in L^{p',\smooth'}_{av}(\Omega)$, then there is a unique solution to the Dirichlet problem
\begin{equation*}
{L}^* \vec u = \Div_m \arr \Phi \text{ in }\Omega,
\quad \Tr_{m-1}^\Omega\vec u = 0,
\quad
\doublebar{\vec u}_{\dot W^{p',\smooth'}_{m,av}(\Omega)} \leq C_1C_0 \doublebar{\arr \Phi}_{L^{p',\smooth'}_{av}(\Omega)}
\end{equation*}
or the Neumann problem 
\begin{equation*}
{L}^* \vec u = \Div_m \arr \Phi \text{ in }\Omega,
\quad \M_{\mat A^*,\arr\Phi}^\Omega\vec u = 0,
\quad
\doublebar{\vec u}_{\dot W^{p,\smooth}_{m,av}(\Omega)} \leq  C_1C_0 \doublebar{\arr \Phi}_{L^{p,\smooth}_{av}(\Omega)}
.\end{equation*}
Again by Theorem~\ref{thm:exist:perturb}, there must exist solutions to the corresponding problems for the operator ${M}^*$, and so another application of Theorem~\ref{thm:exist:unique} implies that the solutions to the problems \eqref{eqn:Dirichlet:1} or~\eqref{eqn:Neumann:1} must be unique, as desired.

\section{Energy solutions and well posedness near \texorpdfstring{$p=2$, $\smooth=1/2$}{a special point}}
\label{sec:neighborhood}

In this section we will prove Theorem~\ref{thm:neighborhood}. Interpolation methods will be essential to our argument; thus, we will also prove Lemma~\ref{lem:interpolation}.

In Section~\ref{sec:newton:bounded}, we will define the Newton potential and bound it for constant coefficients on our weighted averaged Lebesgue spaces $L^{p,\smooth}_{av}(\Omega)$. %We will use boundedness of the Newton potential in Section~\ref{sec:interpolation}; we will also use boundedness of the Newton potential in Section~\ref{sec:biharmonic} to establish well posedness results for biharmonic operators.

We will review interpolation theory and establish interpolation results for $L^{p,\smooth}_{av}(\Omega)$ and related spaces in Section~\ref{sec:interpolation}; in particular, we will use boundedness of the Newton potential to establish interpolation results for the spaces $\dot W^{p,\smooth}_{m,av}(\Omega)$. (We will also use boundedness of the Newton potential in Section~\ref{sec:biharmonic} to establish well posedness results for biharmonic operators.) We will then use these interpolation results to prove Lemma~\ref{lem:interpolation}.

Finally, we will complete the proof of Theorem~\ref{thm:neighborhood} in Section~\ref{sec:neighborhood:proof}.

\subsection{Boundedness of the Newton potential for constant coefficients}
\label{sec:newton:bounded}

Suppose that $L$ is an elliptic operator of the form~\eqref{eqn:divergence} associated to some coefficients $\mat A$ that satisfy the bound~\eqref{eqn:elliptic:bounded} and the ellipticity condition~\eqref{eqn:elliptic:everywhere}. By the Lax-Milgram lemma, if $\arr H\in L^2(\R^\dmn)$, then there is a unique function $\vec u=\vec\Pi^L\arr H$  in $\dot W^2_m(\R^\dmn)$ that satisfies $L(\vec\Pi^L\arr H)=\Div_m\arr H$, that is, that satisfies
\begin{equation}
\label{dfn:newton}
\bigl\langle\nabla^m\vec\varphi, \mat A \nabla^m(\vec\Pi^L\arr H) \bigr\rangle_{\R^\dmn} 
=
\bigl\langle\nabla^m\vec\varphi, \arr H \bigr\rangle_{\R^\dmn}
\end{equation}
for all $\vec\varphi\in \dot W^2_m(\R^\dmn)$.
Furthermore, $\vec\Pi^L$ is a linear operator and is bounded $L^2(\R^\dmn)\mapsto \dot W^2_m(\R^\dmn)$, with operator norm at most $1/\lambda$. The kernel of $\vec\Pi^L$ is called the fundamental solution and was constructed for general higher order operators in \cite{Bar16}; we refer the interested reader to \cite{Bar16} for a more detailed discussion of the Newton potential~$\vec\Pi^L$.

By formula~\eqref{eqn:L2:2:half}, $\arr H\mapsto \vec \Pi^L (\mathcal{E}^0_\Omega\arr H)\big\vert_\Omega$ is bounded $L^{2,1/2}_{av}(\Omega)\mapsto \dot W^{2,1/2}_{m,av}(\Omega)$, where $\mathcal{E}^0_\Omega$ denotes extension by zero.
Under some circumstances we can bound this operator on 
$L^{p,\smooth}_{av}(\Omega)$ for more general $p$ and~$\smooth$. 
Observe the presence of the extension operator $\mathcal{E}_\Omega^0$ and the restriction operator $\smash{\big\vert_\Omega}$ in the above expression. It will be more convenient to consider $\smash{\vec\Pi^L}$ without extension and restriction operators. To this end, we will work with with global analogues of $L^{p,\smooth}_{av}(\Omega)$ and~$\dot W^{p,\smooth}_{m,av}(\Omega)$.

Observe that if $\Omega$ is an open set and $\partial\Omega$ has measure zero, then the norm in the space $L^{p,\smooth}_{av}(\R^\dmn\setminus\partial\Omega)$ satisfies
\begin{equation*}\doublebar{\arr H}_{L^{p,\smooth}_{av}(\R^\dmn\setminus\partial\Omega)}^p = 
\doublebar{\arr H\big\vert_\Omega}_{L^{p,\smooth}_{av}(\Omega)}^p
+\doublebar{\arr H\big\vert_{\R^\dmn\setminus\overline\Omega}}_ {L^{p,\smooth}_{av}(\R^\dmn\setminus\overline\Omega)}^p
.\end{equation*}
We define the global analogue of $\dot W^{p,\smooth}_{m,av}(\Omega)$ as follows.
\begin{align}
\label{eqn:W:global}
\widetilde W_{m,av}^{p,\smooth}(\partial\Omega)
&=\{\vec F\in  W^1_{m,loc}(\R^\dmn): \nabla^m\vec F\in L^{p,\smooth}_{av}(\R^\dmn\setminus\partial\Omega)\}
%\\&= \{\vec F: \1_\Omega\vec F\in \dot W^{p,\smooth}_{m,av}(\Omega),\> \1_{\R^\dmn\setminus\overline\Omega}\vec F\in \dot W^{p,\smooth}_{m,av}({\R^\dmn\setminus\overline\Omega}), 
%\\&\qquad\qquad \qquad\qquad\qquad
%\Tr_{m-1}^\Omega (\1_\Omega\vec F) = \Tr_{m-1}^{\R^\dmn\setminus\overline\Omega} (\1_{\R^\dmn\setminus\overline\Omega}\vec F)\}
.\end{align}
Notice that $\vec F\in \smash{\widetilde W_{m,av}^{p,\smooth}(\partial\Omega)}$ is a stronger condition than $\smash{\nabla^m\vec F\big\vert_{\R^\dmn\setminus\partial\Omega}}\in L_{av}^{p,\smooth}(\R^\dmn\setminus\partial\Omega)$; specifically, we require some compatibility of $\vec F$ across the boundary.

%
%Because $\vec \Pi^L\arr H$ is the \emph{unique} solution to formula~\eqref{dfn:newton}, we have that
%\begin{align}
%\label{eqn:newton:identity}
%\vec \Pi^L(\mat A \nabla^m\vec F)=\vec F &\text{ for all $\vec F\in\dot W^2_{m}(\R^\dmn)$}
%.\end{align}
%
%

We may now state a boundedness result for the Newton potential.

\begin{lem}\label{lem:polyharmonic:bounded} Let $\Omega\subset\R^\dmn$ be a Lipschitz domain. Let $L_0$ be an operator of the form~\eqref{eqn:divergence} that satisfies the ellipticity condition~\eqref{eqn:elliptic:everywhere} and has constant coefficients.

Then the operator $\vec\Pi^{L_0}$ defined by formula~\eqref{dfn:newton} extends to an operator that is bounded $L_{av}^{p,\smooth}(\R^\dmn\setminus\partial\Omega)\mapsto \widetilde W_{m,av}^{p,\smooth}(\partial\Omega)$ for any $0<\smooth<1$ and any $\pmin<p\leq\infty$. 
\end{lem}

The remainder of this subsection will be devoted to a proof of Lemma~\ref{lem:polyharmonic:bounded}. We begin with the following bound (in unaveraged spaces) in the case $1<p<\infty$.

\begin{lem}\label{lem:polyharmonic:Ap}
Let $L_0$ and $\Omega$ be as in Lemma~\ref{lem:polyharmonic:bounded}. Let $0<\smooth<1$ and $1<p<\infty$, and let $\arr H\in L^2(\R^\dmn)\cap L^{p,\smooth}_{av}(\R^\dmn\setminus\partial\Omega)$.
Then 
\begin{equation*}\int_{\R^\dmn} \abs{\nabla^m \vec\Pi^{L_0} \arr H(x)}^p \dist(x,\partial\Omega)^{p-1-p\smooth}\,dx
\leq C \int_{\R^\dmn} \abs{\arr H(x)}^p \dist(x,\partial\Omega)^{p-1-p\smooth}\,dx.\end{equation*}
\end{lem}

\begin{proof}
We claim that $\nabla^m\vec\Pi^{L_0}$ is a Calder\'on-Zygmund operator. By construction, $\nabla^m\vec\Pi^{L_0}$ is bounded on $L^2(\R^\dmn)$, and so we need only study its kernel.

For constant coefficients, an elementary argument involving Plancherel's theorem yields that the ellipticity condition~\eqref{eqn:elliptic:everywhere} is equivalent to %the statement that $\mat{\widetilde A}(\omega)$ is positive definite with $\langle \vec v, \mat{\widetilde A}(\omega)\vec v\rangle \geq \lambda \abs{\vec v}^2 \abs{\omega}^{2m}$. This is 
the ellipticity condition (4.15) of \cite{MitM13A}. By results of \cite{Sha45,Mor54,Joh55,Hor03}, assembled as \cite[Theorem~4.2]{MitM13A}, we have that $L$ has a fundamental solution $\mat E^{L_0}$ that satisfies the conditions
\begin{equation*}
\partial^\alpha\vec\Pi^{L_0}_j\arr H(x) 
	= \sum_{k=1}^N \sum_{\abs{\beta}=m} \int_{\R^\dmn} 	\partial_x^\alpha\partial_y^\beta E^{L_0}_{j,k}(x-y)\,H_{k,\beta}(y)\,dy
	\quad\text{for a.e.\ $x\notin\supp \arr H$}
\end{equation*}
for $\abs\alpha=m$, $1\leq j\leq N$, and
\begin{equation}\label{eqn:fundamental:decay}
\abs{\nabla^{2m} \mat E^{L_0}(x)}\leq \frac{C}{\abs{x}^\dmn},
\quad \abs{\nabla^{2m+1} \mat E^{L_0}(x)}\leq \frac{C}{\abs{x}^{\dmn+1}}.\end{equation}

This may also be verified by considering the fundamental solution $\mat E^{L_0}(x,y)$ of \cite{Bar16} (constructed for general variable coefficients); by translational symmetry of $L_0$, we have that $\partial_x^\alpha\partial_y^\beta \mat E^{L_0}(x,y)=\partial_x^\alpha\partial_y^\beta \mat E^{L_0}(x-y,0)$. The bound \cite[formula~(63)]{Bar16} yields an $L^2$ estimate on $\nabla^{2m} \mat E^{L_0}$, and the Caccioppoli inequality \cite[Corollary~22]{Bar16}, Morrey's inequality, and the fact that any derivative of a solution $\vec u$ to $L_0\vec u=0$ is itself a solution, allows us to pass to pointwise bounds on $\nabla^{2m} \mat E^{L_0}$ and~$\nabla^{2m+1} \mat E^{L_0}$.

Thus, $\nabla^m\vec\Pi^{L_0}$ is a Calder\'on-Zygmund operator.
Recall (see, for example, \cite[Chapter~V]{Ste93}) that a function $\omega$ defined on $\R^\dmn$ is a Muckenhoupt $A_p$ weight if, for every ball $B\subset\R^\dmn$, we have that
\begin{equation*}\fint_B \omega(x)\,dx \biggl(\fint_B \omega(x)^{-p'/p}\,dx\biggr)^{p/p'}\leq A\end{equation*}
for some constant $A=A_p(\omega)$ independent of the choice of ball~$B$, where $1/p+1/p'=1$. By a corollary in \cite[Chapter~V, Section~4.2]{Ste93}, if $T$ is a Calder\'on-Zygmund operator and $\omega$ is an $A_p$ weight for some $1<p<\infty$, then $T$ is bounded from $L^p(\omega(x)\,dx)$ to itself.

As noted in \cite[formula~(2.5)]{BreM13} and the proof of~\cite[Lemma~2.3]{MitT06}, if $\Omega\subset\R^\dmn$ is a Lipschitz domain, then $\omega(x)=\dist(x,\partial\Omega)^{p-1-p\smooth}$ is an $A_p$ weight for any $1<p<\infty$ and any $0<\smooth<1$. Furthermore, observe that the constant $A=A_p(\omega)$ depends only on the Lipschitz character of~$\Omega$.
Thus, $\nabla^m\vec\Pi^{L_0}$ is bounded from $L^p(\omega(x)\,dx)$ to itself, as desired.
\end{proof}

We now must pass to weighted averaged spaces.
The following theorem was established in \cite{Bar16}; it is a straightforward consequence of \cite[Lemma~33]{Bar16} and the proof of \cite[Theorem~24]{Bar16}.
\begin{thm}
\label{thm:Meyers:interior} 
Let $L$ be an operator of the form~\eqref{eqn:divergence} of order~$2m$ associated to coefficients~$\mat A$ that satisfy the ellipticity conditions~\eqref{eqn:elliptic:bounded} and~\eqref{eqn:elliptic:everywhere}.
%Then there is some number $p^+=p^+_L>2$ depending only on $m$, the dimension $\dmn$, and the constants $\lambda$ and~$\Lambda$ in the bounds \eqref{eqn:elliptic:bounded} and~\eqref{eqn:elliptic:everywhere}, such that the following statement is true.

Let $x_0\in\R^\dmn$ and let $r>0$. Suppose that $\vec u\in \dot W^2_m(B(x_0,2r))$, $\arr H\in L^2(B(x_0,2r))$, and that $L\vec u=\Div_m\arr H$ in $B(x_0,2r)$.

If $0<p<2$, then
\begin{align*}
%\label{eqn:Meyers}
\biggl(\fint_{B(x,r)}\abs{\nabla^m \vec  u}^2\biggr)^{1/2}
&\leq 
	C(p)
	\biggl(\fint_{B(x,2r)}\abs{\nabla^m  \vec u}^p\biggr)^{1/p} 
	%\\&\qquad\nonumber
	+ C(p) \biggl(\fint_{B(x,2r)}\abs{\arr H}^2\biggr)^{1/2}
\end{align*}
for some constant $C(p)$ depending only on $p$ and the standard parameters.
\end{thm}

This theorem allows us to bound $\vec\Pi^{L_0}$ on $L^{p,\smooth}_{av}(\Omega)$ for $1<p\leq 2$.

\begin{lem}\label{lem:polyharmonic:1:2}
Let $L_0$ and $\Omega$ be as in Lemma~\ref{lem:polyharmonic:bounded}. Let $0<\smooth<1$ and $1<p\leq 2$. Then $\vec\Pi^{L_0}$ extends to an operator that is bounded ${L^{p,\smooth}_{av}(\R^\dmn\setminus\partial\Omega)}\mapsto {\widetilde W^{p,\smooth}_{av}(\partial\Omega)}$.
%\begin{equation*}\doublebar{\nabla^m\vec\Pi^{L_0}\arr H}_{L^{p,\smooth}_{av}(\R^\dmn\setminus\partial\Omega)}
%\leq
%C\doublebar{\arr H}_{L^{p,\smooth}_{av}(\R^\dmn\setminus\partial\Omega)}
%.\end{equation*}
\end{lem}

\begin{proof}%[Proof of Lemma~\ref{lem:polyharmonic:1:2}]
Let $\arr H\in L^2(\R^\dmn)\cap L^{p,\smooth}_{av}(\R^\dmn\setminus\partial\Omega)$.
Divide $\R^\dmn\setminus\partial\Omega$ into a grid $\mathcal{G}$ of Whitney cubes, as in the norm~\eqref{eqn:L:norm:whitney}.
By Theorem~\ref{thm:Meyers:interior} with $\vec u=\vec\Pi^{L_0}\arr H$, we have that
\begin{multline*}
\sum_{Q\in\mathcal{G}} \biggl(\fint_Q \abs{\nabla^m \vec\Pi^{L_0}\arr H}^2\biggr)^{p/2}\ell(Q)^{\dmnMinusOne+p-p\smooth}
\\\leq
C(p,\eta)\sum_{Q\in\mathcal{G}} \biggl(\fint_{\eta Q} \abs{\nabla^m \vec\Pi^{L_0}\arr H}^p\biggr)\ell(Q)^{\dmnMinusOne+p-p\smooth}
\\+
C(p,\eta)\sum_{Q\in\mathcal{G}} \biggl(\fint_{\eta Q} \abs{\arr H}^2\biggr)^{p/2}\ell(Q)^{\dmnMinusOne+p-p\smooth}
\end{multline*}
for any $\eta>1$, where $\eta Q$ is the cube concentric to $Q$ with side-length $\eta\ell(Q)$. If $\eta-1$ is small enough, then  $\dist(x,\partial\Omega)\approx \ell(Q)$ whenever $x\in\eta Q$ and $Q\in\mathcal{G}$, and furthermore if $x\in\R^\dmn\setminus\partial\Omega$ then $x\in\eta Q$ for at most $C$ cubes $Q\in\mathcal{G}$. Thus,
\begin{multline*}
\sum_{Q\in\mathcal{G}} \biggl(\fint_Q \abs{\nabla^m \vec\Pi^{L_0}\arr H}^2\biggr)^{p/2}\ell(Q)^{\dmnMinusOne+p-p\smooth}
\\\leq
C(p,\eta)\int_{\R^\dmn} \abs{\nabla^m \vec\Pi^{L_0} \arr H(x)}^p \dist(x,\partial\Omega)^{p-1-p\smooth}\,dx
\\+
C(p,\eta)\sum_{Q\in\mathcal{G}} \biggl(\fint_{\eta Q} \abs{\arr H}^2\biggr)^{p/2}\ell(Q)^{\dmnMinusOne+p-p\smooth}
\end{multline*}
By Lemma~\ref{lem:polyharmonic:Ap}, the norm~\eqref{eqn:L:norm:whitney}, and H\"older's inequality, we have that
\begin{equation*}\doublebar{\nabla^m\vec\Pi^{L_0}\arr H}_{L^{p,\smooth}_{av}(\Omega)}
\leq C(p)\doublebar{\arr H}_{L^{p,\smooth}_{av}(\Omega)}\end{equation*}
for any $0<\smooth<1$ and $1<p\leq 2$.

By density $\nabla^m\vec\Pi^{L_0}$ extends to an operator bounded from ${L^{p,\smooth}_{av}(\R^\dmn\setminus\partial\Omega)}$ to itself. If $\arr H\in L^2(\R^\dmn)$, then $\vec\Pi^{L_0}\in \dot W^2_m(\R^\dmn)\subset \dot W^1_{m,loc}(\R^\dmn)$. Furthermore, by Lemma~\ref{lem:L:L1}, $\nabla^m\vec\Pi^{L_0}$ is bounded from ${L^{p,\smooth}_{av}(\R^\dmn\setminus\partial\Omega)}$ to $L^1(K)$ for any compact set $K\subset\R^\dmn$, and so by density $\vec\Pi^{L_0}$ extends to a bounded operator ${L^{p,\smooth}_{av}(\R^\dmn\setminus\partial\Omega)}\mapsto \widetilde W^{p,\smooth}_{m,av}(\partial\Omega)$.
\end{proof}

We now consider the case $p\leq 1$.

\begin{lem}\label{lem:polyharmonic:<1}
Let $L_0$ and $\Omega$ be as in Lemma~\ref{lem:polyharmonic:bounded}. Let $0<\smooth<1$ and $\pmin<p\leq 1$. %, and let $\arr H\in L^{p,\smooth}_{av}(\R^\dmn\setminus\partial\Omega)$ be compactly supported in $\R^\dmn\setminus\partial\Omega$.
Then $\vec\Pi^{L_0}$ extends to a bounded operator ${L^{p,\smooth}_{av}(\R^\dmn\setminus\partial\Omega)}\mapsto \widetilde W^{p,\smooth}_{m,av}(\partial\Omega)$.
\end{lem}

\begin{proof}%[Proof of Lemma~\ref{lem:polyharmonic:<1}]
Let $\arr H\in L^{p,\smooth}_{av}(\R^\dmn\setminus\partial\Omega)$ be compactly supported in $\R^\dmn\setminus\partial\Omega$.
Again, let $\mathcal{G}$ be a grid of dyadic Whitney cubes in~$\R^\dmn\setminus\partial\Omega$. Choose some $Q\in\mathcal{G}$. 
Because $\nabla^m\vec\Pi^{L_0}$ is bounded on $L^2(\R^\dmn)$, we have that
\begin{multline*}\smash{\sum_{\substack{R\in\mathcal{G}\\\dist(R,Q)=0}}}
\biggl(\fint_R \abs{\nabla^m\vec \Pi^{L_0}(\1_Q \arr H)}^2\biggr)^{p/2} \ell(R)^{\dmnMinusOne+p-p\smooth}
\\\leq C \ell(Q)^{\dmnMinusOne+p-p\smooth-\pdmn p/2}\doublebar{\nabla^m\vec \Pi^{L_0}(\1_Q \arr H)}_{L^2(\R^\dmn)}^p
\\\leq 
C \ell(Q)^{\dmnMinusOne+p-p\smooth} \biggl(\fint_Q \abs{\arr H}^2\biggr)^{p/2}.
\end{multline*}
We seek to bound ${\nabla^m\vec \Pi^{L_0}(\1_Q \arr H)}$ far from~$Q$.
By the bound~\eqref{eqn:fundamental:decay} on the fundamental solution, if $\dist(x,Q)>0$ then
\begin{equation*}%\label{eqn:Pi:pointwise}
\abs{\nabla^m\vec \Pi^{L_0}(\1_Q \arr H)(x)} \leq \biggl(\frac{\ell(Q)}{\dist(x,Q)}\biggr)^\dmn \biggl(\fint_Q \abs{\arr H}^2\biggr)^{1/2}.\end{equation*}
Thus, by Lemma~\ref{lem:L:L-infinity}, and because $p>\pmin$,
\begin{equation*}\doublebar{\nabla^m\vec \Pi^{L_0}(\1_Q \arr H)}_{L^{p,\smooth}_{av}(\R^\dmn\setminus\partial\Omega)}
\leq
C \ell(Q)^{\pdmnMinusOne/p+1-\smooth} \biggl(\fint_Q \abs{\arr H}^2\biggr)^{1/2}
.\end{equation*}
Because $p\leq 1$, we have that
\begin{equation*}\doublebar{\nabla^m\vec \Pi^{L_0}\arr H}_{L^{p,\smooth}_{av}(\R^\dmn\setminus\partial\Omega)}^p
\leq \sum_{Q\in\mathcal{G}}\doublebar{\nabla^m\vec \Pi^{L_0}(\1_Q \arr H)}_{L^{p,\smooth}_{av}(\R^\dmn\setminus\partial\Omega)}^p
\end{equation*}
and so by the norm~\eqref{eqn:L:norm:whitney} we have the bound \begin{equation*}\doublebar{\nabla^m\vec \Pi^{L_0}\arr H}_{L^{p,\smooth}_{av}(\R^\dmn\setminus\partial\Omega)}\leq C\doublebar{\arr H}_{L^{p,\smooth}_{av}(\R^\dmn\setminus\partial\Omega)}.\end{equation*}
Again $\vec\Pi^{L_0}$ extends by density to a bounded operator ${L^{p,\smooth}_{av}(\R^\dmn\setminus\partial\Omega)}\mapsto \widetilde W^{p,\smooth}_{m,av}(\partial\Omega)$, as desired.
\end{proof}

Finally, we turn to the case $2<p\leq \infty$. We seek to use the duality relation~\eqref{eqn:L:dual}. We need only bound the adjoint to $(\nabla^m\vec\Pi^L)^*$ on $L^{p',1-\smooth}_{av}(\Omega)$. We will use the following formula from \cite{Bar16}.

\begin{lem}[{\cite[Lemma~42]{Bar16}}] \label{lem:newton:adjoint}
Suppose $L$ is an operator of the form~\eqref{eqn:divergence} associated to coefficients~$\mat A$ that satisfy the bounds~\eqref{eqn:elliptic:bounded} and~\eqref{eqn:elliptic:everywhere}.

Let $(A^*)^{jk}_{\alpha\beta}=\overline{A^{kj}_{\beta\alpha}}$ and let $L^*$ be the associated elliptic operator.

Then the adjoint $(\nabla^m\vec\Pi^L)^*$ to the operator $\nabla^m\vec\Pi^L$ is $\nabla^m\vec\Pi^{L^*}$.
\end{lem}

Thus, $\arr T=\nabla^m\vec\Pi^{L_0}$ extends to a bounded operator on ${L^{p,\smooth}_{av}(\R^\dmn\setminus\partial\Omega)}$. It remains to show that $\arr T\arr H$ is in fact the gradient of a $\dot W^1_{m,loc}(\R^\dmn)$-function. If $p<\infty$, this is true by density as usual. 

If $p=\infty$, this is true by weak density. That is, let $\arr H\in L^{\infty,\smooth}_{av}(\R^\dmn\setminus\partial\Omega)$ and let $\arr H_n\in L^{\infty,\smooth}_{av}(\R^\dmn\setminus\partial\Omega)\cap L^2(\R^\dmn)$ converge weakly to $\arr H$ (in $L^{\infty,\smooth}_{av}(\R^\dmn\setminus\partial\Omega)= (L^{1,1-\smooth}_{av}(\R^\dmn\setminus\partial\Omega))^*$). We may require $\doublebar{\arr H_n}_{L^{\infty,\smooth}_{av}(\R^\dmn\setminus\partial\Omega)} \leq \doublebar{\arr H}_{L^{\infty,\smooth}_{av}(\R^\dmn\setminus\partial\Omega)}$.
By \cite[Lemma~3.7]{Bar16pB} and the Poincar\'e inequality, we have that $\vec\Pi^{L_0}\arr H_n$ is locally in $\dot W_m^p(\R^\dmn)$ for some $p>1$. By the Poincar\'e inequality, $\vec \Pi^{L_0}\arr H_n-P_{n,B}\in L^p(B)$ for all balls $B\subset\R^\dmn$, where $P_{n,B}$ is an appropriate polynomial of degree at most $m-1$. Then $\vec \Pi^{L_0}\arr H_n-P_{n,B}$ is a bounded sequence in a reflexive Banach space, and so has a weak limit $\vec F$. It is elementary to show that $\arr T\arr H=\nabla^m\vec F$ (in the sense of weak derivatives); thus, the proof of Lemma~\ref{lem:polyharmonic:bounded} is complete.

\subsection{Interpolation}
\label{sec:interpolation}

In this subsection we will discuss interpolation theory and its application to the weighted averaged spaces $L^{p,\smooth}_{m,av}(\Omega)$, weighted averaged Sobolev spaces $\dot W^{p,\smooth}_{m,av}(\Omega)$, and boundary spaces $\dot W\!A^p_{m-1,\smooth}(\partial\Omega)$.

We will use interpolation theory, and in particular stability of invertibility on interpolation scales, to prove Theorem~\ref{thm:neighborhood}.
We will also use interpolation to prove Lemma~\ref{lem:interpolation}; recall that we used this lemma in Section~\ref{sec:A0:introduction} to establish well posedness results.

We refer the reader to the classic reference \cite{BerL76} for an extensive background on interpolation theory; in this section we will provide some definitions and summarize a few results.

Following \cite{BerL76}, we say that two quasi-normed vector spaces $A_0$, $A_1$ are \emph{compatible} if there is a Hausdorff topological vector space $\mathfrak{A}$ such that $A_0\subset\mathfrak{A}$, $A_1\subset\mathfrak{A}$. Then $A_0\cap A_1$ and $A_0+A_1$ may be defined in the natural way.

We will use two interpolation functors, the real method of Lions and Peetre, and the complex interpolation method of Lions, Calder\'on and Krejn.
We refer the reader to \cite{BerL76} for a precise definition of these interpolation functors. Loosely speaking, if $A_0$ and $A_1$ are compatible, these functors produce spaces that in some sense lie between $A_0$ and~$A_1$.
More precisely, for any number $\sigma$ with $0<\sigma<1$, any number $r$ with $0<r\leq\infty$, and any compatible quasi-normed spaces $A_0$ and $A_1$, the real interpolation functor produces a space  $(A_0,A_1)_{\sigma,r}$ contained in $A_0+A_1$ and containing $A_0\cap A_1$, and if $A_0$ and $A_1$ are normed vector spaces then the complex interpolation functor produces a (possibly different) function space $[A_0,A_1]_\sigma$ also contained in $A_0+A_1$ and containing $A_0\cap A_1$.

The following is a fundamental and very useful result of interpolation theory.
Let $A_0$, $A_1$ and $B_0$, $B_1$ be two compatible pairs. Then by \cite[Theorems~3.11.2 and 4.1.2]{BerL76}, we have that if $T:A_0+A_1\mapsto B_0+B_1$ is a linear operator such that $T(A_0)\subseteq B_0$ and $T(A_1)\subseteq B_1$, then $T$ is bounded on appropriate interpolation spaces, with
\begin{align}
\label{eqn:interpolation:real}
\doublebar{T}_{(A_0,A_1)_{\sigma,r}\mapsto(B_0,B_1)_{\sigma,r}} &\leq 
\doublebar{T}_{A_0\mapsto B_0}^{1-\sigma}\doublebar{T}_{A_1\mapsto B_1}^{\sigma},
\\
\label{eqn:interpolation:complex}
\doublebar{T}_{[A_0,A_1]_{\sigma}\mapsto[B_0,B_1]_{\sigma}} &\leq 
\doublebar{T}_{A_0\mapsto B_0}^{1-\sigma}\doublebar{T}_{A_1\mapsto B_1}^{\sigma}
\end{align}
for $0<\sigma<1$ and $0<r\leq\infty$.

In order to use this result, we will need to identify the spaces $[A_0,A_1]_\sigma$ for various known spaces $A_j$, for instance, in the case where $A_j=L^{p_j,\smooth_j}_{av}(\Omega)$.
We begin with some known interpolation properties for sequence spaces. Let $\mathcal{G}$ be a grid of Whitney cubes in~$\Omega$ and recall the norm \eqref{eqn:L:norm:whitney}. Let $Q_0$ be the unit cube, and define the sequence space $\ell^{p,\smooth}_\Omega=\ell^{p,\smooth}_\Omega(L^2(Q_0))$ by
\begin{equation*}\ell^{p,\smooth}_\Omega(L^2(Q_0)) = \Bigl\{\bigl(\arr H_Q\bigr)_{Q\in\mathcal{G}}: \Bigl(\sum_{Q\in\mathcal{G}} \doublebar{\arr H_Q}_{L^2(Q_0)}^p \ell(Q)^{\dmnMinusOne+p-p\smooth}\Bigr)^{1/p}<\infty\Bigr\}\end{equation*}
with the natural norm.
%Specifically, let $A$ be a Banach space, and let
%\begin{equation*}\ell^p(A) = \Bigl\{(a_k)_{k=1}^\infty: a_k\in A,\>\Bigl(\sum_{k=1}^\infty \doublebar{a_k}_A^p\Bigr)^{1/p}<\infty\Bigr\}.\end{equation*}
%If $p=\infty$ then instead of summing we take an appropriate supremum.
%Let $0<\sigma<1$, and let $1/p_\sigma=(1-\sigma)/p_0+\sigma/p_1$.
%By \cite[Theorems~5.6.1 and~5.6.3]{BerL76}, we see that if $0<p_j\leq \infty$, then
%\begin{equation*}(\ell^{p_0}(A),\ell^{p_0}(A))_{\sigma,p_\sigma} = \ell^{p_\sigma}(A)\end{equation*}
%with equivalent norms, and if $1\leq p_j\leq \infty$ then 
%\begin{equation*}[\ell^{p_0}(A),\ell^{p_0}(A)]_{\sigma} = \ell^{p_\sigma}(A)\end{equation*}
%with equal norms. We will need the equality of norms in the proof of Theorem~\ref{thm:neighborhood}, and so in this case the complex method is preferable; however, we will also want to study boundary value problems with solutions in $\dot W^{p,\smooth}_{m,av}(\Omega)$ for $\pmin<p<1$ as well as $1\leq p\leq\infty$, and only the real interpolation method is available in this case.

%Specifically, let $A$ be a Banach space, and let $S^j=(S_k^j)_{k=1}^\infty$, $j=0$, $1$ be two sequences of positive real numbers. 
Let $0<p_j<\infty$ and let $\smooth_j\in \R$. Let $0<\sigma<1$, let $1/p_\sigma=(1-\sigma)/p_0+\sigma/p_1$, and let $\smooth_\sigma=(1-\sigma)\smooth_0+\sigma\smooth_1$. 
%Define
%\begin{equation*}\ell^p_{S}(A) = \Bigl\{(a_k)_{k=1}^\infty: a_k\in A,\>\Bigl(\sum_{k=1}^\infty (S_k)^{p} \doublebar{a_k}_A^p\Bigr)^{1/p}<\infty\Bigr\}\end{equation*}
%with the obvious norm.

By \cite[Theorem~5.5.1]{BerL76}, %if $0<p_j<\infty$, then
%Examining the proof of \cite[Theorem~5.6.1]{BerL76}, except not, because this requires that we sum over all $Q$ a quantity that might not converge
\begin{equation}\label{eqn:l-l:interpolation:real} 
(\ell^{p_0,\smooth_0}_\Omega,\ell^{p_1,\allowbreak \smooth_1}_\Omega)_{\sigma,\allowbreak p_\sigma} = \ell^{p_\sigma,\smooth_\sigma}_\Omega
\end{equation}
with equivalent norms.
%Examining the proof of \cite[Theorem~5.6.3]{BerL76}, we see that if $1\leq p_j\leq \infty$, then
By \cite[Theorem~5.5.3]{BerL76},  if in addition $p_j\geq 1$, then
\begin{equation}\label{eqn:l-l:interpolation:complex} 
[\ell^{p_0,\smooth_0}_\Omega,\ell^{p_1,\smooth_1}_\Omega]_{\sigma} = \ell^{p_\sigma,\smooth_\sigma}_\Omega
\end{equation}
with equal norms.

We will use the following lemma to extend these results from $\ell^{p,\smooth}_\Omega$ to $L^{p,\smooth}_{av}(\Omega)$.
This is essentially \cite[Theorem~6.4.2]{BerL76} and \cite[Theorem~1.2.4]{Tri78}; see also \cite[Section~3]{Pee71}.

\begin{lem}\label{lem:retract} Suppose that $(A_0,A_1)$ and $(B_0,B_1)$ are two compatible couples, and that there are linear operators $\mathcal{I}:A_0+A_1\mapsto B_0+B_1$ and $\mathcal{P}:B_0+B_1\mapsto A_0+A_1$ such that $\mathcal{P}\circ\mathcal{I}$ is the identity operator on $A_0+A_1$, and such that $\mathcal{I}:A_0\mapsto B_0$, $\mathcal{I}:A_1\mapsto B_1$, $\mathcal{P}:A_0\mapsto B_0$, and $\mathcal{P}:A_1\mapsto B_1$ are all bounded operators.

If $0<\sigma<1$ and $0<r\leq\infty$, then
\begin{equation}\label{eqn:retract:1}
[A_0,A_1]_{\sigma}=\mathcal{P}([B_0,B_1]_{\sigma})
\quad\text{and}\quad
(A_0,A_1)_{\sigma,\allowbreak r}=\mathcal{P}((B_0,\allowbreak B_1)_{\sigma,r})
\end{equation}
with equivalent norms; that is,
\begin{gather*}
%\frac{1}{\doublebar{\mathcal{I}}_0^{1-\sigma} \doublebar{\mathcal{I}}_1^{\sigma}}
%\doublebar{\mathcal{I}a}_{[B_0,B_1]_{\sigma}}
%&\leq
\frac{1}{\doublebar{\mathcal{I}}_\sigma}
\doublebar{\mathcal{I}a}_{[B_0,B_1]_{\sigma}}
%\\&
\leq
\doublebar{a}_{[A_0,A_1]_{\sigma}} 
%\\&
\leq \bigl(\doublebar{\mathcal{P}}_\sigma\bigr)\doublebar{\mathcal{I}a}_{[B_0,B_1]_{\sigma}}
%\leq \bigl(\doublebar{\mathcal{P}}_0^{1-\sigma} \doublebar{\mathcal{P}}_1^{\sigma} \bigr) \doublebar{\mathcal{I}a}_{[B_0,B_1]_{\sigma}}
,\\
\frac{1}{\doublebar{\mathcal{I}}_{\sigma,r}}
\doublebar{\mathcal{I}a}_{(B_0,B_1)_{\sigma,r}}
%\\&
\leq
\doublebar{a}_{(A_0,A_1)_{\sigma,r}} 
%\\&
\leq \bigl(\doublebar{\mathcal{P}}_{\sigma,r}\bigr)\doublebar{\mathcal{I}a}_{(B_0,B_1)_{\sigma,r}}
\end{gather*}
where $\doublebar{\mathcal{I}}_\sigma$ and  $\doublebar{\mathcal{P}}_\sigma$ denote operator norms between appropriate interpolation spaces.
\end{lem}

\begin{proof}
By the bound~\eqref{eqn:interpolation:complex},  $\mathcal{P}([B_0,B_1]_{\sigma})\subseteq [A_0,A_1]_{\sigma}$. Conversely, $[A_0,A_1]_{\sigma}=\mathcal{P}\mathcal{I}([A_0,A_1]_{\sigma})$ and $\mathcal{I}([A_0,A_1]_{\sigma})\subseteq [B_0,B_1]_{\sigma}$, and so $[A_0,A_1]_{\sigma}\subseteq \mathcal{P}([B_0,B_1]_\sigma)$. 
An analogous argument involving the bound~\eqref{eqn:interpolation:real} is valid for the real interpolation method.
The norm inequalities follow from the relation $a=\mathcal{P}\mathcal{I}a$ and the bound \eqref{eqn:interpolation:real} or~\eqref{eqn:interpolation:complex}.
\end{proof}

We now consider the spaces $L^{p,\smooth}_{av}(\Omega)$. Interpolation results for the spaces $L^{p,\smooth}_{av}(\R^\dmn_+)$
were established in \cite[Theorem~4.13]{BarM16A}; generalizing to arbitrary Lipschitz domains is straightforward. %, but we will need a bit more care to establish that the norms in these two spaces are equal rather than merely equivalent.

\begin{lem}\label{lem:L-L:interpolation}
Let $\Omega\subset\R^\dmn$ be a Lipschitz domain.
Let $\smooth_j\in\R$, $0<p_j<\infty$ and $0<\sigma<1$,  and let $1/p_\sigma=(1-\sigma)/p_0+\sigma/p_1$ and  $\smooth_\sigma=(1-\sigma)\smooth_0+\sigma\smooth_1$.
Then 
\begin{equation}\label{eqn:L-L:interpolation:real} 
(L^{p_0,\smooth_0}_{av}(\Omega),L^{p_1,\smooth_1}_{av}(\Omega))_{\sigma,p_\sigma} = 
L^{p_\sigma,\smooth_\sigma}_{av}(\Omega).
\end{equation}
If in addition $p_j\geq 1$, then
\begin{equation}\label{eqn:L-L:interpolation:complex} 
[L^{p_0,\smooth_0}_{av}(\Omega),L^{p_1,\smooth_1}_{av}(\Omega)]_{\sigma} = 
L^{p_\sigma,\smooth_\sigma}_{av}(\Omega).
\end{equation}
If we let the $L^{p,\smooth}_{av}(\Omega)$ norm be given by formula~\eqref{eqn:L:norm:whitney} rather than~\eqref{eqn:L:norm:2}, then we have equal norms in formula~\eqref{eqn:L-L:interpolation:complex}.
\end{lem}

\begin{proof}
Let $\mathcal{I}:L^{p,\smooth}_{av}(\Omega)\mapsto \ell^{p,\smooth}_\Omega$ be given by $(\mathcal{I}\arr H)_Q(x)=\arr H(x_Q+\ell(Q) x)$ for some appropriate $x_Q\in Q$. Observe that $\mathcal{I}$ is an isomorphism; let $\mathcal{P}$ be its inverse. Then the result follows by Lemma~\ref{lem:retract} and formulas~\eqref{eqn:l-l:interpolation:real} and~\eqref{eqn:l-l:interpolation:complex}.
\end{proof}

We remark that, in the $(\smooth,1/p)$-plane, the set of points $\{(\smooth_\sigma,1/p_\sigma):0<\sigma<1\}$ is the line segment connecting $(\smooth_0,1/p_0)$ and $(\smooth_1,1/p_1)$.

%We will eventually need to establish interpolation results for the weighted, averaged Sobolev spaces $\dot W_{m,av}^{p,\smooth}(\Omega)$. This may be done using Lemma~\ref{lem:retract} and formulas \eqref{eqn:l-l:interpolation:real} and~\eqref{eqn:l-l:interpolation:complex}, with the map $\mathcal{I}:\dot W_{m,av}^{p,\smooth}(\Omega)\mapsto L_{av}^{p,\smooth}(\Omega)$ given by $\mathcal{I}\vec F=\nabla^m\vec F$. We will need a globally defined, bounded left inverse~$\mathcal{P}$ to complete the argument. The most convenient choice of $\mathcal{P}$ is the Newton potential for the polyharmonic operator $\Delta^m$. We will define the Newton potential in Section~\ref{sec:green:newton}, and certain properties thereof will be useful to us; thus, we will defer interpolation results for the spaces $\dot W_{m,av}^{p,\smooth}(\Omega)$ until Section~\ref{sec:interpolation:W}.

We now use Lemmas~\ref{lem:retract} and~\ref{lem:L-L:interpolation} to produce interpolation results for other spaces.

\begin{lem}\label{lem:interpolation:W}
Let $\Omega\subset\R^\dmn$ be a Lipschitz domain with connected boundary.
Let $0<\smooth_j<1$, $0<p_j<\infty$ and $0<\sigma<1$,  and let  $p_\sigma$ and $\smooth_\sigma$ be as in Lemma~\ref{lem:L-L:interpolation}.

%Define the equivalence relation $\sim$ on $L^{p,\smooth}_{av}(\Omega)$ by
%\begin{equation}\label{eqn:sim}
%\arr H\sim\arr\Phi \text{ if } \int_\Omega \langle \nabla^m \vec\varphi,\arr H\rangle = \int_\Omega \langle \nabla^m \vec\varphi,\arr \Phi\rangle \text{ for all }\vec\varphi\in C^\infty_0(\R^\dmn)\end{equation}
%and let $L_{av}^{p,\smooth}(\R^\dmn\setminus\partial\Omega)/\sim$ be the quotient space of $L_{av}^{p,\smooth}(\R^\dmn\setminus\partial\Omega)$ under the relation~$\sim$.

If $\pmin[\smooth_j]<p_j<\infty$, then we have the real interpolation formulas
\begin{align}\label{eqn:W-W:interpolation:real} 
(\dot W_{m,av}^{p_0,\smooth_0}(\Omega), \dot W_{m,av}^{p_1,\smooth_1}(\Omega))_{\sigma,p_\sigma} &= 
\dot W_{m,av}^{p_\sigma,\smooth_\sigma}(\Omega)
,\\
\label{eqn:WA-WA:interpolation:real} 
(\dot W\!A^{p_0}_{m-1,\smooth_0}(\partial\Omega), \dot W\!A^{p_1}_{m-1,\smooth_1}(\partial\Omega))_{\sigma,p_\sigma} &= 
\dot W\!A^{p_\sigma}_{m-1,\smooth_\sigma}(\partial\Omega)
.\end{align}
If in addition $p_j\geq 1$,
then we have the complex interpolation formulas
\begin{align}\label{eqn:W-W:interpolation:complex} 
[\dot W_{m,av}^{p_0,\smooth_0}(\Omega), \dot W_{m,av}^{p_1,\smooth_1}(\Omega)]_{\sigma} &= 
\dot W_{m,av}^{p_\sigma,\smooth_\sigma}(\Omega)
,\\
\label{eqn:WA-WA:interpolation:complex} 
[\dot W\!A^{p_0}_{m-1,\smooth_0}(\partial\Omega), \dot W\!A^{p_1}_{m-1,\smooth_1}(\partial\Omega)]_{\sigma} &= 
\dot W\!A^{p_\sigma}_{m-1,\smooth_\sigma}(\partial\Omega)
.\end{align}
If $0<\smooth_j<1$ and $1\leq p_j<\infty$ for both $j=0$ and $j=1$,  and if $1<p_j<\infty$ for at least one of $j=0$ and $j=1$, then
\begin{equation}\label{eqn:W-W:interpolation:dual}
[(\dot W_{m,av}^{p_0,\smooth_0}(\Omega))^*, (\dot W_{m,av}^{p_1,\smooth_1}(\Omega))^*]_{\sigma} = 
(\dot W_{m,av}^{p_\sigma,\smooth_\sigma}(\Omega))^*
\end{equation}
where $(\dot W_{m,av}^{p,\smooth}(\Omega))^*$ is the dual space to~$\dot W_{m,av}^{p,\smooth}(\Omega)$.

% Svante Janson 1993:
% It is well known that even if we know the interpolation spaces of a certain couple of spaces, by the real or complex method, say, there is no general formula that enables us to directly obtain the interpolation spaces of a couple of subspaces or quotient spaces of a given couple. Indeed, there are examples that show that interpolation of subspaces (or quotient spaces) in general may be ill-behaved, see Triebel \cite{Tri70} and Wallst\'en \cite{Wal88}.

\end{lem}

\begin{proof}
Recall the space $\widetilde W_{m,av}^{p,\smooth}(\partial\Omega)$ of formula~\eqref{eqn:W:global}. We will use $\widetilde W_{m,av}^{p,\smooth}(\partial\Omega)$ as an intermediate space between $L^{p,\smooth}_{av}(\R^\dmn\setminus\partial\Omega)$ and~$\dot W^{p,\smooth}_{m,av}(\Omega)$.

Let $\mathcal{I}\vec F=\nabla^m\vec F$; then $\mathcal{I}:\widetilde W_{m,av}^{p,\smooth}\mapsto L^{p,\smooth}_{av}(\R^\dmn\setminus\partial\Omega)$ 
is bounded and one-to-one. Let $\mathcal{P}\arr H = \vec\Pi^{L_0}(\mat A\arr H)$ for some constant coefficients~$L_0$. By Lemma~\ref{lem:polyharmonic:bounded},  $\mathcal{P}$ is bounded $L^{p,\smooth}_{av}(\R^\dmn\setminus\partial\Omega) \mapsto \widetilde W_{m,av}^{p,\smooth}(\partial\Omega)$.

Recall that if $\arr H\in L^2(\R^\dmn)$, then $\vec \Pi^L\arr H$ is the \emph{unique} solution to formula~\eqref{dfn:newton}; we then have that
\begin{align*}
%\label{eqn:newton:identity}
\vec \Pi^L(\mat A \nabla^m\vec F)=\vec F &\text{ for all $\vec F\in\dot W^2_{m}(\R^\dmn)$}
.\end{align*}
By density of smooth functions (see \cite[Theorem~3.15]{Bar16pB}) this is true for all $\vec F\in \widetilde W_{m,av}^{p,\smooth}(\partial\Omega)$.
Thus $\mathcal{P}\circ\mathcal{I}$ is the identity, and by Lemma~\ref{lem:polyharmonic:bounded} and the definition of $\widetilde W_{m,av}^{p,\smooth}(\partial\Omega)$ we have that $\mathcal{I}$ and $\mathcal{P}$ are bounded whenever $0<\smooth<1$ and $\pmin<p\leq\infty$.
Thus, by Lemma~\ref{lem:retract} and formulas~\eqref{eqn:L-L:interpolation:real} and~\eqref{eqn:L-L:interpolation:complex}, with appropriate restrictions on~$p_j$,
\begin{equation*}%\label{eqn:W-W:interpolation}
(\widetilde W_{m,av}^{p_0,\smooth_0}(\partial\Omega), \widetilde W_{m,av}^{p_1,\smooth_1}(\partial\Omega))_{\sigma,p_\sigma} = 
[\widetilde W_{m,av}^{p_0,\smooth_0}(\partial\Omega), \widetilde W_{m,av}^{p_1,\smooth_1}(\partial\Omega)]_{\sigma} = 
\widetilde W_{m,av}^{p_\sigma,\smooth_\sigma}(\partial\Omega),
\end{equation*}
where $p_\sigma$ and $\smooth_\sigma$ are as in Lemma~\ref{lem:L-L:interpolation}.

We now pass to the familiar weighted spaces $\dot W_{m,av}^{p,\smooth}(\Omega)$.
Let $\mathcal{P}: \widetilde W_{m,av}^{p,\smooth}(\partial\Omega) \mapsto  \dot W_{m,av}^{p,\smooth}(\Omega)$ be simply the restriction map; $\mathcal{P}$ is clearly bounded. Let 
\begin{equation*}\mathcal{I}\vec F = \mathcal{E}^0_\Omega\vec F+\mathcal{E}^0_{\R^\dmn\setminus\overline\Omega} \Ext^{\R^\dmn\setminus\overline\Omega} \Tr_{m-1}^{\Omega} \vec F,\end{equation*}
where $\Ext$ is the extension operator of \cite[Theorem~4.1]{Bar16pB} and $\mathcal{E}^0_\Omega$ denotes extension by zero. An examination of the proof of \cite[Theorem~4.1]{Bar16pB}  reveals that if $\Omega$ is a Lipschitz domain with connected boundary then $\Ext$ is a bounded linear operator; by \cite[Theorem~5.1]{Bar16pB}, $\Tr_{m-1}^\Omega$ is also bounded, and so $\mathcal{I}$ is bounded. Clearly $\mathcal{P}\circ\mathcal{I}\vec F=\vec F$, and so by Lemma~\ref{lem:retract}, the interpolation formulas~\eqref{eqn:W-W:interpolation:real} and~\eqref{eqn:W-W:interpolation:complex} are valid.

To establish the interpolation results \eqref{eqn:WA-WA:interpolation:real} and~\eqref{eqn:WA-WA:interpolation:complex} for Whitney spaces, let $\mathcal{I}=\Ext^{\Omega}$ and $\mathcal{P}=\Tr_{m-1}^{\Omega}$. By \cite[Theorems~4.1 and~5.1]{Bar16pB}, these operators are bounded and $\mathcal{P}\circ\mathcal{I}$ is the identity, and so Lemma~\ref{lem:retract} yields the desired results.

Finally, by \cite[Corollary~4.5.2]{BerL76}, if $(A_0,A_1)$ is a compatible couple of Banach spaces, and if at least one of $A_0$ and $A_1$ is reflexive, then the dual space $[A_0,A_1]_\sigma^*$ to $[A_0,A_1]_\sigma$ is $[A_0^*,A_1^*]_\sigma$. We may identify $\dot W^{p,\smooth}_{m,av}(\Omega)$ with a closed subspace of $L^{p,\smooth}_{av}(\Omega)$; thus, if $1<p<\infty$ then $\dot W^{p,\smooth}_{m,av}(\Omega)$ is reflexive, and so
\begin{equation*}
[(\dot W_{m,av}^{p_0,\smooth_0}(\Omega))^*, (\dot W_{m,av}^{p_1,\smooth_1}(\Omega))^*]_{\sigma} = 
(\dot W_{m,av}^{p_\sigma,\smooth_\sigma}(\Omega))^*
\end{equation*}
as desired.
\end{proof}

We will conclude this subsection by proving Lemma~\ref{lem:interpolation}.

\begin{proof}[Proof of Lemma~\ref{lem:interpolation}]
By Lemma~\ref{lem:zero:boundary}, we may consider only the case of homogeneous boundary data (that is, $\arr f=0$ and $\arr g=0$).

Define the operator $T$ as follows. If $\arr H\in L^{p_0,\smooth_0}_{av}(\Omega)$ or $\arr H\in L^{p_1,\smooth_1}_{av}(\Omega)$, let $T\arr H=\vec u$, where $\vec u$ is the solution to the Dirichlet problem~\eqref{eqn:interpolation:Dirichlet:0} or~\eqref{eqn:interpolation:Dirichlet:1} or Neumann problem \eqref{eqn:interpolation:Neumann:0} or~\eqref{eqn:interpolation:Neumann:1} with data $\arr\Phi=\arr H$ or $\arr\Psi=\arr H$. 

Then $T$ is a bounded linear operator $L^{p_j,\smooth_j}_{av}(\Omega)\mapsto \dot W^{p_j,\smooth_j}_{m,av}(\Omega)$ for $j=0$ and~$j=1$.
By the compatibility condition, we have that $T$ extends to a well-defined linear operator $L^{p_0,\smooth_0}_{av}(\Omega)+L^{p_1,\smooth_1}_{av}(\Omega) \mapsto \dot W^{p_0,\smooth_0}_{m,av}(\Omega)+\dot W^{p_1,\smooth_1}_{m,av}(\Omega)$. Thus, by formulas~\eqref{eqn:interpolation:real}, \eqref{eqn:L-L:interpolation:real}, and~\eqref{eqn:W-W:interpolation:real}, %(if $p_j<\infty$) or by formulas~\eqref{eqn:interpolation:complex}, \eqref{eqn:L-L:interpolation:complex}, and~\eqref{eqn:W-W:interpolation:complex} (if $p_j\geq 1$), 
$T$ is a bounded linear operator $L^{p_\sigma,\smooth_\sigma}_{av}(\Omega)\mapsto \dot W^{p_\sigma,\smooth_\sigma}_{m,av}(\Omega)$ for any $0<\sigma<1$.

Thus, for each $\arr H\in L^{p_\sigma,\smooth_\sigma}_{av}(\Omega)$ there exists a solution to the boundary value problem \eqref{eqn:interpolation:Dirichlet} or~\eqref{eqn:interpolation:Neumann}.

We must now establish uniqueness. %If $\pdmnMinusOne/p_0-\smooth_0=\pdmnMinusOne/p_1-\smooth_1$, then $\pdmnMinusOne/p_\sigma-\smooth_\sigma$ is constant and uniqueness follows by Corollary~\ref{cor:unique:extrapolate}.
Suppose that $p_j>1$.
By Theorems~\ref{thm:exist:unique} and~\ref{thm:unique:exist}, we have that the corresponding boundary value problems for $p_j'$, $\smooth_j'$, with $p'$, $\smooth'$ as in Theorem~\ref{thm:unique:exist}, are well posed. Then $(p')_\sigma=(p_\sigma)'$ and $(\smooth')_\sigma=(\smooth_\sigma)'$, and so by the same interpolation argument, we have that the boundary value problem \eqref{eqn:Dirichlet:p:smooth} or~\eqref{eqn:Neumann:p:smooth},  with $p=p_\sigma'$ and $\smooth=\smooth_\sigma'$, are well posed. Another application of Theorem~\ref{thm:exist:unique} completes the proof.
\end{proof}

\subsection{Proof of Theorem~\ref*{thm:neighborhood}}
\label{sec:neighborhood:proof}

It is a well known consequence of the Lax-Milgram lemma that, if $L$ is a divergence form elliptic operator, then the %Dirichlet problem with homogeneous boundary data
%\begin{equation*}L\vec u=\Div_m \arr H \text{ in }\Omega,\quad \vec u \in \mathring W^2_m(\Omega),\quad \doublebar{\nabla^m \vec u}_{L^2(\Omega)} \leq \frac{1}{\lambda} \doublebar{\arr H}_{L^2(\Omega)}\end{equation*}
%and the 
Neumann problem  with homogeneous boundary data
\begin{equation*}L\vec u=\Div_m \arr H \text{ in }\Omega,\quad \M_{\mat A,\arr H}^\Omega\vec u =0,\quad \doublebar{\nabla^m \vec u}_{L^2(\Omega)} \leq \frac{1}{\lambda} \doublebar{\arr H}_{L^2(\Omega)}\end{equation*}
is well posed. %Here $\mathring W^2_m(\Omega)$ is the completion in $\dot W^2_m(\Omega)$ of the set of all smooth, compactly supported functions with compact support. 
%
%If $\Omega$ is a Lipschitz domain with connected boundary, then $\mathring W^2_m(\Omega)=\{\vec v\in \dot W^2_m(\Omega):\Tr_{m-1}^\Omega\vec v=0\}$. Furthermore, by the trace and extension theorems of Section~\ref{sec:trace}, we may extend to well posedness of the corresponding problems with inhomogeneous boundary values.
By formula~\eqref{eqn:L2:2:half}, we have that the Neumann problem
\begin{gather}%L\vec u=\Div_m \arr H \text{ in }\Omega,\quad \Tr_{m-1}^\Omega\vec u =0,\>\> \doublebar{\vec u}_{\dot W^{p,\smooth}_{m,av}(\Omega)} \leq C \doublebar{\arr H}_{L^{p,\smooth}_{av}(\Omega)} %+ C\doublebar{\arr f}_{\dot W\!A^p_{m-1,\smooth}(\partial\Omega)}
%,\\
\label{eqn:Neumann:energy}
L\vec u=\Div_m \arr H \text{ in }\Omega,\quad \M_{\mat A,\arr H}^\Omega\vec u =0,\quad \doublebar{\vec u}_{\dot W^{p,\smooth}_{m,av}(\Omega)} \leq C \doublebar{\arr H}_{L^{p,\smooth}_{av}(\Omega)} %+ C\doublebar{\arr g}_{\dot N\!A^p_{m-1,\smooth-1}(\partial\Omega)}
\end{gather}
is well posed for $p=2$ and $\smooth=1/2$, in any Lipschitz domain with connected boundary.
We must show that well posedness still holds if $\abs{p-2}$ and $\abs{\smooth-1/2}$ are small but not necessarily zero.

We will use the following lemma. This lemma appeared originally as {\cite[Lemma~7]{Sne74}}; see also \cite[Section~2]{KalM98}.
\begin{lem}
\label{lem:invertible:stable}
Let $(A_0,A_1)$ and $(B_0,B_1)$ be two compatible couples. For each $0<\sigma<1$, let $A_\sigma=[A_0,A_1]_\sigma$ and $B_\sigma=[B_0,B_1]_\sigma$, where $[\,\cdot\,,\,\cdot\,]_\sigma$ is the complex interpolation functor of Lions, Calder\'on, and Krejn.

Suppose that $T:A_0+A_1\mapsto B_0+B_1$ is a linear operator with $T(A_0)\subseteq B_0$ and $T(A_1)\subseteq B_1$; by the bound~\eqref{eqn:interpolation:complex}, $T$ is bounded $A_\sigma\mapsto B_\sigma$ for any $0<\sigma<1$.

Suppose that for some $\sigma_0$ with $0<\sigma_0<1$, we have that $T:A_{\sigma_0}\mapsto B_{\sigma_0}$ is invertible. Then  there is some $\varepsilon>0$ such that $T:A_\sigma\mapsto B_\sigma$ is invertible for all $\sigma$ with $\sigma_0-\varepsilon<\sigma<\sigma_0+\varepsilon$. 
\end{lem}

Thus, our goal is to reframe well posedness as invertibility of some bounded linear operator on an interpolation scale.

If $\vec u \in \dot W^{p,\smooth}_{m,av}(\Omega)$, we will let $T\vec u$ be the element of $(\dot W^{p',\smooth'}_{m,av}(\Omega))^*$ given by 
$T\vec u(\vec\varphi)=\langle\mat A\nabla^m \vec u, \nabla^m\vec \varphi \rangle_\Omega$. By the bound~\eqref{eqn:elliptic:bounded} on $\mat A$, the duality relation~\eqref{eqn:L:dual}, and the definition of $\dot W^{p,\smooth}_{m,av}(\Omega)$, if $1<p<\infty$ and $0<\smooth<1$ then $T$ is bounded $\dot W^{p,\smooth}_{m,av}(\Omega)\mapsto (\dot W^{p',\smooth'}_{m,av}(\Omega))^*$.

If $\arr H\in L^{p,\smooth}_{av}(\Omega)$, then we may identify $\arr H$ with the element of $(\dot W^{p',\smooth'}_{m,av}(\Omega))^*$ given by $\vec\varphi\mapsto \langle \arr H,\nabla^m\vec\varphi\rangle_\Omega$. By the definition~\eqref{dfn:Neumann} of Neumann boundary values, $\vec u$ is a solution to the Neumann problem~\eqref{eqn:Neumann:energy} if and only if $T\vec u=\arr H$ as elements of $(\dot W^{p',\smooth'}_{m,av}(\Omega))^*$. Thus, the Neumann problem~\eqref{eqn:Neumann:energy} is well posed if and only if $T:\dot W^{p,\smooth}_{m,av}(\Omega)\mapsto (\dot W^{p',\smooth'}_{m,av}(\Omega))^*$ is invertible.

In particular, by the above remarks, $T$ is invertible if $p=2$ and $\smooth=1/2$. By Lemma~\ref{lem:invertible:stable}, we have that if $\ell$ is a line in the $(\smooth,1/2)$-plane then there is some $\varepsilon(\ell)>0$ such that if $\abs{1/p-1/2}+\abs{\smooth-1/2}<\varepsilon(\ell)$ and $(\smooth,1/p)\in \ell$ then the Neumann problem~\eqref{eqn:Neumann:energy} is well posed. Examining the arguments in \cite[Section~2]{KalM98}, we see that $\varepsilon(\ell)$ can be bounded from below, independently of~$\ell$. This completes the proof.

\section{Known results reformulated in the notation of the present paper}
\label{sec:A0}

Recall that Theorem~\ref{thm:perturb} allows us to establish well posedness for certain coefficients $\mat B$ given well posedness for nearby coefficients~$\mat A$. In Section~\ref{sec:A0:introduction}, we described new well posedness results arising from these theorems and from known results of \cite{BarM16A} and \cite{MitMW11,MitM13B}.

The results of \cite{BarM16A} were stated in terms of the spaces $L^{p,\smooth}_{m,av}(\R^\dmn_+)$, $\dot W^{p,\smooth}_{m,av}(\R^\dmn_+)$, $\dot W\!A^p_{m-1,\smooth}(\R^\dmnMinusOne)$, and $\dot N\!A^p_{m-1,\smooth-1}(\R^\dmnMinusOne)$ used in the present paper. However, the results of \cite{MitMW11,MitM13B} were stated in terms of other, related spaces. In this section we will convert the results of \cite{MitM13B} into equivalent results in terms of our spaces; we will also (Section~\ref{sec:BarM16A:2}) complete the argument of Remark~\ref{rmk:BarM16A:2}.

\subsection{The biharmonic equation}
\label{sec:biharmonic}

We begin by recalling the following result from \cite{MitM13B}.
%\cite{MitMW11,MitM13B}.

%\begin{thm}
%\textup{({\cite[Theorem~1.1]{MitMW11}} and {\cite[Theorem~6.10]{MitM13B}})}.\label{thm:MitM13B}
%Let $\Omega\subset\R^\dmn$, $\dmn\geq 2$, be a bounded Lipschitz domain with connected boundary. % If \dmn=3, of arbitrary topology
%Then there is some $\kappa>0$ such that if $p$ and $\smooth$ satisfy the inequalities~\eqref{eqn:biharmonic:stripe} (if $\dmn\geq 3$) or~\eqref{eqn:biharmonic:stripe:2} (if $\dmn=2$),
%then the Dirichlet problem 
%\begin{equation}
%\label{eqn:Dirichlet:MitM13B}\left\{\begin{gathered} \Delta^2 u = w \quad\text{in }\Omega,\quad
%\Trace u = f_0,\quad
%\Tr_1^\Omega u = \arr f,\\
%\doublebar{u}_{B^{p,p}_{\smooth+1/p+1}(\Omega)} 
% \leq
%C \bigl(\doublebar{f_0}_{L^p(\partial\Omega)}
%+ \doublebar{\arr f}_{B^{p,p}_\smooth(\partial\Omega)}
%+ \doublebar{\arr f}_{L^p(\partial\Omega)}
%+ \doublebar{w}_{B^{p,p}_{\smooth+1/p-3}(\Omega)}\bigr)
%\end{gathered}\right.
%\end{equation}
%has a unique solution for each $w\in {B^{p,p}_{s+1/p-3}(\Omega)}$, each $\arr f\in L^p(\partial\Omega)\cap \dot W\!A^p_{1,\smooth}(\partial\Omega)$, and each $f_0\in \dot W_1^p(\partial\Omega)$ which satisfies the compatibility condition $\nu_j\partial_k f_0-\nu_k\partial_j f_0 = \nu_j f_k-\nu_k f_j$ for $1\leq j\leq \dmn$ and $1\leq k\leq \dmn$.
%\end{thm}

\begin{thm}[{\cite[Theorems~ 6.10 and~6.16]{MitM13B}}]\label{thm:MitM13B}
Let $\Omega$ be a bounded Lipschitz domain  in $\R^\dmn$, $\dmn\geq 2$, with connected boundary.  Then there is some $\kappa>0$, depending on $\rho$, $\dmn$, and the Lipschitz character of~$\Omega$, such that if $\dmn\geq 4$ and the condition~\eqref{eqn:biharmonic:stripe} is valid, or if $\dmn=2$ or $\dmn=3$ and the condition~\eqref{eqn:biharmonic:stripe:2} is valid, 
then the Dirichlet problem 
\begin{equation}
\label{eqn:Dirichlet:MitM13B}\left\{\begin{gathered} \Delta^2 u = w \quad\text{in }\Omega,\quad
\Trace u = f_0,\quad
\Tr_1^\Omega u = \arr f,\\
\doublebar{u}_{F^{p,2}_{\smooth+1/p+1}(\Omega)} 
\leq
C \bigl(\doublebar{(f_0,\arr f)}_{ B^{p,p}_{1,\smooth}(\partial\Omega)}
+ \doublebar{w}_{F^{p,2}_{\smooth+1/p-3}(\Omega)}\bigr)
\end{gathered}\right.
\end{equation}
has a unique solution for each $w\in {B^{p,p}_{s+1/p-3}(\Omega)}$ and each $(f_0,\arr f)\in { B^{p,p}_{1,\smooth}(\partial\Omega)}$. % which satisfies the compatibility condition $\nu_j\partial_k f_0-\nu_k\partial_j f_0 = \nu_j f_k-\nu_k f_j$ for $1\leq j\leq \dmn$ and $1\leq k\leq \dmn$.

Furthermore, let $-1/\pdmnMinusOne<\rho<1$  and let $\mat A_\rho$ be as in formula~\eqref{eqn:biharmonic:coefficients}. Under the above assumptions on $\Omega$, $p$ and~$\smooth$, the Neumann problem 
\begin{equation}\label{eqn:Neumann:MitM13B}\left\{\begin{aligned} \Delta^2 u &= w \quad\text{in }\Omega,\\
\langle \nabla^2\varphi,\mat A_\rho \nabla^2 u\rangle_{\Omega} &= \langle \Trace^\Omega \varphi, g_0\rangle_{\partial\Omega} + \langle \Tr_1^\Omega \varphi, \arr g\rangle_{\partial\Omega}\text{ for all $\varphi\in C^\infty_0(\R^\dmn)$},\\
\doublebar{u}_{F^{p,2}_{\smooth+1/p+1}(\Omega)} & \leq
C \doublebar{(g_0,\arr g)}_{(B^{p',p'}_{1,(1-\smooth)}(\partial\Omega))^*}
+ C\doublebar{w}_{(F^{p',2}_{2-\smooth+1/p}(\Omega))^*}
\end{aligned}\right.
\end{equation}
has a unique solution for each $w\in {(F^{p',2}_{2-\smooth+1/p}(\Omega))^*}$ and  $(g_0,\arr g)\in {(B^{p',p'}_{1,(1-\smooth)}(\partial\Omega))^*}$ with $\langle (g_0,\arr g),(P,\nabla P)\rangle_{\partial\Omega}=\langle w,P\rangle_\Omega$ for every linear function~$P$. %Here $\mat A_\rho$ is the symmetric constant coefficient matrix such that 
%\begin{equation*}\langle \nabla^2 \psi(x),\mat A_\rho \nabla^2 \varphi(x)\rangle = 
%\rho\, \overline{\Delta \psi(x)}\,\Delta\varphi(x)
%+ (1-\rho) \sum_{j=1}^\dmn \sum_{k=1}^\dmn \overline{\partial_j \partial_k \psi(x)}\,\partial_j\partial_k \varphi(x).
%\end{equation*}

\end{thm}

In the case $\dmn=3$, the Dirichlet problem~\eqref{eqn:Dirichlet:MitM13B} was shown in \cite{MitMW11} to be well posed for all $p$ and $\smooth$ satisfying the condition~\eqref{eqn:biharmonic:stripe}. 
Well posedness was also established for solutions in the Besov spaces $\smash{B^{p,q}_{\smooth+1/p+1}(\Omega)}$ and more general Triebel-Lizorkin spaces ${F^{p,q}_{\smooth+1/p+1}(\Omega)}$; however, the stated particular case of ${F^{p,2}_{\smooth+1/p+1}(\Omega)}$ suffices for our purposes.
%One of the listed new results of Section~\ref{sec:A0:introduction} in the present paper is well posedness of the Dirichlet problem in the case $\dmn=2$, and for the Neumann problem for $\dmn=2$ or~$\dmn=3$, for such $p$ and~$\smooth$. %We will use the same argument in this subsection to pass from Theorem~\ref{thm:MitM13B} to Theorem~\ref{thm:biharmonic:introduction} in the Dirichlet case  with $\dmn=3$ and where $p$ and $\smooth$ satisfy the condition \eqref{eqn:biharmonic:stripe} but not~\eqref{eqn:biharmonic:stripe:2}; however, as this range of $p$ and $\smooth$ has been considered, albeit with different function spaces ${F^{p,2}_{\smooth+1/p+1}(\Omega)}$ rather than $\dot W^{p,\smooth}_{m,av}(\Omega)$, this cannot be viewed as a novel result.

For convenience, we will apply this theorem only in the case $w=0$.
The space ${B^{p,p}_{1,\smooth}(\partial\Omega)}$ defined in \cite[Section~2]{MitM13B} % p. 333
is given by
\begin{multline*}B^{p,p}_{1,\smooth}(\partial\Omega)=\{(f_0,\arr f):
f_0\in B^{p,p}_\smooth(\partial\Omega), \>\arr f\in B^{p,p}_\smooth(\partial\Omega), \\ \nu_j\partial_k f_0-\nu_k\partial_j f_0 = \nu_j f_k-\nu_k f_j\text{ for }1\leq j\leq \dmn\text{ and }1\leq k\leq \dmn.\}\end{multline*}
The (inhomogeneous) Besov space $B^{p,p}_\smooth(\partial\Omega)$ is defined via interpolation in \cite{MitM13B}; it is well known  (see, for example, \cite[formulas~(2.401), (2.421), (2.490)]{MitM13A}) that this definition means that
\begin{align*}\doublebar{f}_{B^{p,p}_\smooth(\partial\Omega)}
&\approx 
	\doublebar{f}_{L^p(\partial\Omega)}+\doublebar{f}_{\dot B^{p,p}_\smooth(\partial\Omega)}
%\\&=
%	\biggl(\int_{\partial\Omega} \abs{f(x)}^p\,dx\biggr)^{1/p}
%	+\biggl(\int_{\partial\Omega}\int_{\partial\Omega} \frac{\abs{f(x)-f(y)}}{\abs{x-y}^{\dmnMinusOne+p\smooth}}\,d\sigma(x)\,d\sigma(y)\biggr)^{1/p}.
\end{align*}
where ${\dot B^{p,p}_\smooth(\partial\Omega)}$ is as defined in \cite[Section~2.2]{Bar16pB} and used elsewhere in this paper.

We comment upon the condition $\nu_j\partial_k f_0-\nu_k\partial_j f_0 = \nu_j f_k-\nu_k f_j$. In this expression $\nu_j$ denotes the $j$th component of the unit outward normal $\nu$ to~$\Omega$, and so if $f_0$ is defined on~$\partial\Omega$ and lies in the boundary Sobolev space $\dot W^p_1(\partial\Omega)$ then the derivative $\nu_j\partial_k f_0-\nu_k\partial_j f_0$ is meaningful. The functions $f_k$ have the following significance: if $\arr f\in \dot W\!A^p_{1,\smooth}(\partial\Omega)$, then $\arr f$ is an array of functions indexed by multiindices of length~$1$, that is, by unit vectors; we let $f_k=f_{\vec e_k}$.

If $f_0$ is continuous and $\arr f=\Tr_1^\Omega \varphi$ for some $\varphi\in C^\infty_0(\R^\dmn)$, then the compatibility condition $\nu_j\partial_k f_0-\nu_k\partial_j f_0 = \nu_j f_k-\nu_k f_j$ almost everywhere on~$\partial\Omega$, for all $1\leq j\leq \dmn$, $1\leq k\leq \dmn$, is true if and only if $f_0=c_0+\Tr_0^\Omega\varphi=c_0+\varphi\big\vert_{\partial\Omega}$ for some constant~$c_0$.

In lieu of defining the Triebel-Lizorkin spaces ${F^{p,2}_{\smooth+1/p+1}(\Omega)}$ appearing in Theorem~\ref{thm:MitM13B}, we will simply state the following result relating these norms to more familiar norms.
\begin{lem}[{\cite[Proposition~S, p.~162]{AdoP98}}]\label{lem:AdoP98}
Let $\Omega$ be a bounded Lipschitz domain and let $u$ satisfy $\Delta^2 u=0$ in~$\Omega$. Let $1< p< \infty$.

If $k\geq 0$ is an integer and if
$k\leq \smooth+1/p+1\leq k+1$, then
\begin{equation*}\int_\Omega \abs{\nabla^{k+1} u(x)}^p \dist(x,\partial\Omega)^{pk-1-p\smooth}\,dx + \doublebar{\nabla^k u}_{L^p(\Omega)}^p + \doublebar{u}_{L^p(\Omega)}^p
\approx \doublebar{u}_{F^{p,2}_{\smooth+1/p+1}(\Omega)}
\end{equation*}
provided either the left-hand side or the right-hand side is finite.
\end{lem}

In this section we will derive the following well posedness result.
\begin{thm}\label{thm:MitM13B:translated}
Let $\Omega\subset\R^\dmn$ be a bounded Lipschitz domain with connected boundary. Then there is some $\kappa>0$ such that if $0<\smooth<1$, $1<p<\infty$, and if $\dmn\geq 4$ and the condition~\eqref{eqn:biharmonic:stripe} is valid, or if $\dmn=2$ or $\dmn=3$ and the condition~\eqref{eqn:biharmonic:stripe:2} is valid,
then the Dirichlet problem 
\begin{multline}
\label{eqn:Dirichlet:MitM13B:translated}
\Delta^2 u = \Div_2\arr H \text{ in }\Omega,\quad
\Tr_1^\Omega u = \arr f,\\
\doublebar{u}_{\dot W^{p,\smooth}_{2,av}(\Omega)} 
 \leq
C \doublebar{\arr f}_{\dot W\!A^p_{1,\smooth}(\partial\Omega)}
+ C\doublebar{\arr H}_{L^{p,\smooth}_{av}(\Omega)}
\end{multline}
has a unique solution for each ${\arr H}\in{L^{p,\smooth}_{av}(\Omega)}$ and each $\arr f\in {\dot W\!A^p_{1,\smooth}(\partial\Omega)}$.

Let $-1/\pdmnMinusOne<\rho<1$ and let $\mat A_\rho$ be as in formula~\eqref{eqn:biharmonic:coefficients}. Then there is some $\kappa>0$, depending on $\rho$, $\dmn$, and the Lipschitz character of~$\Omega$, such that under the above conditions on $\smooth$ and~$p$, the Neumann problem 
\begin{multline}
\label{eqn:Neumann:MitM13B:converted}
\Delta^2 u =\Div_2\arr H\text{ in }\Omega,\quad
\M_{\mat A_\rho,\arr H}^\Omega u=\arr g, \\ \doublebar{u}_{\dot W^{p,\smooth}_{2,av}(\Omega)} 
 \leq
C \doublebar{\arr g}_{\dot N\!A^p_{1,\smooth-1}(\partial\Omega)}
+ C\doublebar{\arr H}_{L^{p,\smooth}_{av}(\Omega)}
\end{multline}
has a unique solution for each ${\arr H}\in{L^{p,\smooth}_{av}(\Omega)}$ and each $\arr g\in {\dot N\!A^p_{1,\smooth-1}(\partial\Omega)}$.
\end{thm}

Recall that by \cite{MitMW11}, if $\dmn=3$ then the Dirichlet problem~\eqref{eqn:Dirichlet:MitM13B} is well posed if $p$ and $\smooth$ satisfy the weaker condition~\eqref{eqn:biharmonic:stripe}.
In a forthcoming paper, we intend to treat the issue of well posedness for $p<1$ in much more detail. Therein we will establish well posedness of the revised problem~\eqref{eqn:Dirichlet:MitM13B:translated} if the condition~\eqref{eqn:biharmonic:stripe}, and not merely~\eqref{eqn:biharmonic:stripe:2}, is valid; we will also establish similar results if $\dmn=2$ or for the Neumann problem.

The remainder of this section will be devoted to the proof of Theorem~\ref{thm:MitM13B:translated}.
We begin with the following lemma; this lemma will allow us to contend with the $L^p$ norms of $\nabla^k u$ and $u$ in Lemma~\ref{lem:AdoP98}.

\begin{lem}\label{lem:poincare:av} Let $\Omega\subset\R^\dmn$ be a bounded Lipschitz domain, let $0<\smooth<1$ and $1<p<\infty$, and let $w\in \dot W^1_{1,loc}(\overline\Omega)$. There is some constant $c$ such that
\begin{equation*}\int_\Omega \abs{w-c}^p\leq (\diam\Omega)^{1+p\smooth}\int_\Omega \abs{\nabla w(x)}^p\dist(x,\partial\Omega)^{p-1-p\smooth}\,dx\end{equation*}
provided the right-hand side is finite.
\end{lem}

\begin{proof}
Let $x_0\in\partial\Omega$ and let $r_\Omega$, $c_0$ and $V$ be as in the definition \cite[Definition~2.2]{Bar16pB} of Lipschitz domain. Then there are some coordinates and some Lipschitz function~$\psi$ such that $x_0=(x_0',t_0)$, such that $V=\{(x',t):x'\in\R^\dmnMinusOne,\allowbreak\>t>\psi(x')\}$ and such that $B(x_0,r_\Omega/c_0)\cap V=B(x_0,r_\Omega/c_0)\cap\Omega$.

Let $\Delta=\{x'\in\R^\dmnMinusOne:\abs{x'-x_0'}<r_\Omega/C_1\}$, and let $Q=\{(x',t): x'\in\Delta,\allowbreak\> \psi(x')<t<\psi(x')+r_\Omega/C_1\}$ for some large constant~$C_1$. If $C_1$ is large enough then $Q\subset B(x_0,r_\Omega/c_0)\cap\Omega$.

Let $\tau$ satisfy $r_\Omega/2C_1<\tau<r_\Omega/C_1$. Then
\begin{multline*}\biggl(\int_Q \abs{w-c}^p\biggr)^{1/p}
\leq 
\biggl(\int_\Delta 
\int_0^{r_\Omega/C_1} \abs[bigg]{\int_t^\tau \partial_r w(x',\psi(x')+r)\,dr}^p\,dt\,dx'\biggr)^{1/p}
\\+
\biggl(\frac{r_\Omega}{C_1}\int_\Delta  \abs{w(x',\psi(x')+\tau)-c}^p\,dx'\biggr)^{1/p}
.\end{multline*}
By H\"older's inequality,
\begin{equation*}\abs[bigg]{\int_t^\tau \partial_r w(x',\psi(x')+r)\,dr}^p\!
\leq \abs[bigg]{\int_t^\tau \abs{\nabla w(x',\psi(x')+r)}^p r^{p-1-p\smooth}\,dr}
\abs[bigg]{\int_t^\tau r^{p'\smooth-1}\,dr}^{p/p'}\!
\!.\end{equation*}
If $\smooth>0$, then the second integral converges and so
{\multlinegap=0pt\begin{multline*}\biggl(\int_Q \abs{w-c}^p\biggr)^{1/p}\!
\leq 
C_p(r_\Omega)^{\smooth+1/p}
\biggl(\int_\Delta 
{\int_0^{r_\Omega/C_1} \! \abs{\nabla w(x',\psi(x')+r)}^p r^{p-1-p\smooth}\,dr}
\,dx'\biggr)^{1/p}
\\+
\biggl(\frac{r_\Omega}{C_1}\int_\Delta  \abs{w(x',\psi(x')+\tau)-c}^p\,dx'\biggr)^{1/p}
\end{multline*}}%
whenever $r_\Omega/2C_1<\tau<r_\Omega/C_1$.
Averaging in~$\tau$, we see that 
{\multlinegap=0pt\begin{multline*}\biggl(\int_Q \abs{w-c}^p\biggr)^{1/p}
\leq 
C_p(r_\Omega)^{\smooth+1/p}
\biggl(\int_Q \abs{\nabla w(x)}^p \dist(x,\partial\Omega)^{p-1-p\smooth}\,dx
\biggr)^{1/p}
\\+
C\biggl(\int_{\{x\in Q:\dist(x,\partial\Omega)>r_\Omega/C\}}\abs{w-c}^p\biggr)^{1/p}
.\end{multline*}}%
Applying a standard patching argument, we see that
\begin{multline*}\int_\Omega
\abs{w-c}^p \leq C_p (\diam\Omega)^{1+p\smooth} \int_\Omega \abs{\nabla w(x)}^p\dist(x,\partial\Omega)^{p-1-p\smooth}\,dx
\\+C_p\int_{\{x\in \Omega:\dist(x,\partial\Omega)>r_\Omega/C\}}
\abs{w-c}^p.
\end{multline*}
Applying the Poincar\'e inequality to bound the final term completes the proof.
\end{proof}

We now establish uniqueness of solutions.

\begin{lem}\label{lem:biharmonic:unique} Let $p$ and $\smooth$ be as in Theorem~\ref{thm:MitM13B:translated}, and suppose in addition that $1/(1-\smooth)\leq p<\infty$.

Then solutions to the problems~\eqref{eqn:Dirichlet:MitM13B:translated} and~\eqref{eqn:Neumann:MitM13B:converted} are unique.
\end{lem}
\begin{proof}
Suppose that $\Delta^2 u=0$ in~$\Omega$, that $u\in \dot W^{p,\smooth}_{2,av}(\Omega)$, and that $\Tr_1^\Omega u=0$ or $\M_{\mat A_\rho,0}^\Omega u=0$. In the Dirichlet case we may normalize $u$ so that $\Trace^\Omega u=0$ as well.

By Theorem~\ref{thm:MitM13B} and Lemma~\ref{lem:AdoP98}, it suffices to show that 
\begin{equation*}\int_\Omega \abs{\nabla^2 u(x)}^p \dist(x,\partial\Omega)^{p-1-p\smooth}\,dx + \doublebar{\nabla u}_{L^p(\Omega)}^p + \doublebar{u}_{L^p(\Omega)}^p
<\infty.
\end{equation*}

Observe that $\partial^\alpha u$ is biharmonic in~$\Omega$ for any multiindex~$\alpha$. By the Caccioppoli inequality (\cite{Cam80,Bar16}), we have that if $x\in\Omega$, then
\begin{equation*}\doublebar{\nabla^{k} u}_{L^2(B(x,\dist(x,\partial\Omega)/4))}\leq C_{j,k}\dist(x,\partial\Omega)^{j-k} \doublebar{\nabla^j u}_{L^2(B(x,\dist(x,\partial\Omega)/2))}\end{equation*}
for any integers $k>j\geq 0$. Thus, by Morrey's inequality,
\begin{equation*}\abs{\nabla^2 u(x)}
\leq C\biggl(\fint_{B(x,\dist(x,\partial\Omega)/2)} \abs{\nabla^2 u}^2\biggr)^{1/2}\end{equation*}
and so 
\begin{equation*}%\label{eqn:biharmonic:1}
\int_\Omega \abs{\nabla^2 u(x)}^p \dist(x,\partial\Omega)^{p-1-p\smooth}\,dx 
\leq 
C\doublebar{u}_{\dot W^{p,\smooth}_{2,av}(\Omega)}^p
.\end{equation*}

By Lemma~\ref{lem:poincare:av}, we have that $\nabla u\in L^p(\Omega)$. By the Poincar\'e inequality, we have that $u\in L^p(\Omega)$. This completes the proof.
\end{proof}

The following lemma shows that, if $\arr f\in \dot  W\!A^p_{1,\smooth}(\Omega)$, then there is some $f_0$ such that $f_0$, $\arr f$ satisfy the conditions of Theorem~\ref{thm:MitM13B}; we will use this lemma and Theorem~\ref{thm:MitM13B} to establish existence of solutions.

\begin{lem}\label{lem:poincare:whitney} Let $\Omega\subset\R^\dmn$ be a bounded Lipschitz domain with connected boundary. Suppose that $0<\smooth<1$, $1\leq p< \infty$ and that $\arr f\in \dot W\!A^p_{1,\smooth}(\Omega)$.

Let $P$ be the linear function that satisfies
\begin{equation*}\fint_{\partial\Omega} \arr f -\nabla P \,d\sigma=0.\end{equation*}
Then 
\begin{equation*}\doublebar{\arr f-\nabla P}_{L^p(\Omega)} \leq C(\diam\Omega)^\smooth \doublebar{\arr f}_{\dot W\!A^p_{1,\smooth}(\Omega)}.\end{equation*}
Furthermore, there is some $f_0\in L^p(\partial\Omega)\cap \dot W^p_1(\partial\Omega)$ with 
$\nu_j\partial_k f_0-\nu_k\partial_j f_0 = \nu_j (f_k-\partial_k P)-\nu_k (f_j-\partial_j P)$ and with
\begin{equation*}\doublebar{f_0}_{L^p(\partial\Omega)}\leq C(\diam\Omega)^{1+\smooth} \doublebar{\arr f}_{\dot W\!A^p_{1,\smooth}(\Omega)}.\end{equation*}
\end{lem}

\begin{proof}
If $\arr f=\Tr_1^\Omega\varphi$ for some smooth compactly supported function~$\varphi$, then $f_0$ exists and satisfies $f_0=(\varphi-P)\big\vert_{\partial\Omega}$. The bound on  $\doublebar{f_0}_{L^p(\partial\Omega)}$ follows from the claimed bound on $\doublebar{\arr f-\nabla P}_{L^p(\Omega)}$ by the Poincar\'e inequality.
Existence of $f_0$ for general~$\arr f$ follows by density of such arrays in $\dot W\!A^p_{1,\smooth}(\partial\Omega)$; see the definition of $\dot W\!A^p_{1,\smooth}(\partial\Omega)$ in \cite[Section~2.2]{Bar16pB}.

We are left with the given estimates on~$\arr f-\nabla P$.
Suppose without loss of generality that $\int_{\partial\Omega} \arr f\,d\sigma=0$. Then
\begin{equation*}\int_{\partial\Omega} \abs{f(x)}^p\,d\sigma(x) 
= \int_{\partial\Omega} \abs[bigg]{\fint_{\partial\Omega}f(x)-f(y)\,d\sigma(y)}^p\,d\sigma(x) .\end{equation*}
By H\"older's inequality,
\begin{align*}\int_{\partial\Omega} \abs{\arr f(x)}^p\,d\sigma(x) 
&\leq
\frac{1}{\sigma(\partial\Omega)}
\int_{\partial\Omega} \int_{\partial\Omega}\abs{\arr f(x)-\arr f(y)}^p\,d\sigma(y)\,d\sigma(x) 
\\&\leq
\frac{(\diam\Omega)^{\dmnMinusOne+p\smooth}}{\sigma(\partial\Omega)}
\int_{\partial\Omega} \int_{\partial\Omega}\frac{\abs{\arr f(x)-\arr f(y)}^p}{\abs{x-y}^{\dmnMinusOne+p\smooth}}\,d\sigma(y)\,d\sigma(x) 
.\end{align*}
Applying the definition of $\dot B^{p,p}_\smooth(\partial\Omega)$ (see \cite[Section~2.2]{Bar16pB}) completes the proof.
\end{proof}

The following lemma establishes existence of solutions in the $\arr H=0$, $p\geq 1/( 1-\smooth)$ case. 

\begin{lem}\label{lem:MitM13B:translated}
Let $\Omega$ be as in Theorem~\ref{thm:MitM13B:translated}.
Suppose that $0<\smooth<1$, that $1/(1-\smooth)\leq p<\infty$, and that $-1/\pdmnMinusOne<\rho<1$. 
Suppose that $\dmn\geq 4$ and the condition~\eqref{eqn:biharmonic:stripe} is valid, or $\dmn=2$ or $\dmn=3$ and the condition~\eqref{eqn:biharmonic:stripe:2} is valid.

Then for each $\arr f\in {\dot W\!A^p_{1,\smooth}(\partial\Omega)}$, 
there is a solution to the problem
\begin{equation}\label{eqn:biharmonic:Dirichlet:homogeneous}
\Delta^2 u = 0 \text{ in }\Omega,\quad
\Tr_1^\Omega u = \arr f,\quad
\doublebar{u}_{\dot W^{p,\smooth}_{2,av}(\Omega)} 
 \leq
C \doublebar{\arr f}_{\dot W\!A^p_{1,\smooth}(\partial\Omega)}
.\end{equation}

Also, for each $\arr g\in {\dot N\!A^p_{1,\smooth-1}(\partial\Omega)}$, 
there is a solution to the problem
\begin{equation}\label{eqn:biharmonic:Neumann:homogeneous}
\Delta^2 u =0\text{ in }\Omega,\quad
\M_{\mat A_\rho,0}^\Omega u=\arr g, \quad \doublebar{u}_{\dot W^{p,\smooth}_{m,av}(\Omega)} 
 \leq
C \doublebar{\arr g}_{\dot N\!A^p_{1,\smooth-1}(\partial\Omega)}
.\end{equation}
\end{lem}

\begin{proof}
Without loss of generality we may assume that $\diam\Omega=1$.
Let $\arr f\in \dot W\!A^p_{1,\smooth}(\partial\Omega)$. By
Theorem~\ref{thm:MitM13B}, Lemma~\ref{lem:AdoP98} and Lemma~\ref{lem:poincare:whitney}, there is some $u$ that satisfies
\begin{multline*}\Delta^2 u=0 \text{ in }\Omega, \quad \Tr_{m-1}^\Omega \vec u=\arr f,
\\
\int_\Omega \abs{\nabla^2 u(x)}^p\dist(x,\partial\Omega)^{p-1-p\smooth}\,dx \leq C \doublebar{\arr f}_{\dot W\!A^p_{1,\smooth}(\partial\Omega)}
.\end{multline*}
By H\"older's inequality (if $p\geq 2$) or by \cite[Theorem~24]{Bar16} (if $p<2$), we have that
\begin{equation*}\doublebar{u}_{\dot W^{p,\smooth}_{2,av}(\Omega)}^p
\leq 
C\int_\Omega \abs{\nabla^2 u(x)}^p \dist(x,\partial\Omega)^{p-1-p\smooth}\,dx 
\end{equation*}
and so the proof is complete.

We now turn to the Neumann problem~\eqref{eqn:biharmonic:Neumann:homogeneous}.
Let $\arr g\in \dot N\!A^p_{1,\smooth-1}(\partial\Omega)$, and let $g_0=0$. If $(\varphi_0,\arr \varphi)\in W\!A^{p'}_{1,(1-\smooth)}(\partial\Omega)$, then
\begin{equation*}\abs{\langle (\varphi_0,\arr\varphi),(0,\arr g)\rangle_{\partial\Omega}}
\leq C \doublebar{\arr \varphi}_{\dot W\!A^{p'}_{1,1-\smooth}(\partial\Omega)}
\doublebar{\arr g}_{\dot N\!A^p_{1,\smooth-1}(\partial\Omega)}.\end{equation*}
But $\doublebar{\arr \varphi}_{\dot W\!A^{p'}_{1,1-\smooth}(\partial\Omega)}\leq \doublebar{(\varphi_0,\arr \varphi)}_{B^{p',p'}_{1,1-\smooth}(\partial\Omega)}$, and so $\arr g$ is a bounded linear operator on ${B^{p',p'}_{1,1-\smooth}(\partial\Omega)}$. We may complete the proof using Theorem~\ref{thm:MitM13B} and Lemma~\ref{lem:AdoP98} as before.
\end{proof}

The following lemma allows us to pass to the case $\arr H\neq 0$. This lemma is a converse to Lemma~\ref{lem:zero:boundary}.

\begin{lem}\label{lem:zero:interior} Let $L$ be an operator of the form~\eqref{eqn:divergence}. Let $\Omega$ be a Lipschitz domain with connected boundary. Let $0<\smooth<1$ and $\pmin<p\leq \infty$ be such that $\vec\Pi^L$ is a bounded operator $L^{p,\smooth}_{av}(\Omega)\mapsto \dot W^{p,\smooth}_{m,av}(\Omega)$.

Suppose that for every $\arr\eta\in {\dot W\!A^p_{m-1,\smooth}(\partial\Omega)}$ there exists a solution $\vec u$ to the Dirichlet problem
\begin{equation}\label{eqn:Dirichlet:boundary}
L \vec u = 0 \text{ in }\Omega,
\quad \Tr_{m-1}^\Omega\vec u = \arr\eta,
\quad
\doublebar{\vec u}_{\dot W^{p,\smooth}_{m,av}(\Omega)} \leq C \doublebar{\arr\eta}_{\dot W\!A^p_{m-1,\smooth}(\partial\Omega)}
.\end{equation}

Then for each $\arr H\in  L^{p,\smooth}_{av}(\Omega)$ and for each $\arr f\in {\dot W\!A^p_{m-1,\smooth}(\partial\Omega)}$ there is a solution to the Dirichlet problem
\begin{multline*}
%\label{eqn:Dirichlet:full}
%\left\{\begin{aligned}
L \vec u = \Div_m \arr H \text{ in }\Omega,
\quad \Tr_{m-1}^\Omega\vec u = \arr f,
\\
\doublebar{\vec u}_{\dot W^{p,\smooth}_{m,av}(\Omega)} \leq C \doublebar{\arr H}_{L^{p,\smooth}_{av}(\Omega)}
+C\doublebar{\arr f}_{\dot W\!A^p_{m-1,\smooth}(\partial\Omega)}
.%\end{aligned}\right.
\end{multline*}

Suppose that $\Omega$ is a Lipschitz domain with connected boundary, $0<\smooth<1$, and $1<p\leq\infty$.  Suppose that that for every $\arr \gamma\in {\dot N\!A^p_{m-1,\smooth-1}(\partial\Omega)}$ there exists a solution $\vec u$ to the Neumann problem
\begin{equation}
\label{eqn:Neumann:boundary}
L \vec u = 0 \text{ in }\Omega,
\quad \M_{\mat A,0}^\Omega\vec u = \arr \gamma,
\quad
\doublebar{\vec u}_{\dot W^{p,\smooth}_{m,av}(\Omega)} \leq C \doublebar{\arr \gamma}_{\dot N\!A^p_{m-1,\smooth-1}(\partial\Omega)}
\end{equation}

Then for each $\arr H\in  L^{p,\smooth}_{av}(\Omega)$ and for each $\arr g\in {\dot N\!A^p_{m-1,\smooth-1}(\partial\Omega)}$ there is a solution to the Neumann problem
\begin{multline*}
%\label{eqn:Neumann:full}
L \vec u = \Div_m \arr H \text{ in }\Omega,
\quad \M_{\mat A,\arr H}^\Omega\vec u = \arr g,
\\
\doublebar{\vec u}_{\dot W^{p,\smooth}_{m,av}(\Omega)} \leq C \doublebar{\arr H}_{L^{p,\smooth}_{av}(\Omega)}
+C\doublebar{\arr g}_{\dot N\!A^p_{m-1,\smooth-1}(\partial\Omega)}
.
\end{multline*}
\end{lem}
%Of course passing from the problems \eqref{eqn:Dirichlet:interior} or~\eqref{eqn:Neumann:interior} to the full boundary value problems \eqref{eqn:Dirichlet:full} or~\eqref{eqn:Neumann:full} is simply another application of Lemma~\ref{lem:zero:boundary}.

\begin{proof}
By assumption, $\vec\Pi^L\arr H\in \dot W^{p,\smooth}_{m,av}(\Omega)$. Let 
\begin{equation*}\arr\eta = \Tr_{m-1}^\Omega \vec\Pi^L\arr H,\qquad\arr\gamma = \M_m^\Omega (\mat A\nabla^m\vec\Pi^L\arr H -\arr H)\end{equation*}
where $\M_m^\Omega$ is as in \cite[formula~(1.8)]{Bar16pB}.
By \cite[Theorem~5.1]{Bar16pB}, $\arr\eta\in {\dot W\!A^p_{m-1,\smooth}(\partial\Omega)}$. By \cite[Theorem~7.1]{Bar16pB}, $\arr \gamma \in {\dot N\!A^p_{m-1,\smooth-1}(\partial\Omega)}$.

Let $\vec v$ be the solution to the problem \eqref{eqn:Dirichlet:boundary} or~\eqref{eqn:Neumann:boundary} with boundary data $\arr\eta-\arr f$ or $\arr\gamma-\arr g$. Let $\vec u=\vec\Pi^L\arr H-\vec v$. Then 
\begin{equation*}L\vec u=L\vec\Pi^L\arr H -L\vec v=\Div_m\arr H\end{equation*}
and either
\begin{equation*}\Tr_{m-1}^\Omega \vec u = \Tr_{m-1}^\Omega \vec\Pi^L\arr H-\Tr_{m-1}^\Omega \vec v=
\arr\eta-\arr\eta+\arr f =\arr f\end{equation*}
or, for every smooth test function~$\vec\varphi$,
\begin{align*}\langle \nabla^m\varphi,\mat A\nabla^m \vec u-\arr H\rangle_\Omega
&=\langle \nabla^m\varphi,\mat A\nabla^m \vec \Pi^L\arr H-\arr H\rangle_\Omega
-\langle \nabla^m\varphi,\mat A\nabla^m \vec v\rangle_\Omega
\\&=
\langle \Tr_{m-1}^\Omega\vec\varphi, \M_{m}^\Omega(\mat A\nabla^m\vec\Pi^L\arr H -\arr H)\rangle_{\partial\Omega}
\\&\qquad-
\langle \Tr_{m-1}^\Omega\vec\varphi, \M_{\mat A,0}^\Omega\vec v\rangle_{\partial\Omega}
\\&=
\langle \Tr_{m-1}^\Omega\vec\varphi, \arr \gamma\rangle_{\partial\Omega}
-
\langle \Tr_{m-1}^\Omega\vec\varphi, \arr \gamma-\arr g\rangle_{\partial\Omega}
=
\langle \Tr_{m-1}^\Omega\vec\varphi, \arr g\rangle_{\partial\Omega}
\end{align*}
as desired.
\end{proof}

By Lemmas~\ref{lem:biharmonic:unique}, \ref{lem:MitM13B:translated} and~\ref{lem:zero:interior}, we have that Theorem~\ref{thm:MitM13B:translated} is valid  if $0<\smooth<1$, $1/(1-\smooth)\leq p<\infty$, and if $\dmn\geq 4$ and the condition~\eqref{eqn:biharmonic:stripe} is valid,  or if $\dmn=2$ or $\dmn=3$ and the condition~\eqref{eqn:biharmonic:stripe:2} is valid.

We may pass to the case $1<p<1/(1-\smooth)$ using Theorems~\ref{thm:exist:unique} and~\ref{thm:unique:exist}. This completes the proof of Theorem~\ref{thm:MitM13B:translated}.

\subsection{Real symmetric $t$-independent coefficients if $m=N=1$ and $\dmn=2$} \label{sec:BarM16A:2}

In this section we complete the argument of Remark~\ref{rmk:BarM16A:2} by proving the following lemma.

\begin{lem}\label{lem:BarM16A:2}
Let $\Omega=\{(x',t):x'\in\R,\>t>\psi(x)\}$ be a Lipschitz graph domain in~$\R^2$. Let $L$ be an elliptic operator of the form~\eqref{eqn:divergence} with $m=N=1$, associated to real symmetric coefficients $\mat A$ that satisfy the ellipticity conditions~\eqref{eqn:elliptic:bounded} and~\eqref{eqn:elliptic:everywhere} and are $t$-independent in the sense of formula~\eqref{eqn:t-independent}.

Then there is some $\kappa>0$ such that the Dirichlet problem~\eqref{eqn:Dirichlet:p:smooth} and the Neumann problem~\eqref{eqn:Neumann:p:smooth} are well posed whenever 
\begin{equation*}0<\smooth<1,\quad 0<p\leq \infty, \quad -\frac{1}{2}-\kappa<\frac{1}{p}-\smooth<\frac{1}{2}+\kappa.
\end{equation*}
\end{lem}

We begin with the following known result.
The case $1/q_++1/q_-=1$, $q_+=\widetilde q_+$, is the Lipschitz graph domain case of \cite[Theorem~9.1]{Bar13}; a careful inspection of the proof therein reveals the general case. (For the sake of simplicity we will consider only the Dirichlet case, and derive results for the Neumann problem as in Remark~\ref{rmk:BarM16A:2}.)
\begin{lem}\label{lem:Bar13}
Let $L$ and $\Omega$ be as in Lemma~\ref{lem:BarM16A:2}. Let $a:\partial\Omega\mapsto \C$ satisfy 
\begin{equation}\label{eqn:atom}
\doublebar{\partial_\tau a}_{L^\infty(\partial\Omega)}\leq 1/r,
\quad
\supp a\subset B(x_0,r)\cap\partial\Omega \text{ for some $x_0\in\partial\Omega$ and $r>0$.}\end{equation}
Here $\partial_\tau$ is the derivative tangential to~$\partial\Omega$.

Suppose that for some $1<q_-<\infty$ and $1<q_+<\infty$, the boundary value problems
\begin{gather}\label{eqn:Dirichlet:Bar13}
Lu=0 \text{ in }\Omega,\quad \Trace^\Omega u=f,\quad \doublebar{Nu}_{L^{q_-}(\partial\Omega)}\leq c_- \doublebar{f}_{L^{q_-}(\partial\Omega)},
\\\label{eqn:regularity:Bar13}
Lu=0 \text{ in }\Omega,\quad \Trace^\Omega u=f,\quad \doublebar{N(\nabla u)}_{L^{q_+}(\partial\Omega)}\leq c_+ \doublebar{\nabla_\tau f}_{L^{q_+}(\partial\Omega)},
\end{gather}
are compatibly well posed. Suppose that there is some $1<\widetilde q_+<\infty$ such that the boundary value problem
\begin{equation}\label{eqn:regularity:Bar13:local}
Lu=0 \text{ in }Q,\quad \Trace^Q u=f,\quad \doublebar{N(\nabla u)}_{L^{\widetilde q_+}(\partial Q)}\leq \widetilde c_+ \doublebar{\partial_\tau f}_{L^{\widetilde q_+}(\partial Q)},
\end{equation}
is well posed whenever $Q=Q(x_0',\rho)=\{(x',t):\abs{x'-x_0'}<\rho,\psi(x')<t<\psi(x')+\rho\}$ for some $x_0'\in\R$ and some $\rho>0$. 

Then the solution $u$ to the problems \textup{(\ref{eqn:Dirichlet:Bar13}--\ref{eqn:regularity:Bar13})}, with $f=a$, satisfies
\begin{equation}\label{eqn:NT:decay}\int_{\partial\Omega} N(\nabla u)(x) \,(1+\abs{x-x_0}/r)^\kappa \,d\sigma(x)\leq C\end{equation}
for any $0<\kappa<1/q_-$, where $C$ depends on $\kappa$, $q_-$, $q_+$, $\widetilde q_+$, the Lipschitz character of~$\Omega$, and the numbers $c$ above.
\end{lem}
Here $N$ is the nontangential maximal function common in the literature. 

By \cite{Rul07}, we have well posedness of the local boundary value problems~\eqref{eqn:regularity:Bar13:local} for some (possibly small) $\widetilde q_+>1$.  By \cite{KenP93}, we have well posedness of the problem~\eqref{eqn:regularity:Bar13} for $q_+=2$, while by \cite{JerK81A} we have that there is some $\varepsilon>0$ such that  the problem~\eqref{eqn:Dirichlet:Bar13} is well posed for all $2-\varepsilon<q_-<\infty$. These problems are compatibly well posed; see the above papers or \cite{AusAH08,AusM14,AusS14p}. 

Fix some $a$ as in Lemma~\ref{lem:Bar13} and let $u$ be as in Lemma~\ref{lem:Bar13}. Then the estimate~\eqref{eqn:NT:decay} is valid for all $0<\kappa<1/(2-\varepsilon)$.

An elementary argument involving H\"older's inequality shows that if $1\geq p_0>1/(1+\kappa)$, then
\begin{equation*}\int_{\partial\Omega} N(\nabla u)(x)^{p_0}\,d\sigma(x)\leq Cr^{1-p_0}.\end{equation*}
By \cite[Theorem~7.11]{BarM16A} and a change of variables, if $\Omega$ is a Lipschitz graph domain, then 
\begin{equation*}\doublebar{u}_{\dot W^{p,\smooth}_{1,av}(\Omega)} \leq C r^{1/p-\smooth}\end{equation*}
whenever $0<\smooth<1$, $p_0<p\leq \infty$ and $\smooth-1/p=1-1/p_0$.

Let $p$ and $\smooth$ satisfy the given conditions and be such that such a $p_0$ and $\kappa$ exist. We impose the additional condition $p\leq 1$; we thus require
\begin{equation*}0<p\leq 1,\quad 0<\smooth<1,\quad 1/p-\smooth<1/(2-\varepsilon).\end{equation*} 
Let $f\in {\dot B^{p,p}_\smooth(\partial\Omega)}$. Then by \cite[Definition~2.6]{Bar16pB}, we have that $f=\sum_{j=1}^\infty\lambda_j a_j$, where $\sum_j \abs{\lambda_j}^p\approx \doublebar{f}_{\dot B^{p,p}_\smooth(\partial\Omega)}^p$ and where $a_j$ satisfies the conditions
\begin{equation*}\supp a_j\subset B(x_j,r_j)\cap\partial\Omega,\quad \doublebar{\partial_\tau a_j}_{L^\infty(\partial\Omega)}\leq r_j^{\smooth-1-1/p}\end{equation*}
for some $x_j\in \partial\Omega$ and some $r_j>0$. Let $u_j$ be as in Lemma~\ref{lem:Bar13} with $a=r_j^{1/p-s}a_j$ and let $u=\sum_j r_j^{s-1/p}u_j$; then
\begin{equation*}\doublebar{u}_{\dot W^{p,\smooth}_{1,av}(\Omega)}^p\leq C\sum_j\abs{\lambda_j}^p \approx C\doublebar{f}_{\dot B^{p,p}_\smooth(\partial\Omega)}^p\end{equation*}
and so we have existence of solutions to the Dirichlet problem~\eqref{eqn:Dirichlet:boundary} provided $0<\smooth<1$ and $1\leq 1/p<\smooth+1/(2-\varepsilon)$. By \cite[Theorem~3.1]{BarM16A} and \cite[Th\'eor\`eme~II.2]{AusT95}, the Newton potential is bounded $L^{p,\smooth}_{av}(\Omega)\mapsto \dot W^{p,\smooth}_{1,av}(\Omega)$ whenever $0<\smooth<1$ and $0\leq 1/p<1+\smooth$; thus, by Lemma~\ref{lem:zero:interior}, solutions to the problem \eqref{eqn:Dirichlet:p:smooth} exist whenever $1\leq 1/p<\smooth+1/(2-\varepsilon)$.

By \cite{BarM16A} (see Theorem~\ref{thm:BarM16A} above), there is some $q_0$ with $1<q_0<2$ such that the Dirichlet problem~\eqref{eqn:Dirichlet:p:smooth} is well posed for all $0<\smooth<1$ and all $q_0< q< q_0'$, where $1/q_0+1/q_0'=1$. We impose the additional assumption that  $1/p-\smooth<1/q_0$. There is then some $\sigma$, $q$ with $0<\sigma<1$, $q_0<q<q_0'$, and $1/p-\smooth=1/q-\sigma$. Solutions to the Dirichlet problem~\eqref{eqn:Dirichlet:p:smooth} with $p=q$ and $\smooth=\sigma$ are unique; thus, by Corollary~\ref{cor:unique:extrapolate}, solutions to the the Dirichlet problem~\eqref{eqn:Dirichlet:p:smooth} with $p$ and $s$ as above are unique. 

Furthermore, by Corollary~\ref{cor:compatible}, the Dirichlet problem with $p$, $\smooth$ as above and the Dirichlet problem with $p=q$, $\smooth=\sigma$ are compatibly well posed in the sense of Lemma~\ref{lem:interpolation}; thus, by Lemma~\ref{lem:interpolation}, we have that the Dirichlet problem is well posed whenever $0<\smooth<1$, $0<p<q_0'$ and $1<1/p-\smooth<\min(1/q,1/(2-\varepsilon))$. 

By Theorem~\ref{thm:BarM16A} and the above remarks, we have well posedness whenever $0<\smooth<1$, $0<p<q_0'$ and $1/p-\smooth<\min(1/q,1/(2-\varepsilon))$. By Theorems~\ref{thm:exist:unique} and~\ref{thm:unique:exist}, we have well posedness whenever $q_0'<p\leq \infty$ and $1/p-\smooth>\max(1/(2-\varepsilon),\allowbreak 1/q_0)$. This completes the proof.

\bibliographystyle{amsalpha}
\bibliography{bibli}

\end{document}